\documentclass[10pt,twoside, a4paper, english, reqno]{amsart}
\usepackage[dvips]{epsfig}
\usepackage{amscd}
\usepackage{amssymb}
\usepackage{amsthm}
\usepackage{amsmath}
\usepackage{latexsym}

\usepackage{upref}
\usepackage{hyperref}
\usepackage{color}
\usepackage[usenames,dvipsnames]{xcolor}
\theoremstyle{plain}
\newtheorem{thm}{Theorem}[section]
\theoremstyle{plain}
\newtheorem{lem}[thm]{Lemma}
\newtheorem{prop}[thm]{Proposition}
\newtheorem{cor}[thm]{Corollary}

\theoremstyle{definition}
\newtheorem{defi}{Definition}[section]
\newtheorem*{rem}{Remark}

\newenvironment{Assumptions}
{
\setcounter{enumi}{0}

\begin{enumerate}}
{\end{enumerate} }

\newenvironment{Assumptions2}
{
\setcounter{enumi}{0}

\begin{enumerate}}
{\end{enumerate} }

\newcommand{\eps}{\ensuremath{\varepsilon}}
\newcommand{\R}{\ensuremath{\mathbb{R}}}

\newcommand{\con} {\ast}

\newcommand{\dz}{\ensuremath{\, dz}}

\newcommand{\rd}{\ensuremath{\mathbb{R}^d}}
\newcommand{\supp}{\ensuremath{\mathrm{supp}\,}}
\newcommand{\goto}{\ensuremath{\rightarrow}}
\newcommand{\grad}{\ensuremath{\nabla}}

\numberwithin{equation}{section} \allowdisplaybreaks

\title[L\'{e}vy driven conservation laws]{Conservation laws 
driven by L\'{e}vy white noise}
\date{}

\author[Imran H. Biswas]{Imran H. Biswas}
\address[Imran H. Biswas]{\newline
 Centre for Applicable Mathematics,
 Tata Instiute of Fundamental Research,
  P.O.\ Box 6503, GKVK Post Office,
  Bangalore 560065, India}
\email[]{imran@math.tifrbng.res.in}

\author[K. H. Karlsen]{Kenneth H. Karlsen}
\address[Kenneth Hvistendahl Karlsen]{\newline
Centre of Mathematics for Applications,
University of Oslo,
P.O.\ Box 1053, Blindern,
NO--0316 Oslo, Norway}
\email[]{kennethk@math.uio.no}
\urladdr{folk.uio.no/kennethk}

\author[Ananta K. Majee]{Ananta K. Majee}
\address[Ananta K. Majee]{\newline
 Centre for Applicable Mathematics,
 Tata Instiute of Fundamental Research,
  P.O.\ Box 6503, GKVK Post Office,
  Bangalore 560065, India}
\email[]{ananta@math.tifrbng.res.in}


\thanks{}

\begin{document}
\begin{abstract}
We consider multidimensional conservation laws perturbed 
by multiplicative L\'{e}vy noise. We establish  
existence and uniqueness results for entropy solutions. 
The entropy inequalities are formally obtained by the It\'{o}-L\'{e}vy chain rule. 
The multidimensionality requires a generalized interpretation of 
the entropy inequalities to accommodate Young measure-valued solutions. 
We first establish the existence of entropy solutions in 
the generalized sense via the vanishing viscosity method, and then 
establish the $L^1$-contraction principle. Finally, the $L^1$ contraction 
principle is used to argue that the generalized entropy solution 
is indeed the classical entropy solution. 
\end{abstract}

\maketitle
\tableofcontents
\section{Introduction}\label{intro}
We are interested in stochastic perturbations 
of nonlinear conservation laws. A conservation law with source 
term (balance law) is an equation of the type 
\begin{align}
    \label{eq:balance laws} \frac{\partial u(t,x)}{\partial t} 
    + \mbox{div}_x F(u(t,x)) 
    = q(t,x, u(t,x)), 
    \qquad t>0, ~ x\in \R^d,
\end{align} 
where $F$ is known as the flux function. In a deterministic context 
the source term $q(t,x,u)$ is given by a nicely behaved function 
and Kru{\v{z}}kov's entropy solution framework provides a comprehensive 
understanding of the related Cauchy problem. There are multiple ways of 
interpreting $q$ and we are particularly interested in the scenario where the 
source $q(t,x, u)$ represents a multiplicative white noise.  
This would make \eqref{eq:balance laws} a stochastic 
balance law and this equation has attracted significant attention in 
recent years. However, all studies have been limited 
to the case where the source $q(t,x,u)$ represents a Brownian multiplicative 
white noise i.e $ q(t,x, u) = \sigma(t,x,u) \frac{dB_t}{dt}$, 
where $(B_t)_{t\ge 0}$ is a Brownian motion. 

In this paper, we intend to study the Cauchy problem related 
to  \eqref{eq:balance laws}  where the source term $q(t,x,u)$ 
represents a multiplicative L\'{e}vy white noise. 
A more precise description of our problem is as follows. 
Let $\big(\Omega, P, \mathcal{F}, \{\mathcal{F}_t\}_{t\ge 0} \big)$ be 
a filtered probability space satisfying the usual hypothesis. 
We are looking for a  $L^2(\R^d)$-valued predictable 
process $u(t)$ satisfying
\begin{align} du(t,x) + \mbox{div}_x F(u(t,x)) \,dt
=\int_{|z|> 0} \eta(x, u(t,x);z)\, \tilde{N}(dz,dt), 
\quad t>0, ~ x\in \R^d,
 \label{eq:levy_stochconservation_laws}
\end{align}
with the initial condition    
\begin{align}\label{initial_cond} 
u(0,x) = u_0(x), \quad \quad x\in \R^d.
\end{align} 
In \eqref{eq:levy_stochconservation_laws}, $F:\R\rightarrow \R^d$ is a 
given nonlinear flux function, and 
$\tilde{N}(dz,dt)= N(dz,dt)-\, m(dz)\,dt$, where  $N$ is 
a Poisson random measure on $ \R\times (0,\infty) $ 
with intensity measure $m(dz)$ such that $\int
(1\wedge |z|^2)\,m(dz) <  + \infty$. Moreover, $\eta(x,u; z)$ is a 
real valued function defined on the domain $\R^d\times\R\times \R$.  
We point out that adding a Brownian component to the 
white noise term on the right hand side of \eqref{eq:levy_stochconservation_laws} 
would make it more general, and the results of this 
paper are still valid under appropriate conditions. 

The equation \eqref{eq:levy_stochconservation_laws} becomes a 
multidimensional deterministic conservation law 
if $\eta =0$. It is well-documented
that solutions of deterministic conservation laws develop 
discontinuities (shocks) in finite time. 
Therefore the solutions must be interpreted in the 
weak sense and a so-called entropy condition is required to 
identify the physically relevant (unique) solution \cite{dafermos,godu}. 

The study of stochastic balance laws has so far been limited 
to equations driven by Brownian white noise. 
For some first results in that direction, see Holden and Risebro \cite{risebroholden1997}. 
E, Khanin, Mazel, and Sinai \cite{E:2000lq} described 
the statistical properties of the Burgers equation with Brownian noise. 
Kim \cite{Kim2003} extended the Kru{\v{z}}kov well-posedness theory 
to one dimensional balance laws that are driven by 
additive Brownian noise. This approach does not apply to 
the multiplicative noise case. Indeed, a straightforward adaptation 
of the deterministic ``doubling technique" leads to anticipating 
stochastic integrands, and so the standard route 
leading to the $L^1$-contraction principle cannot be followed. 
In a recent work, Feng and Nualart \cite{nualart:2008} came up with 
a way to address this issue, giving raise to what they referred to 
as {\em strong} entropy solutions, which in turn are intimately connected to
vanishing viscosity solutions.  In \cite{nualart:2008}, the authors 
established the uniqueness of strong entropy solutions in a 
multidimensional $L^p$-framework. The existence, however, was restricted to 
one space dimension. We refer to 
Vovelle and Debussche \cite{Vovelle2010} (see 
also Chen {\em et al.}~\cite {Chen:2012fk}) for an existence result 
in the multidimensional case. In \cite{Vovelle2010} the authors 
obtain the existence via the kinetic formulation, while \cite{Chen:2012fk} uses 
the $BV$ framework. Another recent contribution to the 
multidimensional problem is Bauzet, Vallet, and Wittbold \cite{BaVaWit}, where the  
question of existence is settled via the Young measure approach. 
We also mention the very recent contributions \cite{Lions2013, Lions2014} by 
Lions, Perthame, and Souganidis on conservation 
laws with rough (stochastic) fluxes.

During the last decade there has been many contributions in the larger area of 
stochastic partial differential equations that are driven by L\'{e}vy noise. 
An worthy reference on this subject is \cite{peszat}. 
However, there are few results on the specific problem 
of conservation laws with L\'{e}vy noise. The present article marks a first step 
in our endeavor to build a comprehensive theory of mixed hyperbolic-parabolic 
equations driven by noise containing both diffusion and jump effects. 
We draw inspiration from \cite{BaVaWit, Eymard1995, panov} and the 
notion of {\em entropy process solutions} when utilizing the theory of 
Young measures as a tool to prove the existence of entropy solutions 
to L\'{e}vy driven conservation laws. The presence of L\'{e}vy noise asks for 
solutions that have discontinuous sample paths. 
Also, the entropy inequalities will have non-localities in them as 
a consequence of the It\^{o}-L\'{e}vy chain rule. 
As a result the ``strong entropy" approach 
of Feng and Nualart \cite{nualart:2008} seems 
difficult to adapt to the present situation.

The remaining part of the paper is organized as follows.  We state the assumptions, 
detail the technical framework, and state the main 
results in Section \ref{technical}.  In Section \ref{existence-apriori-estimate}, we 
establish the wellposedness and derive apriori estimates for the 
viscous approximations. Section \ref{sec:existence-of-gen-solution} deals with 
the existence of entropy solutions via Young measure 
valued limits of viscous approximations.  
Finally, Section \ref{uniqueness} is devoted to the 
question of uniqueness of entropy solutions. 

\section{Technical framework and statements of the main results}\label{technical}
            
Here and in the sequel we use the letters
$C,K$, etc.~to denote various generic constants. There are situations 
where constants may change from line to line, but the notation is kept 
unchanged so long as it does not impact the central idea. 
The Euclidean norm on any $\mathbb{R}^d$-type space is denoted by $|\cdot|$. 
The space $C^n(\R^d)$ consist of the real valued functions on $\R^d$ 
that are $n$-times continuously differentiable. For a constant $T>0$, the 
space-time cylinder $[0,T)\times \R^d$ is 
denoted by $\Pi_T$ and the symbol $\Pi_{\infty}$ stands for 
$[0,\infty)\times \R^d$. The spaces $C_c^{1,2}(\Pi_T)$ 
and $C_c^{1,2}(\Pi_\infty)$ contain the compactly supported 
functions on $\Pi_T$ and $\Pi_\infty$, respectively, which 
are continuously differentiable in the time variable and 
twice continuously differentiable in the space variable.

\subsection{Entropy inequalities}
We begin this section with a formal derivation of the entropy 
inequalities \'{a} la Kru{\v{z}}kov, keeping in mind the need to replace 
the traditional chain rule of deterministic calculus by the It\^{o}-L\'{e}vy chain rule. 
Let $0\le \beta\in C^2(\R)$ be a real valued convex function, 
and $\zeta$ be such that $\zeta'(r) = \beta'(r)F'(r)$.
For a small positive number $\eps > 0$, assume that 
the parabolic perturbation 
\begin{align}
 d u(t,x) + \mbox{div}_x F(u(t,x))\,dt & 
 =  \int_{|z|> 0} \eta(x,u(t,x);z)\,\tilde{N}(dz,dt) + \eps\Delta u(t,x)\,dt
\notag
\end{align} 
of \eqref{eq:levy_stochconservation_laws} has a 
strong (predictable) solution $u_\eps(t,x)$.  Now we apply the It\^{o}-L\'{e}vy 
formula to $\beta(u_\eps(t,x))$, yielding
\begin{align}
  &d\beta(u_\eps(t,x) ) + \mbox{div}_x \zeta(u_\eps(t,x))\,dt \notag \\
 = &\int_{|z|>0} \Big(\beta(u_\eps(t,x) +\eta(x,u_\eps(t,x);z))
 - \beta(u_\eps(t,x))\Big)\tilde{N}(dz,dt) \notag \\
&\quad+ \int_{|z|>0}\Big(\beta(u_\eps(t,x) +\eta(x,u_\eps(t,x);z))
- \beta(u_\eps(t,x))-\eta(x,u_\eps(t,x);z)\beta^\prime(u_\eps(t,x))\Big)m(dz)\,dt 
\notag \\
&\qquad + \Big(\eps \Delta_{xx}\beta(u_\eps(t,x)) 
-\eps \beta''(u_\eps(t,x))|\nabla_x u_\eps(t,x)|^2\Big)\,dt .\notag
\end{align} 
Given a nonnegative test function 
$\psi\in C_{c}^{1,2}([0,\infty)\times \rd) $, we apply the 
It\'{o}-L\'{e}vy product rule to $\beta(u_\eps(t,\cdot) )\psi(t,\cdot)$, arriving at 
\begin{align}
  &d\big[\beta(u_\eps(t,x))\psi(t,x)\big] 
  = \partial_t\psi(t,x) \beta(u_\eps(t,x)) \,dt -\psi(t,x)\mbox{div}_x \zeta(u_\eps(t,x))\,dt \notag \\
& + \int_{|z|>0} \psi(t,x)\Big(\beta(u_\eps(t,x) +\eta(x,u_\eps(t,x);z))
-\beta(u_\eps(t,x))\Big)\tilde{N}(dz,dt) \notag \\
&+ \int_{|z|>0}\psi(t,x)\Big(\beta(u_\eps(t,x) +\eta(x,u_\eps(t,x);z))
- \beta(u_\eps(t,x))-\eta(x,u_\eps(t,x);z)\beta^{\prime}(u_\eps(t,x))\Big)m(dz)\,dt 
\notag \\
& +\psi(t,x) \Big(\eps \Delta_{xx}\beta(u_\eps(t,x)) 
-\eps \beta''(u_\eps(t,x))|\nabla_x u_\eps(t,x)|^2\Big)\,dt .\notag
\end{align}
We integrate the above equality with respect to $(t,x)$ and 
use  $\langle \cdot,\cdot\rangle$ to denote inner product in $L^2(\R^d)$. 
The result is
\begin{align}
 \notag 
 0 & \le  \langle \beta(u_\eps(T,.)),\psi(T,\cdot)\rangle \\  
 &\leq   \langle \beta(u_\eps(0,.)),\psi(0,\cdot)\rangle+ 
 \int_{0}^{T} \langle \zeta(u_\eps(r,.)),\nabla_x \psi(r,\cdot)\rangle\,dr
 \notag \\&+\int_{0}^T \langle \beta(u_\eps(r,\cdot)),\partial_t \psi(r,\cdot)\rangle\,dr  
 + \mathcal{O}(\eps)\notag \\
& +\int_{0}^T\int_{|z|>0} \langle \beta(u_\eps(r,.) +\eta(.,u_\eps(r,.);z))
- \beta(u_\eps(r,.)),\psi(r,\cdot)\rangle \tilde{N}(dz,dr)  \notag \\
\label{eq:entropy_derivation}& 
+\int_{0}^T\int_{|z|>0} \langle \beta(u_\eps(r,\cdot)+\eta(.,u_\eps(r,.);z))
- \beta(u_\eps(r,.))-\eta(.,u_\eps(r,.);z)\beta^{\prime}(u_\eps(r,.)),\psi(r,\cdot)\rangle\,m(dz)\,dr. 
\end{align} 
The notation $\mathcal{O}(\eps)$ is used to denote quantities 
that depend on $\eps$ and are bounded above by $C\eps$.
Clearly, the above inequality is stable under the limit $\eps \goto 0 $, if 
the family $\{u_\eps \}_{\eps >0}$ has $ L_{\text{loc}}^p $-type stability.
Just as the deterministic equations, the above 
inequality \eqref{eq:entropy_derivation} provides us with the entropy condition. 
We now formally define the entropy solutions.

\begin{defi}[entropy flux pair]
A pair $(\beta,\zeta) $ is called an entropy flux pair 
if $ \beta \in C^2(\R) $ and $\beta \ge0$, and 
$\zeta = (\zeta_1,\zeta_2,....\zeta_d):\R \mapsto\rd $ is a vector field satisfying
$\zeta'(r) = \beta'(r)F'(r)$ for all $r$. An entropy flux pair $(\beta,\zeta)$ is called 
convex if $ \beta^{\prime\prime}(\cdot) \ge 0$.  
\end{defi}

\begin{defi} [entropy solution]\label{defi:stochentropsol}
A $ L^2(\rd)$-valued $\{\mathcal{F}_t: t\geq 0 \}$-predictable 
stochastic process $u(t)= u(t,x)$ is called a stochastic 
entropy solution of \eqref{eq:levy_stochconservation_laws} if\\
(1) For each $ T>0$, $ p=2,3,4,\ldots$, 
$$
\sup_{0\leq t\leq T} E\Big[||u(t)||_{p}^{p}\Big] <\infty.
$$ 

\noindent (2) Given any non-negative test function 
$\psi\in C_{c}^{1,2}([0,\infty )\times\rd)$ 
and any convex entropy pair $(\beta,\zeta)$ 
with $\beta^{\prime}$ bounded, it holds that 
\begin{align}
 & \langle \psi(0,\cdot), \beta(u(0,\cdot))\rangle 
 + \int_{t=0}^T\langle \partial_t \psi(t,\cdot), \beta(u(t,\cdot)) \rangle \,dt
+ \int_{r=0}^T\langle \zeta(u(r,.)),\nabla_x \psi(r,\cdot)\rangle\,dr \notag \\ 
& \qquad +\int_{r=0}^T\int_{|z|>0} \langle \beta(u(r,.) +\eta(.,u(r,.);z))
-\beta(u(r,.)),\psi(r,\cdot)\rangle \tilde{N}(dz,dr)  \notag \\
& \qquad +\int_{r=0}^T\int_{|z|>0} \langle \beta(u(r,.) +\eta(.,u(r,.);z))- \beta(u(r,.))
-\eta(.,u(r,.);z)\,\beta^{\prime}(u(r,.)),\psi(r, \cdot)\rangle\, m(dz)\,dr \notag \\
& \qquad\qquad 
\ge 0 \quad \text{$P$-almost surely.}\notag
\end{align}
\end{defi} 

The aim of this paper is to establish the 
existence and uniqueness of entropy solutions according 
to Definition \ref{defi:stochentropsol}, and we will
do so under the following assumptions:

\vspace{.3cm}

\begin{Assumptions}
\item \label{A1} For $ k = 1,2,\ldots,d$, the functions 
$F_k(s) \in C^2(\R)$; and $F_k(s)$, $F_k^\prime(s)$, 
and $F_k^{\prime\prime}(s)$ have at 
most polynomial growth in $s$. 

\item \label{A2} There exist positive constants 
$K > 0$ and $\lambda^* \in [0,1)$  such that 
\begin{align*}| \eta(x,u;z)-\eta(y,v;z)|  
\leq  (\lambda^* |u-v|+ K |x-y|) ( |z|\wedge 1) 
~\text{for all}~ x,y \in \R^d;~ u,v \in \R;~~z\in \R.
 \end{align*}

 \item \label{A3} The L\'{e}vy measure $m(dz)$ is a Radon measure 
 on $\R\backslash \{0\}$ with a possible singularity at $z=0$, which satisfies
\begin{align*}
    \int_{\R_z}(|z|^2\wedge 1)\, m(dz) < \infty. 
\end{align*}

\item \label{A4} There exists a nonnegative function 
$g\in L^\infty(\R^d)\cap L^2(\R^d)$  such that 
\begin{align*}
	|\eta(x,u;z)| \le g(x)(1+|u|)(|z|\wedge 1)
\end{align*} 
{for all}~ $(x,u,z)\in \R^d \times \R\times \R.$
\end{Assumptions}

\begin{rem}
We are able to accommodate polynomially growing flux 
function as a result of the requirement that the entropy solutions 
satisfy $L^p$ bounds for all $p\ge 2$. This in turn forces us to choose 
initial data that are in $L^p$ for all $p$.  It is possible to accommodate 
initial conditions which are only $L^2$, but we would then 
require the flux function to be globally Lipschitz. 
Furthermore, the assumption \ref{A2} is needed
to handle the nonlocal nature of the 
entropy inequalities. 
\end{rem}

\subsection{Generalized entropy solutions} 
The focus of this paper is well-posedness for multidimensional problems. 
Contrary to one dimensional problems \cite{BisMaj, nualart:2008}, 
compensated compactness is not applicable and 
securing proper compactness for vanishing 
viscosity approximations requires an alternative viewpoint. 
One option is to further weaken the notion of entropy solutions to 
accommodate solutions that are parametrized  measures (Young measures). 
However, in view of \cite{BaVaWit, panov} (and Lemma \ref{lem:measure-conversion}), we 
can equivalently look for generalized entropy solutions that 
are $L^{2}(\R^d\times (0,1))$-valued processes.  

\begin{defi} [Generalized entropy solution]\label{defi: young_stochentropsol}
An $L^2(\rd \times (0,1))$-valued $\{\mathcal{F}_t: t\geq 0 \}$-predictable 
stochastic process $v(t)= v(t,x,\alpha)$ is
called a generalized stochastic entropy solution 
of \eqref{eq:levy_stochconservation_laws} if \\
(1) for each $ T>0$, $ p=2,3,4, \ldots$, 
$$
\sup_{0\leq t\leq T} E\Big[||v(t,\cdot,\cdot)||_{p}^{p}\Big] 
<\infty. 
$$

\noindent (2) For $0\leq \psi\in C_{c}^{1,2}([0,\infty )\times\rd)$ 
and each convex entropy pair $(\beta,\zeta) $ with $\beta^{\prime}$ 
bounded, it holds that
\begin{align}
& \langle \psi(0,\cdot), \beta(v(0,\cdot))\rangle 
+ \int_{0}^T \int_{\alpha=0}^1\langle \partial_t \psi(r,\cdot), \beta(v(r,\cdot,\alpha)) \rangle \,d\alpha\,dr
+ \int_{0}^T \int_{\alpha=0}^1 \langle \zeta(v(r,.,\alpha)),\nabla_x \psi(r,\cdot)\rangle\,d\alpha\,dr \notag \\ 
& \qquad +\int_{0}^T\int_{|z|>0} \int_{\alpha=0}^1 \langle \beta(v(r,.,\alpha) +\eta(.,v(r,.,\alpha);z))
-\beta(v(r,.,\alpha)),\psi(r,\cdot)\rangle\,d\alpha\, \tilde{N}(dz,dr) \notag\\
& \qquad 
+\int_{0}^T\int_{|z|>0} \int_{\alpha=0}^1 \langle \beta(v(r,.,\alpha)+\eta(.,v(r,.,\alpha);z))
- \beta(v(r,.,\alpha)) \notag 
\\ & \qquad \qquad \qquad \qquad \qquad \qquad\qquad 
-\eta(.,v(r,.,\alpha);z)\beta^{\prime}(v(r,.,\alpha)),\psi(r, \cdot)\rangle\,d\alpha\, m(dz)\,dr \notag 
\\ &\qquad\qquad \ge 0 \quad P-\text{a.s}\notag
\end{align}
\end{defi}

We can now state the main results of this paper. 

\begin{thm}[existence]\label{thm:existenc}
Suppose assumptions \ref{A1}-\ref{A4} hold, and that the 
$\bigcap_{p=1,2,\ldots} L^p(\rd)$-valued 
$\mathcal{F}_0$-measurable random 
variable  $ u_0$ satisfies
\begin{equation}\label{eq:initass}
E \Big[||u_0||_{p}^{p} + ||u_0||_{2}^{p}\Big] 
< \infty,  \qquad\text{for}~ p=1,2,\ldots~.
\end{equation}
Then  there exists a generalized entropy solution 
of \eqref{eq:levy_stochconservation_laws}-\eqref{initial_cond} in the 
sense of Definition \ref{defi: young_stochentropsol}.
\end{thm}

\begin{thm}[uniqueness]\label{thm:uniqueness}
Suppose assumptions \ref{A1}-\ref{A4} hold, and that the 
$\bigcap_{p=1,2,\ldots} L^p(\rd)$-valued 
$\mathcal{F}_0$-measurable random variable $u_0$ satisfies \eqref{eq:initass}. 
Then the generalized entropy solution of 
\eqref{eq:levy_stochconservation_laws}-\eqref{initial_cond} 
is unique. Moreover, it is the unique stochastic entropy solution.
\end{thm}

The above definitions do not say anything explicit about how a 
solution satisfies the initial condition. 
However, it follows after simple considerations that it 
satisfies the initial condition in a 
certain weak sense (see \cite{malek, vallet2000}). 

\begin{lem}\label{lem:initial-cond}
Any generalized entropy solution $u(t,\cdot,\cdot)$ of 
\eqref{eq:levy_stochconservation_laws}-\eqref{initial_cond} satisfies the 
initial condition in the following sense: for every non negative 
test function $\psi\in C_c^2(\R^d)$ such that $\supp(\psi) = K$,
\begin{align}
  \lim_{h\rightarrow 0}E \Big[\frac 1h \int_{t=0}^h\int_K 
  \int_{\lambda=0}^1|u(t,x,\lambda) -u_0(x)| 
  \psi(x)\,d\lambda\,dx\, dt \Big]= 0\notag
\end{align}
\end{lem}

\begin{proof}
Since $K$ is of finite measure, it is enough to prove 
\begin{align}
  \lim_{h\rightarrow 0}E \Big[\frac 1h 
  \int_{t=0}^h\int_K\int_{\lambda=0}^1 |u(t,x,\lambda) -u_0(x)|^2 \psi(x)\, d\lambda\,dx\, dt \Big]
  = 0. \notag
\end{align}
For $\delta\in (0,1)$, let $K_\delta = \{x: \text{dist}(x,K) \le \delta \}$.  
Note that, for any $\delta > 0$,
\begin{align*}
&E \Big[\int_{K} \int_{\lambda=0}^1 |u(t,x,\lambda) -u_0(x)|^2\psi(x) \,d\lambda\,dx \Big] 
   \\  & \qquad 
   \le 2 E \Big[\int_{y\in K_\delta}\int_{x\in K}\int_{\lambda=0}^1 |u(t,x,\lambda) -u_0(y)|^2\psi(x)  
     \varrho_{\delta}(x-y)\,d\lambda\,dx\,dy\Big]\notag
     \\ &\qquad\qquad 
     +2 E \Big[\int_{y\in K_\delta}\int_{x\in K} |u_0(y) -u_0(x)|^2 \psi(x) 
     \varrho_{\delta}(x-y)\,dx\,dy\Big],
\end{align*} 
where $\{\varrho_{\delta}\}_{\delta >0}$ is the sequence 
of standard mollifiers in $\R^d$. In other words,
\begin{align}
& E \Big[\frac{1}{h}\int_{t=0}^h\int_{K} \int_{\lambda=0}^1 |u(t,x,\lambda) -u_0(x)|^2
\psi(x) \,d\lambda\,dx \,dt\Big]\notag
\\ \qquad 
& \qquad \le 2 E \Big[\frac{1}{h}\int_{t=0}^h\int_{y\in K_\delta}\int_{x\in K}
\int_{\lambda=0}^1 |u(t,x,\lambda) -u_0(y)|^2
\psi(x)  \varrho_{\delta}(x-y)\,d\lambda\,dx\,dy\,dt\Big]\notag
\\ &\qquad\qquad
+2 E \Big[\int_{y\in K_\delta}\int_{x\in K} |u_0(y) -u_0(x)|^2 
\psi(x) \varrho_{\delta}(x-y)\,dx\,dy\Big].\label{eq:intial_cond_weak-4}
\end{align} 
Let $\psi(t,x)= \gamma(t)\psi(x)\varrho_\delta(x-y)$, 
where $\gamma(t)= \frac{h-t}{h}$ for $0\le t\le h$. 
Now, let $\beta(u) = (u -u_0(y))^2$ and $\xi(u)= \int_0^u 2 (r-u_0(y)) 
F^\prime(r)\, dr = 2\int_0^u r F^\prime(r)\, dr -2u_0(y)(F(u)-F(0))
\le C (1+|u_0(y)|^2+ |u|^p)$ for some positive integer $p$. 
We now apply Definition \ref{defi: young_stochentropsol} with the 
entropy flux pair $(\beta,\xi)$, obtaining
\begin{align*}
& 0 \le E \Big[\int_{y\in K_\delta}\int_{x\in K} |u_0(y) -u_0(x)|^2 
\psi(x) \varrho_{\delta}(x-y)\,dx\,dy \Big]
\\ &\qquad\quad -E \Big[\frac{1}{h}\int_{t=0}^h
\int_{y\in K_\delta}\int_{x\in K}\int_{\lambda=0}^1 |u(t,x,\lambda) -u_0(y)|^2\psi(x)
\varrho_{\delta}(x-y)\,d\lambda\,dx\,dy\,dt\Big]
\\&\qquad\quad\qquad+ C\delta^{-2}\int_{r=0}^hE\Big[ \int_{y\in K_\delta}\int_{x\in K} 
\int_{\lambda=0}^1(1+|u(r,x,\lambda)|^p+|u_0(y)|^2)\,d\lambda \,dx\,dy\Big]\,dr 
\\ & \qquad\quad+\frac{C^{\prime\prime}}{\delta}\int_{r=0}^h E\Big[ 
\int_{x\in K}\int_{\lambda=0}^1\int_{|z|>0}|\eta(x, u(r,x,\lambda), z)|^2\,m(dz)\,d\lambda \,dx\Big]\,dr,
\end{align*} 
and so
\begin{align*}
& E \Big[\frac{1}{h}\int_{t=0}^h\int_{y\in K_\delta}\int_{x\in K}\int_{\lambda=0}^1 
|u(t,x,\lambda) -u_0(y)|^2\psi(x)  \varrho_{\delta}(x-y)\,d\lambda\,dx\,dy\,dt\Big]
\\ &\qquad \le   E \Big[\int_{y\in K_\delta}\int_{x\in K} |u_0(y) -u_0(x)|^2 \psi(x) \varrho_{\delta}(x-y)\,dx\,dy \Big]
\\ &\qquad\qquad
+ C\delta^{-2}\int_{r=0}^h E  \Big[\int_{y\in K_\delta}\int_{x\in K}
\int_{\lambda=0}^1(1+|u(r,x,\lambda)|^p+|u_0(y)|^2)\,d\lambda \,dx\,dy\Big]\,dr 
\\ & \qquad\qquad 
+ \frac{C^{\prime\prime}}{\delta}\int_{r=0}^h E \Big[ 
\int_{x\in K}\int_{\lambda=0}^1\int_{|z|>0} | \eta(x, u(r,x,\lambda), z)|^2\,m(dz)\,d\lambda \,dx\Big]\,dr.
\end{align*} 
Hence, by passing to the limit $h\rightarrow 0$, 
\begin{align}
&\limsup_{h\rightarrow 0} E \Big[\frac{1}{h}\int_{t=0}^h\int_{y\in K_\delta}\int_{x\in K}
\int_{\lambda=0}^1 |u(t,x,\lambda) -u_0(y)|^2\psi(x)  \varrho_{\delta}(x-y)\,d\lambda\,dx\,dy\,dt\Big]\notag
\\ & \qquad 
\le   E \Big[\int_{y\in K_\delta}\int_{x\in K} |u_0(y) -u_0(x)|^2 \psi(x) \varrho_{\delta}(x-y)\,dx\,dy\Big].\label{eq:intial_cond_weak-5}
\end{align} 
Combining \eqref{eq:intial_cond_weak-4} 
and \eqref{eq:intial_cond_weak-5} yields
\begin{align}
\notag
& \limsup_{h\rightarrow 0} E \Big[\frac{1}{h}
\int_{t=0}^h\int_{K} \int_{\lambda=0}^1 |u(t,x,\lambda) -u_0(x)|^2\psi(x) \,d\lambda\,dx \,dt\Big]
\\ & \qquad \le 4 E \Big[\int_{y\in K_\delta}\int_{x\in K} 
|u_0(y) -u_0(x)|^2 \psi(x) \varrho_{\delta}(x-y)\,dx\,dy\Big]
\quad \text{for all} ~\delta > 0.
\label{eq:intial_cond_weak-6}
\end{align}
We now  let $\delta \rightarrow 0$ in 
the right-hand side of 
\eqref{eq:intial_cond_weak-6}, which gives
\begin{align}
\notag& \limsup_{h\rightarrow 0} E \Big[\frac{1}{h}\int_{t=0}^h\int_{K} 
\int_{\lambda=0}^1 |u(t,x,\lambda) -u_0(x)|^2\psi(x) \,d\lambda\,dx \,dt\Big]\le 0;
\end{align} 
the proof is complete since $\psi \ge 0$.
\end{proof}

Before concluding this section, we introduce a class of entropy functions. 
Let $\beta:\R \rightarrow \R$ be a $C^\infty$ function satisfying 
\begin{align*}
      \beta(0) = 0,\quad \beta(-r)= \beta(r),\quad 
      \beta^\prime(-r) = -\beta^\prime(r),\quad \beta^{\prime\prime} \ge 0,
\end{align*} 
and 
\begin{align*}
	\beta^\prime(r)=\begin{cases} -1\quad \text{when} ~ r\le -1,\\
                               \in [-1,1] \quad\text{when}~ |r|<1,\\
                               +1 \quad \text{when} ~ r\ge 1.
                 \end{cases}
\end{align*} 
For any $\vartheta > 0$, define $\beta_\vartheta:\R \rightarrow \R$ by 
$\beta_\vartheta(r) = \vartheta \beta(\frac{r}{\vartheta})$. 
Then
\begin{align}\label{eq:approx to abosx}
 |r|-M_1\vartheta \le \beta_\vartheta(r) \le |r|\quad 
 \text{and} \quad |\beta_\vartheta^{\prime\prime}(r)| 
 \le \frac{M_2}{\vartheta} {\bf 1}_{|r|\le \vartheta},
\end{align} 
where $M_1 = \sup_{|r|\le 1}\big | |r|-\beta(r)\big |$ and 
$M_2 = \sup_{|r|\le 1}|\beta^{\prime\prime} (r)|$.

By simply dropping $\vartheta$, for $\beta= \beta_\vartheta$ ~ we define 
\begin{align*}
&F_k^\beta(a,b)=\int_{b}^a \beta^\prime(\sigma-b)F_k^\prime(\sigma)\,d(\sigma), \\
&F^\beta(a,b)=(F_1^\beta(a,b),F_2^\beta(a,b),\ldots,F_d^\beta(a,b)),\\
&F_k(a,b)= \text{sign}(a-b)(F_k(a)-F_k(b)) ,\\
&F(a,b)= (F_1(a,b),F_2(a,b),\ldots,F_d(a,b)).
\end{align*}

\section{Existence and a-priori estimates for the viscous problem}\label{existence-apriori-estimate}
The entropy inequalities, and the corresponding well-posedness result, are 
reliant on the fact that one can (spatially) regularize the solution of
\eqref{eq:levy_stochconservation_laws} by adding small diffusion operator. 
Therefore, in this section, we will provide a detailed 
analysis of the following viscous problem:
\begin{align} 
du(t,x) + \mbox{div}_x F(u(t,x)) \,dt
& = \int_{|z|> 0} \eta(x,u(t,x);z)\, \tilde{N}(dz,\,dt)
+ \eps \Delta_{xx} u \,dt, 
\quad t>0, ~ x\in \R^d,
\label{eq:levy_stochconservation_laws-viscous}
\end{align}
with initial condition \eqref{initial_cond}. 
To the best of our knowledge, the answers to the wellposedness questions 
for L\'{e}vy driven SPDEs are not readily available in its full generality to cover \eqref{eq:levy_stochconservation_laws-viscous}. However, a relevant 
reference is \cite{xu}, where the one-dimensional viscous Burgers 
equation with L\'{e}vy noise is studied. 

Throughout this section, we impose the 
following regularity assumptions:
\begin{Assumptions2}
    \item \label{B1} The function $F: \R\goto\R^d$ is 
    smooth, i.e., $ F_k\in C^\infty$, and the $n$-th derivative 
    satisfies $|\partial_u^nF_k(u)| \le K_n$ for some constant $K_n$ and 
    for all $n\in \mathbb{N}$ and $k=1,\ldots,d.$
    \item \label{B2} For every $n\in\mathbb{N}$, $\partial_u^n \eta(x,u; z)$  
    and $D_x^n \eta(x,u;z)$ exist and are continuous. 
    Moreover,\\ $\eta(\cdot, u; z)\in \mathcal{S}(\R^d)$.
  
    \item \label{B3} For every $n\in \mathbb{N}$, there exists 
    $K_n(x)\in L^2(\R^d)\cap L^{\infty}(\R^d)$ such that
    $$
    |\partial_u^n\eta(x,u;z)|+ |D_x^n \eta(x,u,z)|\le K_n(x)(1\wedge |z|).
    $$
    
    \item \label{B4} The initial condition $u_0$ belongs to $\mathcal{S}(\R^d)$.                
 \end{Assumptions2}
 
\vspace{.4cm}
It is implied by \ref{B4} that $E[||u_0||_{2}^2]< \infty$, and we 
introduce the following Picard-type iterates: for any natural 
number $n\ge 0$, define 
\begin{align}
  & u^0(t,x) = u_0(x),\notag\\
  \label{eq:picard}& 
  du^n(t,x) + \mbox{div}_x F(u^{n-1}(t,x))dt 
  = \eps \Delta u^n(t,x) dt+ \int_{|z|> 0} \eta(x, u^{n-1}(t,x);z)\,\tilde{N}(dz,dt)\\
    \notag &u^n(0,x) = u_0(x).
\end{align} 
Let $G_\eps(t,x)$ be the heat kernel 
associated with operator $\eps \Delta_{xx}$ i.e 
\begin{align*}
 G(t,x)\equiv G_\eps(t,x) = \frac{1}{(4\pi\eps t)^{\frac{d}{2}}} 
 e^{\frac{-|x|^2}{4\eps t}},\quad t> 0.
\end{align*} 
We are looking for a $L^2(\R^d)$-valued predictable 
process $u^n(t,x)$ that qualifies as the mild solution to \eqref{eq:picard}. 
In other words, we want a predictable process $u^n(t,x)$ that satisfies  
\begin{equation}
\label{eq:iterate_representation} 
   \begin{split}
   u^n(t,x) = & \int_{\rd_y} G(t,x-y) u_0(y) \, dy 
   - \int_{s=0}^t \int_{\rd_y} G(t-s,x-y)\sum_{i=1}^d \partial_{y_i}F_i(u^{n-1}(s,y))\, dy\,ds
   \\ & \qquad 
   + \int_{s=0}^t \int_{|z| > 0}\int_{\rd_y} G(t-s, x-y) \eta(y, u^{n-1}(s,y);z) \,dy\,\tilde{N}(dz, ds),
   \end{split}
\end{equation} 
almost surely, for every $t$. 
Note that the c\'{a}dl\'{a}g solution 
$v(t,x)$ to \eqref{eq:picard} is given by 
\begin{equation}\label{eq:iterate_representation-1} 
\begin{split}
v(t,x) = &\int_{\rd_y} G(t,x-y)u_0(y)\, dy 
- \int_{s=0}^t \int_{\rd_y} G(t-s,x-y)\sum_{i=1}^d \partial_{y_i}F_i(u^{n-1}(s,y)) \,dy\,ds
\\ 
& + \int_{s=0}^t \int_{|z| > 0}\int_{\rd_y} G(t-s, x-y) 
\eta(y, u^{n-1}(s,y);z) \,dy \,\tilde{N}(dz, ds).
\end{split}
\end{equation} 
Moreover, the martingale term on the right-hand 
side of \eqref{eq:iterate_representation-1} 
is stochastically continuous and c\'{a}dl\'{a}g. 
Therefore, $u^n(t,\cdot)= v(t-,\cdot)$ would definitely exist and 
for any fixed $t$, $u^n(t,x)= v(t-,x)$ almost surely. 
In other words, $u^n(t,\cdot)= v(t-,\cdot)$ is 
c\'{a}gl\'{a}d (hence predictable) and 
satisfies \eqref{eq:iterate_representation}.
 
If $u^0(x)$ is assumed to be smooth, then the first iterate $u^1(t,x)$ is 
immediately well defined. However, in order to make sense of $u^n(t,x)$ 
for any $n$, one needs to establish some essential regularity properties for $u^{n-1}$. 
The assumptions \ref{B1}-\ref{B4} will be used for this purpose.
\begin{lem}\label{lem:heat kernel-convolution}
Let $h = h(s) =h(s,x)$ be a predictable 
process with trajectories in $L^2\big([0,T]; H^p(\R^d)\big)$, for 
all $p=1,2,~\ldots$ and $h(s,\cdot) \in \mathcal{S}(\R^d)$. 
Furthermore, let 
\begin{align*}
     V(t,x)=\int_{s=0}^t \int_{\R^d_y} G(t-s,x-y) h(s,y)\,dy\, ds
\end{align*} 
Then $V(t)= V(t,x)$ is a predictable process with paths 
in $L^2\big([0,T]; H^p(\R^d)\big)\cap C\big([0,T]; H^p(\R^d)\big) $.  
In particular,
\begin{align*}
     \partial_{x_k}V(t,x)=
     \int_{s=0}^t \int_{\R^d_y} G(t-s,x-y) \partial_{y_k}h(s,y)\,dy\, ds.
   \end{align*} 
\end{lem}

\begin{proof}
The proof is a simple consequence of properties of convolution.
\end{proof} 

In addition, for a predictable process 
$g(\cdot,\cdot; z)\in L^2([0, T]; H^p(\R^d))$ with 
$p=1, 2,\ldots$, and
 \begin{align*}
      E\Big[\int_{s=0}^T \int_{\R^d_y}\int_{|z|>0}
      \big(g^2(s,y;z)+|D_y^n g(s,y;z)|^2\big)\,m(dz)\,dy\, ds\Big] <\infty,
 \end{align*} 
for every $n\in \mathbb{N}$ and 
$g(s,\cdot, z)\in \mathcal{S}(\R^d)$, the quantity
\begin{align}\notag
N_L(t,x)& =   \int_{s=0}^t \int_{|z|> 0}
\int_{\R_y^d} G(t-s,x-y)g(s,y; z)\,dy\, \tilde{N}(dz, ds) 
\end{align}
satisfies the following property for each $T > 0$.

\begin{lem}\label{lem:diff-integration}
The process $ N_L(t,x) \in  L^2([0, T]; H^p(\R^d))$ 
for $p= 1, 2, 3,\ldots$.  Moreover, $N_L(t,x)$ is c\'{a}dl\'{a}g and stochastically 
continuous and hence admits predictable version. Furthermore,
\begin{align}\label{eq:existence-differentialtion-2.2} 
\partial_{x_k}N_L(t,x)& 
= \int_{s=0}^t \int_{|z|> 0}\int_{\R_y^d} 
G(t-s,x-y)\partial_{y_k}g(s,y;z) \,dy\,\tilde{N}(dz, ds) 
\end{align} 
and $N_L(t,\cdot)\in C^\infty(\R^d)$.
\end{lem}

\begin{proof}
Once the representation \eqref{eq:existence-differentialtion-2.2} 
is established, the proof of the fact that $N_L(t,x) \in  L^2([0, T]; H^p(\R^d))$ 
is a straightforward application of It\^{o}-L\'{e}vy isometry. 
Moreover, the c\'{a}dl\'{a}g property of the right-hand side is the direct 
inheritance of being a stochastic integral, and the stochastic 
continuity is a direct consequence of the It\^{o}-L\'{e}vy isometry. 
   
In order to prove  the representation \eqref{eq:existence-differentialtion-2.2}, we 
have to show that the distributional derivative coincides with the right hand side. 
Let $\varphi\in C_c^\infty (\R^d)$ be a test function. As a consequence 
of Fubini's theorem and It\^{o}-L\'{e}vy isometry 
\begin{align*}
&E\Big[\big| \int_{\rd_x} \int_{s=0}^t\int_{|z|>0} \big\{-G(t-s)*_{x}g)(s,x;z)
\partial_{x_k}\varphi(x)\\&\hspace{5cm}
-G(t-s)*_{x}(\partial_{x_k} g)(s,x;z)\varphi(x) \big\}
\tilde{N}(dz, ds) \,dx\big|^2\Big] \\
& \quad = E\Big[\int_{s=0}^t\int_{|z|> 0} \big|\int_{\rd_x} \big\{\partial_{x_k}G(t-s)*_{x}g)(s,x;z)\varphi(x)\\
&\hspace{5cm}-G(t-s)*_{x}(\partial_{x_k} g)(s,x;z)\varphi(x) \big\}\,dx
\big|^2\,m(dz)\,ds\Big]\\
 & \quad = 0,
\end{align*} 
where we have used the integration by parts 
along with properties of convolution. 
In the above, $*_x$ signifies convolution in $x$ only. 
This representation shows that 
 $\partial_{x_k}N_L(t,\cdot)$ has trajectories 
 in $L^2([0, T]; L^2(\R^d)) $ and it has a predictable version. 

Replace  $g$ by  $\partial_{x_k}g$,  and repeat the above argument 
to conclude that $N_L$ has trajectories in\newline $L^2([0, T]; H^p(\R^d))$ 
and $N_L(t,\cdot) \in C^{p}(\R^d)$ for $p= 1, 2,\ldots$. 
\end{proof}

\begin{lem}\label{lem:scwartz-class}
 $N_L(t,\cdot)\in \mathcal{S}(\R^d)$ almost surely for all $t\in [0,T]$.
 \end{lem}

\begin{proof}
From Lemma \ref{lem:diff-integration}, we already know 
that $N_L(t,\cdot)\in C^\infty(\R^d).$ All we have to show is
\begin{align*}
\sup_{x\in \R^d} \big(|x|^n |N_L(t,x)|\big) < \infty\quad  a.s. 
\end{align*} for all $n\in \mathbb{N}$. 
On one hand, by Morrey's inequality there 
exists a universal constant $C > 0$ and $p > d$ such that 

\begin{align*}
\sup_{x\in \R^d} \big(|x|^n |N_L(t,x)|\big) \le C || |\cdot|^n |N_L(t,\cdot)| ||_{W^{1,p}},
\end{align*} 
for every positive integer $n$. On the other 
hand, direct computation reveals that, for $t> 0$, there exist 
$n$-th order polynomials $C_j(t)$ of $t$ and a non-zero constant $C_0$ such that 
\begin{align*}
 t^n \partial_{x_k}^n G(t,x-y) =\big( C_0x_k^{n} 
 + C_1(t) x_k^{n-1}+\cdots+ C_n(t) \big) G(t,x-y).
\end{align*} 
Therefore, by induction, it is sufficient to show 
that for all $j = 0,1,\ldots, n$,
\begin{align}
 \notag &\Big\|\int_{s=0}^t\int_{|z|> 0}  \int_{\rd_y} |t-s|^j 
 \partial_{x_k}^n G(t-s, x-y) g(s,y;z)\,dy\, \tilde{N}(dz, ds)\Big\|_{W_x^{1,p}}\\
& \qquad 
= \Big\|\int_{s=0}^t\int_{|z|> 0} \int_{\rd_y} |t-s|^j G(t-s, x-y) 
\partial_{y_k}^n g(s,y;z)\,dy\, \tilde{N}(dz, ds) \Big\|_{W_x^{1,p}} < \infty \quad a.s.,\notag
\end{align} 
for some $p > d$.  It follows from the Sobolev 
inequality that if $\ell > \frac{d}{2}$ then there is $p>d$ such that  
\begin{align*}
\big \|\cdot \big\|_{W^{1,p}} \le C \big\|\cdot \big\|_{H^{1+\ell}},
\end{align*}
along with the fact that, for all  $\ell \in \mathbb{N}$,
\begin{align*}
 &E\Big [\int_{\rd_x}\big|\int_{s=0}^t\int_{|z|> 0} 
 \int_{\rd_y} |t-s|^j G(t-s, x-y) \partial_{y_k}^{\ell+1} 
 g(s,y;z)\,dy \,\tilde{N}(dz, ds)\big|^2\,dx\Big]\\
& \qquad =  \int_{\rd_x}\int_{s=0}^t\int_{|z|> 0} E\Big[ |t-s|^{2j} |G(t-s, \cdot)*_x 
\partial_{y_k}^{\ell+1} g(s,\cdot;z)|^2\,m(dz)\,ds\,dx\Big]\\
 & \qquad \le C \int_{s=0}^t\int_{|z|> 0} 
E\Big[ \big\|\partial_{y_k}^{\ell+1} g(s,y;z)\big\|_2^2\Big] \,m(dz)\,ds < \infty.
\end{align*} 
In the above we have used Young's inequality for convolution. 
\end{proof}
 We finally conclude:
\begin{lem}\label{welldefinedness-iteration}
For  each $n = 1, 2,\ldots$, the process 
$u^n(t,\cdot)\in L^2\big([0,T]; H^p(\R^d)\big)$ 
for $p = 1,2,\ldots$ and $u^n(t,\cdot) \in \mathcal{S}(\R^d)$. 
Moreover, $u^n(t,\cdot)$ is stochastically continuous 
and has a c\'{a}gl\'{a}d (hence predictable) version. 
 \end{lem}

\subsection{Equivalence of mild, weak, and strong solutions} 
It is well-known in the context of SPDEs 
governed by diffusions that, under moderate 
conditions, mild solutions coincide with weak solutions. For SPDEs driven by 
jump-diffusions, mild solutions can also be shown to 
coincide with weak solutions under moderate conditions. 
Moreover, Lemma \ref{welldefinedness-iteration} 
ensures that $u^{n}(t,x)$ has the sufficient smoothness 
to be the strong solution of \eqref{eq:picard}. 
In our context, the next lemma states this fact. 
A detailed proof can be given, for example, by adapting the 
arguments given in \cite{xu}.    

\begin{lem}\label{lem:weak-mild}
For each $\varphi\in C_c^\infty (\R^d)$,
\begin{align*}
   &\langle u^n(t), \varphi\rangle - \langle u^n(0), \varphi\rangle \\
   & \qquad = \int_{s=0}^t  \langle F(u^{n-1}(s)), \nabla \varphi\rangle ds
    + \int_{r=0}^t \int_{|z|> 0} \int_{\rd_x} \eta(x, u^{n-1}(r,x); z)
    \varphi(x)\,dx\,\tilde{N}(dz, dr)\\
  &\hspace{6cm}  +\eps \int_{r=0}^t \langle\Delta\varphi, u^{n}(r)\rangle dr,
\end{align*} 
almost surely, for almost all $t\in [0,T]$.
\end{lem}

As in Feng and Nualart \cite{nualart:2008}, we 
also define the energy functional 
$e_{2r} :L^2(\rd)\mapsto [0,\infty] $ as follows:
$$
e_{2r}(u) = \frac{1} {2}||\Delta^r u||_2^2,\qquad  
r=0,1,2,3,\ldots.
$$

\begin{lem} \label{lem:energyestimate}
There exists a finite constant $C_{\eps,r,T} >0$, 
independent of $ n$, such that
\begin{align}\label{eq:energy-iterates}
E\left[e_{2r}(u^n(t))\right] \leq C_{\eps,r,T} 
\left( 1+ \sum_{k=0}^{r} E\left[ e_{2k}(u_0)\right]\right), 
\qquad t\leq T. 
\end{align}
\end{lem}

\begin{proof} 
We have seen that, for each 
$n=1,2,3,\ldots$, $u^n(t,\cdot) \in  \mathcal{S}(\rd)$, 
where $u^n$ is defined by
\begin{align}
u^n(t,x)&= \int_{\rd_y} G(t,x-y)u_0(y)\,dy 
- \int_{s=0}^{t}\int_{\rd_y} G(t-s,x-y)
\sum_{i=1}^{d} \partial_{y_i}F_i(u^{n-1}(s,y))\,dy\,ds \notag \\
& + \int_{s=0}^t  \int_{|z|>0}  \int_{\R^d_y} G(t-s,x-y)
\,\eta(y,u^{n-1}(s,y);z)\,dy\,\tilde{N}(dz,ds).\notag 
\end{align}
As $u^n(t,\cdot)\in\mathcal{S}(\rd)$, using a property of 
convolution, for $r=0,1,2,3,\ldots$, we have
\begin{align}
\Delta^r u^n(t,x)&= \int_{\rd_y} 
G(t,x-y)\Delta^r u_0(y)\,dy - \int_{s=0}^{t} \int_{\R^d_y} 
G(t-s,x-y)\Delta^r\Big(\sum_{i=1}^{d} 
\partial_{y_i}F_i(u^{n-1}(s,y))\Big)\,dy\,ds \notag \\
& + \int_{s=0}^t  \int_{|z|>0} \int_{\R^d_y}G(t-s,x-y)
\Delta^r\eta(y,u^{n-1}(s,y);z)\,dy\,\tilde{N}(dz,ds).\notag 
\end{align}
Therefore,  $\Delta^r u^n(t,x)$ solves 
the stochastic differential equation
\begin{align}
d\Delta^ru^n(t,x) + \grad \cdot \Delta^rF(u^{n-1}(t,x))dt &
=\eps\Delta(\Delta^r u^n)(t,x)dt 
+ \int_{|z|>0} \Delta^r\eta(x,u^{n-1}(t,x);z)\,\tilde{N}(dz,dt).\notag
\end{align}

Now we apply the It\^{o}-L\'{e}vy formula to the 
function $\phi(u)=u^2$, and integrate with 
respect to $x$, returning 
\begin{align}
&\int_{\rd_x} |\Delta^ru^n(t,x)|^2\,dx \notag \\
 & \qquad =\int_{\rd_x} |\Delta^ru_0(x)|^2\,dx 
 + 2\int_{\rd_x}\int_{s=0}^t \grad(\Delta^ru^n(s,x))\cdot
 \Delta^rF(u^{n-1}(s,x))\,ds\,dx\notag\\
 &\hspace{4.5cm}-2\eps\int_{\rd_x}\int_{s=0}^t \grad(\Delta^ru^n(s,x)) \cdot
 \grad(\Delta^ru^n(s,x))\,ds\,dx \notag \\
 &\qquad \qquad +\int_{s=0}^t  \int_{|z|>0} \int_{\rd_x}\left[\Big(\Delta^ru^n(s,x) 
 + \Delta^r\eta(x,u^{n-1}(s,x);z)\Big)^2 -(\Delta^ru^n(s,x))^2\right] 
 \,dx\,\tilde{N}(dz,ds)\notag \\
 &\qquad \qquad + \int_{s=0}^t  
 \int_{|z|>0} \int_{\rd_x} \Big[ \Big(\Delta^ru^n(s,x) 
 + \Delta^r\eta(x,u^{n-1}(s,x);z)\Big)^2 -(\Delta^ru^n(s,x))^2 \notag \\
 &\hspace{3cm}
 -2 \Delta^r\eta(x,u^{n-1}(s,x);z)\Delta^ru^n(s,x)\Big]\,dx\,m(dz)\,ds.\notag
\end{align}
Taking expectation and using Cauchy's inequality, we obtain
\begin{align*}
E\Big[e_{2r}(u^n(t))\Big] & \leq E\Big[ e_{2r}(u^n(0)) 
+ C_\eps \int_{s=0}^{t} ||\Delta^r F(u^{n-1}(s))||_2^2\,ds\Big] \\
& + E\Big[\int_{s=0}^t\int_{|z|>0}\int_{\rd_x} 
|\Delta_x^r\eta(x,u^{n-1}(s,x); z)|^2\,dx\, m(dz)\,ds\Big]. 
\end{align*}
Since $ F$ and $ \eta $ are smooth and $ |F_k^r(s)|\leq C_r$, 
for $r=0,1,2,\ldots$, and $|D_x^r \eta(x,u;z)| \le K_r(x) \min(|z|, 1)$ for some 
$K_r(x)\in L^2(\R^d)\cap L^{\infty}(\R^d)$, there exists a 
finite constant $\tilde{C}_{\eps,r,T} >0$, 
independent of $n$, such that
\begin{align*}
E\Big[e_{2r}(u^n(t))\Big] & \leq E\Big[ e_{2r}(u^n(0))\Big] 
+ \tilde{C}_{\eps,r,T}\Big(1+\int_{s=0}^t 
\sum_{k=1}^{r}E \Big[e_{2k}(u^{n-1}(s))\Big]\,ds \Big).
\end{align*}

Denote
$$
M_n(t)=\left( 1+\sum_{k=0}^{r} E\Big[e_{2k}(u^n(t))\Big]\right), 
\quad 
M(0)=\left(1+\sum_{k=0}^{r} E\Big[e_{2k}(u(0))\Big]\right). 
$$
Then, we have $M_n(t)\leq C M(0) + C \int_{s=0}^{t} M_{n-1}(s)\,ds$, for 
some constant $ C>0$, which is independent of $n$. 
By induction on $n$, we 
conclude that  there is a constant $K > 0$ such that $M_n(t) 
\leq C M(0) e^{KT}$ for every $t\in [0,T]$. 
Therefore, \eqref{eq:energy-iterates} follows.
\end{proof}

We now show that $u^n$ converges, in an appropriate sense, to 
a limiting process. This is done by a classical fixed point argument.

\begin{lem} \label{lem:convergence_proof} 
There exists a $L^2(\R^d)$-valued,
$\mathcal{F}_t$-predictable (and c\'{a}gl\'{a}d) process $u$ satisfying 

\begin{align}\label{eq:convergence_statement} 
\lim_{n\rightarrow \infty}
E \Big[ \sup_{0\le t\le T} ||u(t)-u^n(t)||_2\Big] = 0, 
\end{align}
and for $\ell=0,1,2,\ldots$,
\begin{align}\label{eq:convergence_statement_I}   
E\Big[ e_{2\ell}\big(u(t)\big)\Big] 
\le C_{\eps, \ell, T}\Big(1+\sum_0^\ell E\Big[e_{2k}(u_0)\Big]\Big), 
\qquad t\le T.
\end{align} 
In addition,
\begin{align}\label{eq:convergence_statement_II}  
\sup_{0\le t\le T} E \Big[ ||u(t)||_p^p\Big] < \infty.
\end{align} 
Furthermore, $u$ is a mild solution 
of \eqref{eq:levy_stochconservation_laws-viscous} 
in the following sense: 
\begin{align}
   \notag u(t,x) & = \int_{\rd_y} G(t,x-y)u_0(y)\, dy 
   - \int_{s=0}^t \int_{\rd_y} G(t-s,x-y)
   \sum_{i=1}^d \partial_{y_i}F_i(u(s,y)) \,dy\,ds\\
    \label{eq:iterate_representation-mild} 
    & \qquad 
    + \int_{s=0}^t \int_{|z| > 0}\int_{\rd_y} G(t-s, x-y) \eta(y, u(s,y);z) 
    \,dy\,\tilde{N}(dz,ds),
\end{align} 
almost surely, for every $t$.
\end{lem}

\begin{proof} 
Let $\mathbb{L}([0,T]: L^p(\R^d))$ be the space of c\'{a}gl\'{a}d and 
adapted $L^p(\R^d)$-valued processes on $[0,T]$.  The distance 
function between two processes $X$ and $Y$ is defined as
\begin{align}
|||X-Y|||^T_p= E\big[ \sup_{0\le t \le T} ||X(t)-Y(t)||_{L^p}\big]
\label{eq:metric}
\end{align}
It is well-known (see \cite{peszat,Protter1990}) that  the 
space $\mathbb{L}([0,T]: L^p(\R^d))$ equipped 
with the metric \eqref{eq:metric}  is complete.  
By Lemma \ref{welldefinedness-iteration}, it is easily seen 
that $u^n(t,\cdot) \in \mathbb{L}([0,T]: L^2(\R^d))$ and we want to 
show that $ \{u^n(t,\cdot)\}_n$ converges in this space. 
At first, by direct integration,
\begin{align*}
	\big\|\partial_{x_i}G(t,\cdot)\big\|_1 
	= \int_{\rd_x} |\partial_{x_i} G(t,x)| dx = C t^{\frac{-1}{2}},
	\qquad\text{for}~ t > 0.
\end{align*} 
We denote 
\begin{align*}
	&\mathcal{I}_1(u^n)(t,x) =  \int_{s=0}^t \int_{\rd_y} G(t-s,x-y)
	\sum_{i=1}^d \partial_{y_i}F_i(u^n(s,y)) \,dy\,ds\\
	&\mathcal{I}_2(u^n)(t,x) =    \int_{s=0}^t \int_{|z| > 0}\int_{\rd_y} 
	G(t-s, x-y) \eta(y, u^n(s,y);z) \,dy\,\tilde{N}(dz, ds).
\end{align*}                 
  
We define a deterministic measure on $[0,t]$ by
$$
\gamma(ds)=\gamma_t(ds)
=\big\|\partial_{x_i} G(t-s,.)\big\|_1\,ds 
= 2C d(t^{\frac{1} {2}}-(t-s)^{\frac{1} {2}}).
$$
Then
\begin{align*}
&E \Big[ \sup_{0\le s\le t}
\big\|\mathcal{I}_1(u^n)(s,.)-\mathcal{I}_1(u^k)(s,.)\big\|_2\Big]\notag 
\\ & \quad
=E \Big[\sup_{0\le s\le t} \Big( \int_{\rd_x}\Big|\int_{r=0}^{s}\int_{\rd_y}\sum_{i=1}^{d} 
\big[G(s-r,x-y)F'_i(u^n(r,y))\partial_{y_i}u^n(r,y) \\
 &\hspace{4cm}- G(s-r,x-y)F'_i(u^k(r,y))
 \partial_{y_i}u^k(r,y)\big]\,dr\,dy\Big|^2\,dx\Big)^{\frac 12} \Big]\notag\\
& \quad 
\leq C_1\sum_{i=1}^{d} E\left[\sup_{0\le s\le t} \Big(\int_{\R_x^d}\left |\int_{r=0}^{s} 
\big(|\partial_{x_i}G(s-r)|\con
|(u^n(r)-u^k(r)|\big)(x)\,dr\right |^2\, dx \Big)^{\frac 12} \right]\notag\\
& \quad 
\leq C_2 \sum_{i=1}^{d} E\left[ \sup_{0\le s\le t}   \Big(\int_{r=0}^{s} 
\big\|\big(|\partial_{x_i}G(s-r)|\con
|(u^n(r)-u^k(r))|\big)\big\|_2\,dr \Big)\right]\notag\\
& \quad \leq C_3 E\left[  \int_{r=0}^{t} 
\big\|u^n(r)-u^k(r)\big\|_2 \,\gamma(dr)  \right] \notag \\
& \quad \leq C_4  E\left[ \sup_{0\leq s\leq t} 
 \big\|u^n(s)-u^k(s)\big\|_2 \right]\sqrt{t}. \notag 
\end{align*}
The first inequality follows from integration by parts and 
$|F_k^\ell(r)| \leq C_\ell$, for $\ell=0,1,2,\ldots$, the second one 
follows from Minkowski inequality, while the third inequality 
follows from Young's inequality for convolutions. 
Therefore we obtain, 
\begin{align}\label{eq:convergence_I}
E\Big[ \sup_{0\le s\le t} \big\|\mathcal{I}_1(u^n)(s,.)-\mathcal{I}_1(u^k)(s,.)\big\|_2\Big]
\leq Ct^{\frac{1} {2}} E\Big[ \sup_{0\leq s\leq t} 
 \big\|u^n(s)-u^k(s)\big\|_2 \Big]. 
\end{align}  
We want a similar estimate for $\mathcal{I}_2(u^n)$. 
This requires maximal inequalities for stochastic convolutions with 
respect to a compensated Poisson measure, and relevant results 
are available in \cite{Hausenblas2008,marinelli2010}:
\begin{align}
& E \Big[ \sup_{0\le s\le t} \big\|\mathcal{I}_2(u^n)(s,\cdot)
-\mathcal{I}_2(u^k)(s,\cdot)\big\|_2\Big]  \notag\\
& \big(\text{by maximal inequality for stochastic 
convolution (see \cite[Example 3.1, Prop 1.3]{Hausenblas2008}}\big) \notag \\
& \quad \le C E \Big[\Big( \int_{s=0}^t 
\int_{|z|> 0} \int_{\rd_y} |\eta (y, u^n(s,y); z)-\eta(y,u^k(s,y);z)\big)|^2\,dy 
\, m(dz)\,ds\Big)^{\frac 12} \Big] \notag \\
& \quad \le C  E \Big[\Big( \int_{s=0}^t 
\int_{|z|> 0} \int_{\rd_y} |u^n(s,y)-u^k(s,y)|^2\,dy \min(1, |z|^2) 
\, m(dz)\,ds \Big)^{\frac 12}\Big] \notag \\
& \quad \le C \sqrt{t}\, E\big[\sup_{0\le s \le t}
||u^n(s)-u^k(s)||_2\big].\label{eq:convergence_III}
\end{align} 
Combine estimates \eqref{eq:convergence_I} 
and \eqref{eq:convergence_III}, and use \eqref{eq:iterate_representation} 
to conclude that there exist numbers $\alpha \in (0,1)$ and $T_0 > 0$,
which are independent of the initial condition $u_0$, such that 
\begin{align*}
    |||u^n -u^k|||_2^{T_0}  \le  \alpha ||| u^{n-1}-u^{k-1}|||_2^{T_0}.
\end{align*} 
Hence, by the Banach fixed point argument, we 
have short time existence in $\mathbb{L}([0,T_0]: L^2(\R^d))$, and 
pasting the short time existence one 
can argue for the existence in $\mathbb{L}([0, T]: L^2(\R^d))$. 
In other words, we have shown the existence of 
a c\'{a}gl\'{a}d and adapted process $u$ such 
that \eqref{eq:convergence_statement} holds.  
To conclude \eqref{eq:convergence_statement_I}, we simply apply 
Fatou's lemma and let $n\rightarrow \infty$ in \eqref{eq:energy-iterates}. 
In addition, \eqref{eq:convergence_statement_II} holds as a simple 
consequence of Sobolev embedding and \eqref{eq:convergence_statement_I}. 
The mild solution property \eqref{eq:iterate_representation-mild} is 
automatic once we note that $u$ is a fixed point of the 
right-hand side of \eqref{eq:iterate_representation}.

\begin{rem} 
While the type of convergence in  \eqref{eq:convergence_statement} is 
enough for our existence result, we also point out that in 
view of \eqref{eq:convergence_statement_II}, it is 
easily seen that $ \underset{n\rightarrow \infty} {\lim}
\sup_{0\le t\le T}E \Big[ ||u(t)-u^n(t)||_p\Big] = 0$ for $p =2,3,\ldots$.
\end{rem}

In view of Lemma \ref{lem:convergence_proof}, we pass to 
the limit $n\rightarrow \infty$ in 
Lemma \ref{lem:weak-mild}. The result is 
\begin{lem}\label{lem:weak-sol}
For each $\varphi\in C_c^\infty (\R^d)$,
\begin{align*}
   &\langle u(t), \varphi\rangle - \langle u(0), \varphi\rangle \\
   & \quad 
   = \int_{s=0}^t  \langle F(u(s)), \nabla \varphi\rangle ds 
   + \int_{r=0}^t \int_{|z|> 0} \int_{\rd_x} 
   \eta(x, u(r,x); z)\varphi(x)\,dx\,\tilde{N}(dz, dr)
   + \eps \int_{r=0}^t \langle\Delta\varphi, u(r)\rangle dr,
\end{align*}
almost surely, for almost every $t$.
\end{lem}
\end{proof} 

\begin{lem}\label{lem:classical-solution} 
Suppose that $E\big[ e_{2\ell}(u_0)\big] < \infty$ 
for $2\ell \ge [\frac d2]+ 3$, and let $u = u(t)$ be the 
limit process given by Lemma \ref{lem:convergence_proof}. 
Then $u=u(t)\in L_{\text{loc}}^\infty\big([0,\infty);L^{2}(\R^d) \big)$ and 
it is an $\mathcal{F}_{t}$-predictable (and c\'{a}dl\'{a}g) process that satisfies
\begin{itemize}
	 \item[(1)] $e_{2\ell}\big(u(t)\big) < \infty$, for all $t> 0$.
	 \item[(2)] $\partial_{ij} u = \partial_{x_i\,x_j}u(t,\cdot) \in\, C(\R^d)$ 
	 for all $i, j = 1,\ldots,d.$
\end{itemize}
In other words, the SPDE \eqref{eq:levy_stochconservation_laws-viscous} 
holds in the classical sense, i.e., \eqref{eq:levy_stochconservation_laws-viscous} 
is satisfied as an one dimensional L\'{e}vy driven SDE for every fixed $x$. 
\end{lem}

\begin{proof}
The proof of (1) is immediate from 
Lemma \ref{lem:convergence_proof}. 
The proof of (2) is also immediate if we apply 
the Sobolev embedding \cite{evansweak} along with (1). 
\end{proof} 

\subsection{A priori estimates for $\{u_\eps(t,x)\}_{\eps>0}$} 
We need to approximate the functions $u_0(x), \eta$ 
and $F$ from \ref{A1}-\ref{A4} by appropriate functions 
satisfying the assumptions \ref{B1}-\ref{B4}
Let $J\in C_c^\infty(\R)$ be a one dimensional 
mollifier and $\varphi\in C_c^\infty(\R)$ 
be a cut-off function such that 
\begin{align*}
  \varphi(r)=\begin{cases}
               0\quad\text{for}\quad |r| \ge 2\\
               1 \quad \text{for}\quad |r| \le 1.
             \end{cases}
\end{align*} 
For $\eps > 0$, define the approximations 
$F_\eps$, $\eta_\eps(x,u; z)$, $u^{\eps}_0(x)$ as follows:
\begin{align*}
F_\eps(r) &= \varphi(\eps |r|^2) F(r)*J_\eps(r),\\
 \eta_\eps(x,u;z)& = \int_{\rd_y}
 \int_{\R_v} \Big(\prod_{k=1}^{d} 
 J_\eps\big(x_k-y_k\big) J_\eps(u-v)\Big)
\varphi(\eps(|y|^2+|v|^2))\eta(y,v; z) \,dv\,dy,\\
u_0^\eps(x) &= \int_{\rd_y} 
\Big(\prod_{k=1}^{d} J_\eps\big(x_k-y_k\big) u_0(y) 
\varphi(\eps |y|^2)\Big) \,dy.
\end{align*}   
It follows from direct computations that 
\begin{align}
 |F_\eps(r)-F(r)| &\le C\eps(1+|r|^{2 p_0})
 \quad \text{for some}\quad p_0 \in \mathbb{N}, \notag\\
 |\eta_\eps(x,u;z) -\eta(x,u;z)| &\le C\eps(1+|x|+|u|) (1\wedge |z|)
 \label{eq:regu_error-eta}.
\end{align} 
Obviously, $u^\eps_{0}(\cdot)\in C_c^\infty(\R^d)$ 
and for $p \ge 2$
\begin{align}\notag
\sup_{\eps > 0} E\Big[ ||u^\eps_0(\cdot)||_p^p\Big] <  + \infty.
\end{align} 
Clearly, the functions $F_\eps$ and $\eta_\eps$ 
depend on $\eps$ and satisfy the regularity 
assumptions \ref{B1}-\ref{B3} 
Furthermore, we also have following facts: 
\begin{Assumptions}
    \item[(C.1)] $F_\eps$ satisfies same conditions as $F$.
    \item[(C.2)] $\sup_{\eps > 0} |\eta_\eps(x,u;z)| 
    \le g(x)(1+|u|)(1\wedge |z|)$ 
    where $g\in L^\infty(\R^d)\cap L^2(\R^d)$.
    \item[(C.3)] $ \sup_{\eps > 0} E\Big[||u^\eps_0||_p^p 
    + ||u^\eps_0||_2^p \Big] < \infty$, for $p =1,2,\ldots$.
\end{Assumptions}

We now focus on the equation
\begin{align}
du_\eps(t,x) + \mbox{div}_x F_\eps(u_\eps(t,x)) \,dt 
=  \int_{|z|> 0} \eta_\eps(x,u_\eps(t,x);z)\, \tilde{N}(dz,\,dt)
+\eps \Delta_{xx} u_\eps(t,x) \,dt, ~ t>0, ~ x\in \R^d,
\label{eq:levy_stochconservation_laws-viscous-new}
\end{align} 
with initial condition $u_\eps(0,x)= u_0^\eps(x)$. 
Clearly, by Lemma \ref{lem:classical-solution}, this problem 
possesses a unique strong solution $u_{\eps}(t)$.

\begin{lem}\label{lem:uniform estimates}
For even positive integers, $p=2, 4, 6,\ldots$,
$$
\sup_{\eps > 0} \sup_{0\le t\le T} 
E\Big[ ||u_\eps (t,\cdot)||_p^p\Big] < \infty.
$$
\end{lem}

\begin{proof}
We already know that $ \sup_{0\le t\le T} 
E\Big[ ||u_\eps (t,\cdot)||_p^p\Big] < \infty$ for every $p \ge 2$ and $T > 0$. 
Let $\beta(u)= \frac{1}{p}|u|^p$, and apply the It\^{o}-L\'{e}vy formula 
and integrate over the spatial variable $x$:
\begin{align*}
   &E \Big [||u_\eps(t)||_p^p\Big]- E \Big [||u_\eps(0)||_p^p\Big]\\
  & \quad \le 
   p(p-1) \int_{s=0}^t E\Big[\int_{\rd_x}\int_{|z| > 0}
   \int_{\lambda=0}^1 |u_\eps(s,x)+ \lambda \eta_\eps(x, u_\eps(s,x); z)|^{p-2}
   \eta_\eps^2(x,u_\eps(s,x); z)\,d\lambda\,m(dz) \,dx \Big] \,ds.
\end{align*} 
We now use (C.1)-(C.3) and apply the Gronwall's inequality to 
arrive the conclusion
\begin{align*}
\sup_{0\le t\le T} E\Big[ ||u_\eps (t,\cdot)||_p^p\Big] 
\le C_T \sup_{\eps > 0} E \Big[||u_\eps(0)||_p^p\Big],
\end{align*} 
thereby proving the claim.
\end{proof}

\begin{lem}\label{lem:gradient-estimate}
For each $p = 1, 2,\ldots$,
\begin{align*}
	\sup_{\eps > 0} E\Big[ \big|\eps \int_{s=0}^t 
	\int_{\R_x^d} |\nabla_x u_\eps(s,x)|^2 \,dx\,ds\big|^p\Big]< \infty,
\end{align*} 
for all $t >0$.
\end{lem}

\begin{proof}
Taking $p= 2$ in Lemma \ref{lem:uniform estimates} gives
\begin{align*}
& ||u_\eps(t)||^2_2- ||u_\eps(0)||_2^2
\\ & \quad 
=  \int_{s=0}^t \int_{\rd_x} \int_{|z| > 0} 
\eta_\eps^2(x,u_\eps(s,x); z)\,m(dz) \,dx\,ds 
-2\eps \int_{s=0}^t \int_{\R^d} |\nabla_x u_{\eps}(s,x)|^2 dx\, ds 
\\ & \quad\qquad
+ \int_{s=0}^t \int_{|z|> 0}\int_{\rd_x}\eta_\eps(x,u_\eps(s,x); z)(2u_\eps(s,x)
+ \eta_\eps(x, u_\eps(s,x); z))\,dx \,\tilde{N}(dz, ds).
\end{align*} 
We next apply the It\^{o}-L\'{e}vy formula 
to $||u_\eps(t)||_2^{2p}$ and use the moment estimates from 
Lemma \ref{lem:uniform estimates} with (C.3) to obtain 
\begin{align*}
E\Big[||u_\eps(t)||^{2p}_2\Big]
+E\Big[||u_\eps(0)||_2^{2p}\Big] < \infty.
\end{align*} 
We apply the moment estimates 
once again and conclude
\begin{align*}
& \sup_{\eps > 0} E\Big[\big|\int_{s=0}^t \int_{\rd_x} \int_{|z| > 0} 
\eta_\eps^2(x,u_\eps(s,x); z)\,m(dz)\,ds\,dx\big|^p\Big] < \infty.
\end{align*}
We can now simply use (C.2) along with uniform moment estimates in Lemma \ref{lem:uniform estimates}
and 
apply the BDG inequality, perhaps more 
than once, to conclude
\begin{align*}
&\sup_{\eps > 0} E\Big[\Big|\int_{s=0}^t \int_{|z|> 0} 
\int_{\rd_x}\eta_\eps(x, u_\eps(s,x); z)
\big(2u_\eps(s,x)+\eta_\eps(x,u_{\eps}(s,x);z)\big)\,dx\,
\tilde{N}(dz,ds) \Big |^p\Big] < \infty,
\end{align*} 
and hence the proof follows.
\end{proof}

There is a generalized version of the above lemma.
\begin{lem}
Let $ \beta \in C^2(\R)$ be a 
function with $ \beta,\beta',\beta''$ having at most
polynomial growth. Then
\begin{align}\label{eq:gendisest}
	\sup_{\eps >0} 
	E\left[\left|\eps\int_{t=0}^{T}\int_{\R_x^d} \beta''(u_\eps(t,x))
	|\grad_x u_\eps(t,x)|^2\,dx\,dt\right|^p\right] < \infty, 
	\qquad p=1,2,\ldots,\quad T>0.
\end{align}
\end{lem}
\begin{proof} 
Let $(\beta,\zeta )$ be an entropy-entropy 
flux pair. Let $ \psi_N\in C_c^2(\rd) $ 
be such that
$$
\psi_N(x)=
\begin{cases}
1,&\text{if $ |x|\leq N$},\notag \\
0,&\text{if $ |x|>N+1$.}
\end{cases}
$$
By the It\^{o}-L\'{e}vy formula, we have
\begin{align}
 & \langle \beta(u_\eps(T,.)),\psi_N\rangle
 -\langle \beta(u_0^\eps),\psi_N\rangle \notag 
 \\ & = \int_{r=0}^{T} \langle \zeta(u_\eps(r,.)),\grad_x \psi_N\rangle\,dr
 + \eps\int_{r=0}^{T}\Big(\langle \beta(u_\eps(r,.)),\Delta \psi_N\rangle
 - \langle \beta''(u_\eps(r,.))|\grad_x u_\eps(r,.)|^2 ,\psi_N\rangle\Big)\,dr \notag \\
 & + \int_{r=0}^T \int_{|z|>0}\langle\beta(u_\eps(r,.)
 +\eta_\eps(.,u_\eps(r,.);z))-\beta(u_\eps(r,.)),\psi_N\rangle\,\tilde{N}(dz,dr) \notag \\
 &
 + \int_{r=0}^T \int_{|z|>0}\langle\beta(u_\eps(r,.)
 +\eta_\eps(.,u_\eps(r,.);z))-\beta(u_\eps(r,.))
 -\eta_\eps(.,u_\eps(r,.);z)\beta'(u_\eps(r,.)),\psi_N\rangle
 \,m(dz)\,dr.\notag
\end{align}
Taking expectation and sending $N\goto \infty$ result in
\begin{align}
& E\Big[\Big|\int_{r=0}^T\int_{\R_x^d} \eps\beta''(u_\eps(r,x))
|\grad_xu_\eps(r,x)|^2\,dx\,dr\Big|^p\Big]\notag \\
& \quad 
\leq ||\beta(u_{\eps}(T,\cdot))||_1^p 
+ ||\beta(u_0^{\eps}(\cdot))||_1^p\notag\\
& \quad \quad
+ C_1 E\Big[\Big| \int_{r=0}^T \int_{|z|>0}\int_{\rd_x}\Big(\beta(u_\eps(r,x)
 +\eta_\eps(x,u_\eps(r,x);z))-\beta(u_\eps(r,x))\Big)
 \,dx \,\tilde{N}(dz,dr)\Big|^p\Big] \label{eq:sec} \\
& \quad\quad 
+ C_2E \Big[\Big| \int_{r=0}^T \int_{|z|>0}\int_{\rd_x}\Big(\beta(u_\eps(r,x)
+\eta_\eps(x,u_\eps(r,x);z))-\beta(u_\eps(r,x)) \notag\\
& \hspace{5.5cm}
-\eta_\eps(x,u_\eps(r,x);z)\beta'(u_\eps(r,x))\Big)
\,dx \,m(dz)\,dr\Big|^p\Big].\label{eq:last}
\end{align}
Since $|\eta_\eps(x,u;z)|\leq g(x)(1+|u|)(|z|\wedge 1)$, for 
some $g\in L^2(\rd)\cap L^{\infty}(\rd)$, and $\beta, \beta',\beta''$ 
have at most polynomial growth and 
$\sup_{\eps>0}\sup_{0\leq t\leq T} 
E \Big[||u_\eps(t,\cdot)||_p^p\Big] < \infty$, the 
term \eqref{eq:last} is finite.

Next, we want to estimate \eqref{eq:sec}. 
Using the BDG inequality we obtain, for any $ p\geq 2$,
\begin{align}
&\mathcal{I}(\eps)\notag\\
&:=E \Big[\Big| \int_{r=0}^T \int_{|z|>0}\int_{\rd_x}
\Big(\beta(u_\eps(r,x)+\eta_\eps(x,u_\eps(r,x);z))-\beta(u_\eps(r,x))\Big)
\,dx\,\tilde{N}(dz,dr)\Big|^p\Big]\notag\\
& = E\Big[\Big| \int_{r=0}^T \int_{|z|>0}\int_{\rd_x}
\int_{\lambda=0}^1\beta'(u_\eps(r,x)+\lambda\eta_\eps(x,u_\eps(r,x);z))
\eta_\eps(x,u_\eps(r,x);z) \,d\lambda\,dx\,\tilde{N}(dz,dr)\Big|^p\Big]\notag \\
& \leq C E\Big[\Big| \int_{r=0}^T \int_{|z|>0}\int_{\rd_x}
\int_{\lambda=0}^1\beta'^2(u_\eps(r,x)
+\lambda\eta_\eps(x,u_\eps(r,x);z))\eta_\eps^2(x,u_\eps(r,x);z)
\,d\lambda\,dx\,N(dz,dr)\Big|^{\frac{p} {2}}\Big].\notag
\end{align}
Since $\tilde{N}(dz,dr) = N(dz,dr)-m(dz)\,dr$, 
\begin{align}
\mathcal{I}(\eps)\notag\leq &  E\Big[\Big| \int_{r=0}^T \int_{|z|>0}
\int_{\rd_x}\int_{\lambda=0}^1\beta'^2(u_\eps(r,x)
+\lambda\eta_\eps(x,u_\eps(r,x);z))\\
&\hspace{6cm}\times\eta_\eps^2(x,u_\eps(r,x);z)
\,d\lambda\,dx
\tilde{N}(dz,dr)\Big|^{\frac{p} {2}}\Big]+ C. 
\notag 
\end{align}
Using the BDG inequality, perhaps repeatedly, we see 
that $\sup_{\eps>0}\mathcal{ I}(\eps) < \infty $. 

Finally, assuming $\beta(u) = C |u|^{2(\ell+1)}$ for 
some $\ell \in \mathbb{N}$, the 
above estimates imply
\begin{align}\label{eq:half-gradient-moment}
\sup_{\eps >0} E\left[\left|\eps\int_{t=0}^{T}
\int_{\rd_x} (u_{\eps}(t,x))^{2\ell}|\grad_x
u_\eps(t,x)|^2\,dx\,dt\right|^p\right] < \infty, 
\qquad p=1,2,\ldots,\quad T>0.
\end{align}
For general $\beta$ with polynomially growing 
derivatives, there is $\ell \in \mathbb{N}$ such that 
$|\beta^{\prime\prime}(u)|\le C(1+u^{2\ell})$. We use this 
information along with \eqref{eq:half-gradient-moment} 
and Lemma \ref{lem:gradient-estimate} 
to conclude \eqref{eq:gendisest}.
\end{proof} 

The achieved results can be 
summarized into the following proposition:

\begin{prop} \label{prop-moment-estimate}
Suppose assumptions \ref{A1}-\ref{A4} hold and fix any $\eps>0$. 
Then there exists a unique $C^2(\R^d)$-valued predictable 
process $u_\eps(t,\cdot)$ which solves 
\eqref{eq:levy_stochconservation_laws-viscous} with 
initial data $u_\eps(0,x)= u_0^\eps(x)$. Moreover,
\begin{itemize}
 \item[1)]  For even positive integers $p=2,4,6,\ldots$,
 \begin{align}
  \sup_{\eps>0}\sup_{0\le t\le T} 
  E\Big[ ||u_\eps(t,\cdot)||_p^p\Big] < + \infty. \label{uni:moment-esti}
 \end{align}
\item[2)]  For $\phi \in C^2(\R)$ with $ \phi,\phi^\prime, 
\phi^{\prime\prime}$ having at most polynomial growth,
\begin{align}
 \sup_{\eps>0}E\Big[\Big|\eps \int_{t=0}^T \int_{\R_x^d} 
 \phi^{\prime\prime}(u_\eps(t,x))|\grad_x u_\eps(t,x)|^2 
 \,dx\,dt\Big|^p\Big] < \infty,
 \quad p=1,2,\ldots,\quad T>0.
 \label{gradient-esti}
\end{align}
\end{itemize}
\end{prop}

\section{Existence of generalized entropy solution}\label{sec:existence-of-gen-solution}
The proof of existence depends largely on the 
appropriate compactness of the family $\{u_{\eps}(t,x)\}_{\eps > 0}$. 
The moment estimates \eqref{uni:moment-esti} only guarantee 
weak compactness, which is inadequate in view of 
the nonlinearities in the equation. Drawing inspiration from 
deterministic conservation laws, we look for compactness in 
the space of Young measures. We also mention here that 
similar strategies have been adopted by 
Bauzet, Vallet, and Wittbold \cite{BaVaWit} in 
the context of pure diffusion driven conservation laws. 
Before we proceed further, let us  define the 
Young measures and the notion of narrow convergence. 
We refer to \cite{dafermos,evansweak} for more on the topic of 
Young measures in deterministic settings and 
to \cite{Balder} for the stochastic version of the theory.

 \subsection{A few facts about Young measures}
Let $\big(\Theta, \Sigma, \mu\big)$ be a $\sigma$-finite 
measure space and $\mathcal{P}(\R)$ be 
the space of probability measures on $\R$.
  
\begin{defi}[Young measure]
 A Young measure from $\Theta$ into $\R$ is 
a map $\nu \mapsto \mathcal{P}(\R)$ such that $\nu(\cdot): \theta\mapsto \nu(\theta)(B)$ 
is $\Sigma$-measurable for every Borel subset $B$ of $\R$. The set of all 
Young measures from $\Theta$ into $\R$ is denoted 
by $\mathcal{R}\big(\Theta, \Sigma, \mu \big)$ or simply by  $\mathcal{R}$.
\end{defi}

\begin{defi}[narrow convergence]
A sequence of Young measures $\{\nu_n\}_n$ in $\mathcal{R}$ is 
said to converge $narrowly$ to $\nu_0$ iff for every  $A\in \Sigma$ 
and $h\in C_b(\R)$,
 $$  
 \lim_{n\rightarrow \infty} \int_A
 \Big[\int_{\R_\xi} h(\xi)\nu_n(\theta)(d\xi)\Big]\mu(d\theta) 
 = \int_A \Big[\int_{\R_\xi} h(\xi)\nu_0(\theta)(d\xi)\Big]\,\mu(d \theta). 
$$
\end{defi}

\begin{rem}
Young measures can be viewed as a parametrized family of probability 
measures where the parametrization is measurable. 
Clearly, if $u(\theta)$ is a real-valued measurable 
function on $\big(\Theta, \Sigma, \mu\big)$, 
then $\nu(\theta) = \delta(\xi-u(\theta)) $ 
defines a Young measure on $\Theta$. In other words, with 
an appropriate choice of $\big(\Theta, \Sigma, \mu\big)$, the 
family $\{u_\eps(t,x)\}_{\eps> 0}$ can be thought of as 
a family of Young measures and we are interested 
in extracting a subsequence which converges $narrowly$ 
in $\mathcal{R}$. This requires setting up 
suitable tightness criterion. 
\end{rem}

\begin{defi}[tightness]
A family of Young measures $\{\nu_n\}_n$ in $\mathcal{R}$ is 
called tight if there exists an inf-compact 
integrand $h$ on $\Theta \times \R$ such that 
$$
\sup_{n} \int_{\Theta}\Big[\int_{\R_\xi} h(\theta ,\xi)
\nu_n(\theta)(d\xi)\Big]\,\mu(d\theta)< \infty.
$$ 
\end{defi}

\begin{rem}
Without getting into details about the whole class 
of inf-compact functions, we point out 
that $h(\theta, \xi) = \xi^2$ is one such example. 
With this choice of $h$ and an 
appropriate choice of $\big(\Theta, \Sigma, \mu\big)$, 
by \eqref{uni:moment-esti} the family $\{u_\eps(t,x)\}_{\eps> 0}$ is tight 
when viewed as a family of Young measures.  
\end{rem}

The tightness condition enables us to 
extract a subsequence from 
a tight family and we have the following version 
of Prohorov's theorem to this end, a 
detailed proof which  could be found in \cite{Balder}.
 
 \begin{thm}[Prohorov's theorem] \label{thm:prohorov}
(1) Let $\big(\Theta, \Sigma, \mu\big)$ be a finite 
measure space and $\{\nu_n\}_n $ be a tight family of Young 
measures in $\mathcal{R}$. Then there exists a 
subsequence $\{\nu_{n^\prime}\}$ of $\{\nu_n\}_n $ and 
$\nu_0\in \mathcal{R}$ such that $\{\nu_{n^\prime}\}$ 
converges $narrowly$ to $\nu_0$. 
 
(2) Moreover, if $\nu_n = \delta_{f_n(\theta)}(\xi)$ and given a 
Caratheodory function $h(\theta, \xi)$  on $\Theta\times \R$, 
if $h(\theta, f_{n^\prime}(\theta))$ is uniformly integrable, then
$$
\lim_{n^\prime\rightarrow \infty} \int_{\Theta} 
h(\theta, f_{n^\prime}(\theta))\, \mu(d\theta) 
= \int_\Theta\Big[\int_{\R_\xi} h(\theta,\xi)
\, \nu_0(\theta)(d\xi)\Big]\, \mu(d \theta).
$$
\end{thm}

\subsection{Extraction of an inviscid Young measure limit}
The predictable $\sigma$-field of $\Omega\times(0,T)$ with 
respect to $\{\mathcal{F}_t\}$ is 
denoted by $\mathcal{P}_T$, and we set 
\begin{align*}
	\Theta = \Omega\times (0,T)\times \R^d,\quad \Sigma 
	= \mathcal{P}_T \times \mathcal{L}(\R^d)\quad \text{and} 
	\quad \mu= P\otimes \lambda_t\otimes \lambda_x,
\end{align*} 
where $\lambda_t$ and $\lambda_x$ are respectively the 
Lebesgue measures on $(0,T)$ and $\R^d$. 
Moreover, for $M\in \mathbb{N}$, let 
\begin{align*}
\Theta_M = \Omega\times (0,T)\times B_M,
\quad \Sigma_M = \mathcal{P}_T \times \mathcal{L}(B_M)
\quad \text{and} \quad \mu_M= \mu\big|_{\Theta_M},
\end{align*}
where $B_M$ is the ball of radius $M$ around zero in $\R^d$ 
and $\mathcal{L}(B_M)$ is the Lebesgue sigma algebra on $B_M$.  
Clearly $(\Theta_M, \Sigma_M, \mu_M)$ is a finite measure 
space and $\{u_\eps(\omega; t,x)\}_{\eps >0}$ is a 
tight family of Young measures in 
$\mathcal{R}(\Theta_M, \Sigma_M, \mu_M)$. 
Therefore by Theorem \ref{thm:prohorov} there exists 
a subsequence $\eps_n \goto 0$ and $\nu^M \in 
\mathcal{R}(\Theta_M, \Sigma_M, \mu_M)$ such 
that  $\{u_{\eps_n}(\omega; t,x)\}$ converges 
narrowly to $\nu^M$.

Furthermore, for $\bar{M}>M$, the sequence 
$\{u_{\eps_n}(\omega; t,x)\}$ is tight in $\mathcal{R}(\Theta_{\bar{M}}, 
\Sigma_{\bar{M}}, \mu_{\bar{M}})$, and hence admits a further subsequence, 
say $\{u_{\eps_{n^\prime}}(\omega; t,x)\}$, 
and $\nu^{\bar{M}}\in \mathcal{R}(\Theta_{\bar{M}}, \Sigma_{\bar{M}}, \mu_{\bar{M}})$ 
such that $\{u_{\eps_{n^\prime}}(\omega; t,x)\}$ 
converges narrowly to $\nu^{\bar{M}}$.  
We now invoke diagonalization and conclude that 
there exist a subsequence  $\{u_{\eps_{n^\prime}}(\omega; t,x)\}$ 
with $\eps_n\goto 0$ and a Young mesures $\nu^{M}\in
\mathcal{R}(\Theta_{M}, \Sigma_{M}, \mu_M)$, $M =1,2, 3,\ldots$, such 
that $\{u_{\eps_n}(\omega; t,x)\}$ converges 
narrowly to $\nu^M$ in $\mathcal{R}(\Theta_{M}, \Sigma_{M}, \mu_M)$ 
for every $M=1,2,\ldots$. It is also trivial to prove that 
$$
\text{if} ~ \bar{M} > M~\text{then}~\nu^M
=\nu^{\bar{M}}~\text{on}~(\Theta_{M}, \Sigma_{M}, \mu).
$$ 
Now we define
\begin{align}\label{eq:Young-extracted}
\nu(\theta) = \nu^M(\theta)~\text{if}~\,\theta \in \Theta_M. 
\end{align}
Clearly, $\nu$ is well defined and $\nu$ is a 
Young measure belonging to $\mathcal{R}(\Theta, \Sigma, \mu) $. 

We summarize the findings in a next lemma.
  
\begin{lem}\label{conv-young-measure}
Let  $\{u_\eps(t,x)\}_{\eps> 0}$ be a sequence of 
$L^p(\R^d)$-valued predictable processes 
such that \eqref{uni:moment-esti} holds. 
Then there exists a subsequence $\{\eps_n\}$ with $\eps_n\goto 0$ and 
a Young measure $\nu\in \mathcal{R}(\Theta, \Sigma, \mu) $ 
such that if  $h(\theta,\xi)$ is a Caratheodory function 
on $\Theta\times \R$ such that $\mbox{supp}(h)\subset \Theta_M\times \R$ 
for some $M \in \mathbb{N}$ and $\{h(\theta, u_{\eps_n}(\theta)\}_n$ 
(where $\theta\equiv (\omega; t, x)$) is uniformly integrable, then 
$$
\lim_{\eps_n\rightarrow 0} \int_{\Theta}
h(\theta, u_{\eps_n}(\theta))\, \mu(d\theta) 
= \int_\Theta\Big[\int_{\R_\xi} h(\theta, \xi)\, \nu(\theta)(d\xi)\Big]
\, \mu(d \theta).
$$
\end{lem}
 
\begin{proof}
The extraction of a subsequence is done as described above 
and $\nu$ is defined in \eqref{eq:Young-extracted}. 
Note that if $M\in \mathbb{N}$ such 
that $\mbox{supp}(h)\subset \Theta_M\times \R$, then 
$$
\int_{\Theta}h(\theta, u_{\eps_n}(\theta))\, \mu(d\theta)
=  \int_{\Theta_M}h(\theta, u_{\eps_n}(\theta))\, \mu_M(d\theta)
$$
and
$$
\int_\Theta
\Big[\int_{\R_\xi} h(\theta, \xi)\, \nu(\theta)(d\xi)\Big]\, \mu(d \theta)
=\int_{\Theta_M}\Big[\int_{\R_\xi} 
h(\theta, \xi)\, \nu^M(\theta)(d\xi)\Big]\, \mu_M(d \theta),
$$
and the convergence follows from Theorem \ref{thm:prohorov}.
\end{proof}

\subsection{Construction of a generalized entropy solution}
With the Young measure valued limit 
$\nu$ of $\{u_\eps(t,x)\}_{\eps > 0}$ (upto a subsequence) at 
hand, we follow the standard recipe of 
Panov \cite{panov} (and its adaptation to a 
stochastic case \cite{BaVaWit}) to turn it into a generalized
entropy solution. Define the real valued 
function $u(\theta, \lambda)$ by
\begin{align}
	u(\theta,\lambda)=\inf\Big\{ c\in \R: \nu(\theta)
	\Big((-\infty,c)\Big)>\lambda\Big\}, \qquad 
	\text{for $\lambda \in (0,1)$ and $\theta\in \Theta$}.
	\notag\end{align} 
\begin{lem}\label{lem:measure-conversion}
For fixed $\theta\in \Theta$, the function $u(\theta,\cdot)$ 
is non-decreasing and right-continuous on $(0,1)$.
Moreover, if $h(\theta,\xi)$ is a nonnegative 
Caratheodory function on $\Theta\times \R$, then 
$$
\int_{\Theta}\Big[\int_{\R_\xi} h(\theta,\xi)\, \nu(\theta)(d\xi)\Big]\, \mu(\,d\theta)
= \int_\Theta\int_{\lambda=0}^1 
h(\theta, u(\theta,\lambda))\,d\lambda\, \mu(\,d\theta).
$$
\end{lem}

\begin{proof}
The proof is classical, and we 
refer to [\cite{panov} Lemma 3.1 ] for the details. 
\end{proof}
 
Any prospective generalized entropy solution has to be predictable. 
The presence of L\'{e}vy noise makes this condition 
indispensable. The next lemma affirms 
that condition for $u(\omega; t,x, \lambda)$. 
\begin{lem}\label{lem:measurability_young_solution}
$u$ is $\mathcal{P}_T\times 
\mathcal{L}(\R^d\times (0,1))$ measurable.
\end{lem}

\begin{proof} 
We establish that $u$ satisfies the basic condition of measurability. 
Let $\sigma \in \R$ and $ E_\sigma=\{ (\theta,\lambda):u(\theta,\lambda) < \sigma\}$. 
We want to show that $E_\sigma\in  \mathcal{P}_T\times \mathcal{L}(\R^d\times (0,1))$. 
Let $H_\sigma=\{ (\theta,\lambda):\nu(\theta)\big((-\infty,\sigma)\big)>\lambda\}$. 
For $ (\theta,\lambda)\in E_\sigma $, it holds 
that $  u(\theta,\lambda)< \sigma$ i.e., there exits $c$ 
with $ u(\theta,\lambda)< c<\sigma$ such 
that $\nu (\theta)\big((-\infty,c)\big)>\lambda
$ and hence $\nu(\theta)\big((-\infty,\sigma)\big)>\lambda$, implying 
$E_\sigma\subset H_\sigma$. 
For the converse, let $(\theta,\lambda)\in H_\sigma $.
 
Note that the map  $\sigma \mapsto \nu(\theta)\big((-\infty,\sigma)\big)$ is 
left continuous and therefore $\nu(\theta)\big((-\infty,\sigma)\big)>\lambda$ 
implies that there exists $ c <\sigma $ such that
$\nu(\theta)\big((-\infty,c)\big)>\lambda$. Thus, by the 
definition of $ u$, $u(\theta,\lambda) < \sigma$ 
and hence $H_\sigma\subset E_\sigma$, implying 
$H_\sigma = E_\sigma$.  
Note that $\theta\mapsto \nu(\theta)\big((-\infty,\sigma)\big)$ 
is $\Sigma$-measurable, implying 
$H_\sigma\in  \mathcal{P}_T\times \mathcal{L}(\R^d\times (0,1)) $ 
for all $\sigma\in \R$.
\end{proof}

Let $\Gamma = \Omega\times [0,T] \times \R $, 
$\mathcal{G}= \mathcal{P}_T\times \mathcal{L}(\R)$ and 
$\varsigma = P\otimes \lambda_t \otimes m(dz)$. Then 
$L^2\big((\Gamma, \mathcal{G}, \varsigma); \R\big)$ consists 
of all square integrable predictable processes which are 
Lebesgue measurable functions of $z$-variable. In other words, if 
$\psi(t,z)\in L^2\big((\Gamma, \mathcal{G}, \varsigma); \R\big)$, then
$$
E \Big[\int_{t=0}^T \int_{\R_z} \psi^2(t,z)\, m(dz)\,dt\Big] <  + \infty.
$$
   
The space $L^2\big((\Gamma, \mathcal{G}, \varsigma); \R\big)$ 
represents the square integrable predictable integrands 
for It\^{o}-L\'{e}vy integrals with respect to the compensated 
Poisson random measure $\tilde{N}(dz,dt)$.  
Moreover, an It\^{o}-L\'{e}vy integral defines a linear 
operator from  $L^2\big((\Gamma, \mathcal{G}, \varsigma); \R\big)$ 
to $L^2\big((\Omega, \mathcal{F}_T); \R\big)$ and it preserves 
the norm, thanks to the It\^{o}-L\'{e}vy isometry. 
Furthermore, for any random variable $Y \in L^2\big((\Omega, \mathcal{F}_T); \R\big)$ 
we can invoke the {\it martingale representation theorem} for 
marked point processes and conclude that there 
exists $\psi\in L^2\big((\Gamma, \mathcal{G}, \varsigma); \R\big)$ such that 
\begin{align*}
    Y = \int_{t=0}^T \int_{\R_z} \psi(t,z)\, \tilde{N}(dz,dt) .
\end{align*} 
Hence, the It\^{o}-L\'{e}vy integral 
operator is an isometry from 
$L^2\big((\Gamma, \mathcal{G}, \varsigma); \R\big)$ 
onto $L^2\big((\Omega, \mathcal{F}_T); \R\big)$. 
To this end, note that any isometry between two 
Hilbert spaces preserves weak convergence. 
Therefore for any weakly converging sequence of integrands 
$\{\psi_n(t,z)\} \in L^2\big((\Gamma, \mathcal{G}, \varsigma); \R\big)$, the 
corresponding sequence of It\^{o}-L\^{e}vy integrals with respect 
to $\tilde{N}(dz,dt)$  will also converge weakly in 
$L^2\big((\Omega, \mathcal{F}_T); \R\big)$. Moreover, the 
weak limits are preserved under It\^{o}-L\'{e}vy integral operators. 
To see this, define
\begin{align}
 \chi_n(t,z) =\int_{\R_x^d}\Big( \beta\big( u_{\eps_n}(t,x) 
 + \eta(x, u_{\eps_n}(t,x); z) 
 - \beta\big( u_{\eps_n}(t,x))\Big)\phi(t,x) \,dx,\notag
\end{align} 
where $\beta$ is a smooth function with bounded derivatives 
and $\phi$ is a compactly supported smooth function on $\Pi_T$. 
Then clearly $\chi_n(t,z)\in L^2\big((\Gamma, \mathcal{G}, \varsigma); \R\big)$ 
and the sequence $\{\chi_n(t,z)\}_n$ is bounded in 
$L^2\big((\Gamma, \mathcal{G}, \varsigma); \R\big)$.  

We have the following lemma:
\begin{lem}\label{weak-conv-integrand}
The sequence  $\{\chi_n(t,z)\}_n$ is weakly convergent 
in $L^2\big((\Gamma, \mathcal{G}, \varsigma); \R\big)$ 
and the weak limit $\chi(t,z)$ is given by 
\begin{align}
\chi(t,z)= \int_{\R_x^d}\int_{\alpha=0}^1\Big( \beta\big( u(t,x,\alpha) 
+ \eta(x, u(t,x,\alpha); z) -  \beta\big( u(t,x, \alpha))
\Big)\phi(t,x) \,d\alpha\,dx. \notag
\end{align}
\end{lem}
 
\begin{proof}
Fix $h(t,z) \in L^2\big((\Gamma, \mathcal{G}, \varsigma); \R\big)$. Then  
\begin{align}
&E \Big[ \int_{t=0}^T \int_{\R_z} h(t,z) \chi_n(t,z)\,m(dz)\,dt\Big]\notag\\
\notag
& \qquad 
= \int_{\R_z} E \Big[\int_{t=0}^T\int_{\R_x^d} \Big( \beta\big( u_{\eps_n}(t,x) 
+ \eta(x, u_{\eps_n}(t,x); z)
-\beta\big( u_{\eps_n}(t,x))\Big)\phi(t,x) h(t,z)\, dx\, dt\Big] \,m(dz).
\end{align} 
For $m(\dz)$-almost every $z\in \R$, we 
apply Lemmas \ref{conv-young-measure} 
and \ref{lem:measure-conversion} and conclude that 
\begin{align*}
    & \lim_{n\rightarrow \infty}  E \Big[\int_{t=0}^T\int_{\R_x^d} 
    \Big( \beta\big( u_{\eps_n}(t,x) + \eta(x, u_{\eps_n}(t,x); z) 
    -\beta\big( u_{\eps_n}(t,x))\Big)\phi(t,x) h(t,z)\, dx\, dt\Big] \\
    & \quad 
    =   E \Big[\int_{t=0}^T\int_{\R_x^d}\int_{\alpha=0}^1\Big( \beta\big( u(t,x,\alpha) 
    + \eta(x, u(t,x,\alpha); z) -  \beta\big( u(t,x, \alpha))
    \Big)\phi(t,x)  h(t,z) \,d\alpha\,dx\,dt\Big]. 
\end{align*} 
We now invoke \ref{A4} and uniform 
moment estimates in order to apply the 
bounded convergence theorem, with the result that
\begin{align*}
&\lim_{n\rightarrow \infty} E  \Big[\int_{t=0}^T \int_{\R_z} h(t,z) \chi_n(t,z)\,m(dz)\,dt\Big]\notag\\
& \quad
=E \Big[ \int_{t=0}^T\ \int_{\R_z} \Big[\int_{\R_x^d}\int_{\alpha=0}^1 \Big( \beta\big( u(t,x,\alpha) 
+ \eta(x, u(t,x,\alpha); z) 
\\&\hspace{6cm}-  \beta\big( u(t,x, \alpha))\Big)\phi(t,x) h(t,z) 
\, d\alpha\, dx\Big] \,m(dz)\, dt\Big].
\end{align*} 
This completes the proof. 
\end{proof}

As a consequence of the discussion 
prior to Lemma \ref{weak-conv-integrand}, the 
It\^{o}-L\'{e}vy integrals $\int_{t=0}^T\int_{\R_z} \chi_n(t,z)\, \tilde{N}(dz,dt)$ 
converges weakly to $\int_{t=0}^T\int_{\R_z} \chi(t,z)\, \tilde{N}(dz,dt)$ 
in $L^2\big((\Omega, \mathcal{F}_T); \R\big)$. 
Hence, we have the following 

\begin{cor}\label{cor:martingale} 
For every $B\in \mathcal{F}_T$,
\begin{align*}
  &  \lim_{n\rightarrow \infty} E\Big[{\bf 1}_{B} \int_{t=0}^T 
  \int_{\R_z}\int_{\R_x^d} \Big( \beta\big( u_{\eps_n}(t,x) + \eta(x, u_{\eps_n}(t,x); z) 
  -  \beta\big( u_{\eps_n}(t,x))\Big)\phi(t,x)\, dx\, \tilde{N}(dz,dt) \Big]\\
    &
    = E\Big[{\bf 1}_{B} \int_{t=0}^T \int_{\R_z}\int_{\R_x^d}
    \int_{\alpha=0}^1 \Big( \beta\big( u(t,x,\alpha) 
    + \eta(x, u(t,x,\alpha); z) 
    -  \beta\big( u(t,x,\alpha))\Big)\phi(t,x)\,d\alpha\, dx\,\tilde{N}(dz,dt) \Big].
\end{align*}
\end{cor}

\begin{proof}
The proof is obvious in view of the above 
discussion as ${\bf 1}_{B}\in L^2\big((\Omega, \mathcal{F}_T); \R\big)$.
\end{proof}
At this point we fix a  nonnegative test 
function $ \psi\in C_c^\infty([0, \infty)\times \R^d)$,  $B\in \mathcal{F}_T$, 
and a convex entropy pair $(\beta,\zeta)$. 
Let $\zeta_\eps$ be the entropy flux based on $F_\eps$,  
and thus  $\zeta_\eps$ is approximating $\zeta$.  
We use the It\^{o}-L\'{e}vy formula to compute $\beta(u_\eps(t, x))$,  apply the 
product rule to $\psi(t,x)\beta(u_\eps(t,x))$, and then
integrate. The result is
\begin{align}
 0 & \le E\Big[{\bf 1}_{B}\int_{\R_x^d} \beta(u_0^{\eps_n}(x))\psi(0,x)\,dx\Big]
-{\eps_n} E\Big[{\bf 1}_{B}\int_{\Pi_T}\beta^\prime(u_{\eps_n}(t,x))
\grad u_{\eps_n}(t,x).\grad \psi(t,x)\,dx\,dt\Big] \notag \\
& \quad 
+E\Big[{\bf 1}_{B} \int_{\Pi_T} \Big(\beta(u_{\eps_n}(t,x))\partial_t \psi(t,x)
+\zeta_{\eps_n}(u_{\eps_n}(t,x))\cdot\grad\psi(t,x)\Big) \,dt\,dx\Big] \notag\\
& \quad + E\Big[{\bf 1}_{B}\int_{|z|>0}\int_{\Pi_T}\Big(\beta(u_{\eps_n}(t,x) 
+\eta_{\eps_n}(x,u_{\eps_n}(t,x);z))-\beta(u_{\eps_n}(t,x))\Big)
\psi(t,x)\,dx\,\tilde{N}(dz,dt)\Big]\notag\\
&\quad + E\Big[{\bf 1}_{B}\int_{|z|>0} \int_{\Pi_T}\Big(\beta(u_{\eps_n}(t,x) 
+\eta_{\eps_n}(x,u_{\eps_n}(t,x);z))-\beta(u_{\eps_n}(t,x))
 \notag \\
&\hspace{6cm} 
-\eta_{\eps_n}(x,u_{\eps_n}(t,x);z)\beta^\prime(u_{\eps_n}(t,x))\Big) 
\psi(t,x)\,dx\,m(dz)\,dt\Big].
\label{viscous-measure-inequality}
\end{align} 
With the help of uniform moment estimates and \eqref{eq:regu_error-eta}, it follows
from \eqref{viscous-measure-inequality} that
\begin{align}
0& \le E\Big[{\bf 1}_{B}\int_{\R_x^d} \beta(u_0^{\eps_n}(x))\psi(0,x)\,dx\Big]
-{\eps_n} E\Big[{\bf 1}_{B}\int_{\Pi_T}\beta^\prime(u_{\eps_n}(t,x))
\grad u_{\eps_n}(t,x).\grad \psi(t,x)\,dx\,dt\Big] \notag \\
& \quad +E\Big[{\bf 1}_{B} \int_{\Pi_T} \Big(\beta(u_{\eps_n}(t,x))\partial_t \psi(t,x)
+\zeta(u_{\eps_n}(t,x))\cdot\grad\psi(t,x)\Big) \,dt\,dx\Big] \notag\\
& \quad + E\Big[{\bf 1}_{B}\int_{t=0}^T\int_{|z|>0}\int_{\R_x^d}
\Big(\beta(u_{\eps_n}(t,x) +\eta(x,u_{\eps_n}(t,x);z))
-\beta(u_{\eps_n}(t,x))\Big)\psi(t,x)\,dx\,\tilde{N}(dz,dt)\Big]\notag\\
& \quad + E\Big[{\bf 1}_{B} \int_{t=0}^T\int_{|z|>0}\int_{\R_x^d}
\Big(\beta(u_{\eps_n}(t,x) +\eta(x,u_{\eps_n}(t,x);z))
\notag \\
&\hspace{3cm} 
 -\beta(u_{\eps_n}(t,x))-\eta(x,u_{\eps_n}(t,x);z)
 \beta^\prime(u_{\eps_n}(t,x))\Big) \psi(t,x)\,dx\,dt\,m(dz)\Big]
+ o(\eps_n).
\label{viscous-measure-inequality-2}
\end{align} 
We now pass to the limit $\eps_n \rightarrow 0$ in 
\eqref{viscous-measure-inequality-2}. 
Thanks to \eqref{gradient-esti}, 
\begin{align}\label{eq:passage-1} 
\lim_{\eps_n \rightarrow 0 } {\eps_n} E\Big[{\bf 1}_{B}\int_{\Pi_T}
\beta^\prime(u_{\eps_n}(t,x))
\grad u_{\eps_n}(t,x)\cdot \grad \psi(t,x)\,dx\,dt\Big] = 0.
\end{align} 
Moreover, it is straightforward to see that 
\begin{align}\label{eq:passage-2} 
\lim_{\eps_n \rightarrow 0 } E\Big[{\bf 1}_{B}\int_{\R_x^d} 
\beta(u_0^{\eps_n}(x))\psi(0,x)\,dx\Big]
= E\Big[{\bf 1}_{B}\int_{\R_x^d} \beta(u_0(x))\psi(0,x)\,dx\Big]. 
\end{align} 
We now recall that $ L^2\big(\Theta, \Sigma, \mu\big)$ is 
closed subspace of the larger space $L^2\big(0,T; L^2({(\Omega, \mathcal{F}_T)}, L^2(\R^d))\big)$, 
and hence weak convergence in $ L^2\big(\Theta, \Sigma, \mu\big)$ would 
imply weak convergence in $L^2\big(0,T; L^2({(\Omega, \mathcal{F}_T)}, L^2(\R^d))\big)$. 

In addition, for $B\in \mathcal{F}_T$, the functions 
${\bf 1}_B\partial_t \psi(t,x), {\bf 1}_B \partial_{x_i} \psi(t,x), {\bf 1}_B \psi(t,x)$ 
are all members of \newline$ L^2\big(0,T; L^2({(\Omega, \mathcal{F}_T)}, L^2(\R^d))\big)$, and 
hence by Lemmas \ref{conv-young-measure} 
and \ref{lem:measure-conversion}, we have 
\begin{align}
\notag
& \lim_{\eps_n \rightarrow 0 }E\Big[{\bf 1}_{B} \int_{\Pi_T} 
\Big(\beta(u_{\eps_n}(t,x))\partial_t \psi(t,x)+\zeta(u_{\eps_n}(t,x))\cdot 
\grad\psi(t,x)\Big) \,dt\,dx\Big] \\
\label{eq:passage-3} 
& \quad = E\Big[{\bf 1}_{B} \int_{\Pi_T}\int_{\alpha=0}^1 
\Big(\beta(u(t,x,\alpha))\partial_t \psi(t,x)
+\zeta(u(t,x,\alpha)).\grad\psi(t,x)\Big) \,d\alpha\,dt\,dx\Big],
\end{align}
and 
\begin{align}
\notag 
& \lim_{\eps_n \rightarrow 0 }E\Big[{\bf 1}_{B} 
\int_{\Pi_T}\int_{|z|>0}\Big(\beta(u_{\eps_n}(t,x) 
+\eta(x,u_{\eps_n}(t,x);z))-\beta(u_{\eps_n}(t,x))
 \notag \\
&\hspace{5cm} -\eta(x,u_{\eps_n}(t,x);z)\beta^\prime(u_{\eps_n}(t,x))
\Big) \psi(t,x)\,m(dz)\,dt\,dx\Big]\notag\\
& \quad = E\Big[{\bf 1}_{B} \int_{\Pi_T}\int_{\alpha=0}^1\int_{|z|>0}
\Big(\beta(u(t,x,\alpha) +\eta(x,u(t,x,\alpha);z))-\beta(u(t,x,\alpha))
 \notag \\
\label{eq:passage-4} 
&\hspace{5cm} 
-\eta(x,u(t,x,\alpha);z)\beta^\prime(u(t,x,\alpha))\Big) 
\psi(t,x)\,m(dz)\,d\alpha\,dt\,dx\Big].
\end{align} 

Now, combining \eqref{eq:passage-1}-\eqref{eq:passage-4} 
along with Corollary \ref{cor:martingale}, 
passing to the limit $\eps_n\downarrow 0$ in 
\eqref{viscous-measure-inequality-2} gives 
\begin{align}
0& \le E\Big[{\bf 1}_{B}\int_{\R_x^d} \beta(u_0 (x))\psi(0,x)\,dx\Big]\notag \\
& +E\Big[ {\bf 1}_{B}\int_{\Pi_T} \int_{\alpha=0}^1 \Big(\beta(u(t,x,\alpha))\partial_t \psi(t,x)
+\zeta(u(t,x,\alpha))\cdot\grad\psi(t,x)\Big)\,d\alpha \,dt\,dx\Big] \notag\\
& + E\Big[{\bf 1}_{B}\int_{\Pi_T}\int_{|z|>0} 
\int_{\alpha=0}^1\Big(\beta(u(t,x,\alpha) +\eta(x,u(t,x,\alpha);z))
-\beta(u(t,x,\alpha))\Big)\psi(t,x)\,d\alpha\,\tilde{N}(dz,dt)\,dx\Big]\notag\\
&+ E\Big[{\bf 1}_{B}\int_{\Pi_T}\int_{|z|>0} 
\int_{\alpha=0}^1\Big(\beta(u(t,x,\alpha) +\eta(x,u(t,x,\alpha);z))
-\beta(u(t,x,\alpha)) \notag \\
&\hspace{5cm} 
-\eta(x,u(t,x,\alpha);z)\beta^\prime(u(t,x,\alpha))\Big)   
\psi(t,x) d\alpha\,m(dz)\,dt\,dx\Big].
\label{viscous-measure_1-inequality}
\end{align} 

\begin{proof}[Proof of the Theorem \ref{thm:existenc}]
The predictability of $u(t,\cdot,\cdot)$ follows from 
Lemma \ref{lem:measurability_young_solution}, and the 
uniform moment estimate together with the 
classical Fatou lemma, which gives 
\begin{align}
\sup_{0\le t\le T} E\Big[|| u(t,\cdot,\cdot)||_p^p\Big] < \infty
\qquad \text{for $p=2,3,4$,\ldots.}
\notag.
\end{align}
For any $ 0\le \psi\in C_c^\infty([0, \infty)\times \R^d)$ and 
any convex entropy flux-pair $(\beta,\zeta)$, \eqref{viscous-measure_1-inequality} 
holds for each set $B\in \mathcal{F}_T$. Hence 
\begin{align}
&\int_{\R_x^d} \beta(u_0 (x))\psi(0,x)\,dx
+\int_{\Pi_T} \int_{\lambda=0}^1 \Big(\beta(u(t,x,\lambda))
\partial_t \psi(t,x)+\zeta(u(t,x,\lambda))\cdot\grad\psi(t,x)\Big)
\,d\lambda \, dt\,dx \notag\\
& \quad + \int_{\Pi_T}\int_{|z|>0} \int_{\lambda=0}^1\Big(\beta \big(u(t,x,\lambda) 
+\eta(x,u(t,x,\lambda);z)\big)-\beta(u(t,x,\lambda)) \notag \\
&\hspace{5cm} 
-\eta(x,u(t,x,\lambda);z)\beta^\prime(u(t,x,\lambda))\Big)   
\psi(t,x)\, d\lambda\,m(dz)\,dt\,dx \notag \\
& \quad + \int_{\Pi_T}\int_{|z|>0} \int_{\lambda=0}^1
\Big(\beta \big(u(t,x,\lambda) +\eta(x,u(t,x,\lambda);z)\big)
-\beta(u(t,x,\lambda))\Big)\psi(t,x)\,d\lambda\,\tilde{N}(dz,dt)\,dx\notag\\
& \ge 0 \quad \text{$P$-a.s.,}\notag 
\end{align} 
which completes the proof.
\end{proof}

\section{Uniqueness and existence of entropy solutions} \label{uniqueness}
A natural strategy for proving uniqueness in the presence of noise is to 
adapt the Kru{\v{z}}kov approach for deterministic equations. 
The main difficulty lies in ``doubling" the time variable, which gives rise to 
stochastic integrands that are anticipative and hence cannot be 
interpreted in the usual It\^{o} sense. One way to get around this problem 
seems to be through the vanishing viscosity regularization.  For conservation laws 
with Brownian white noise, there are two routes based on this 
strategy. The first one is by introducing the so called {\it strong entropy condition}
(see \cite{BisMaj,Chen:2012fk,nualart:2008}) and then 
showing that the vanishing viscosity limit obeys this condition.
The other one uses a more direct approach (see \cite{BaVaWit}) by 
comparing the entropy solution against the solution of the viscous problem
and subsequently sending the viscosity parameter to zero, relying 
on ``weak compactness" of the viscous approximations. In the 
presence of L\'{e}vy noise, the paths of the 
solution are discontinuous and the Feng-Nualart strategy of introducing a 
``strong entropy condition" has proven difficult to implement. 
However, as it will be detailed in the sequel, the approach
of directly comparing an entropy solution against that of a weakly converging 
sequence of viscous approximations is successful.  

Let $\rho$ and $\varrho$ be the standard nonnegative 
mollifiers on $\R$ and $\R^d$ respectively such that 
$\supp(\rho) \subset [-1,0]$ and $\supp(\varrho) = B_1(0)$.  
We define  $\rho_{\delta_0}(r) = \frac{1}{\delta_0}\rho(\frac{r}{\delta_0})$ 
and $\varrho_{\delta}(x) = \frac{1}{\delta^d}\varrho(\frac{x}{\delta})$, 
where $\delta$ and $\delta_0$ are two positive constants.  Given a  nonnegative test 
function $\psi\in C_c^{1,2}([0,\infty)\times \rd)$ and two 
positive constants $\delta$ and $ \delta_0 $, we define 
\begin{align}\notag
\phi_{\delta,\delta_0}(t,x, s,y) = \rho_{\delta_0}(t-s) 
\varrho_{\delta}(x-y) \psi(s,y).
\end{align}
Clearly $ \rho_{\delta_0}(t-s) \neq 0$ only 
if $s-\delta_0 \le t\le s$ and hence $\phi_{\delta,\delta_0}(t,x; s,y)= 0$ 
outside  $s-\delta_0 \le t\le s$. 
 
Let $v(t,x,\alpha)$ be a generalized entropy 
solution of \eqref{eq:levy_stochconservation_laws}. 
Moreover, let $\varsigma$ be the standard symmetric 
nonnegative mollifier on $\R$ with support in $[-1,1]$ 
and $\varsigma_l(r)= \frac{1}{l} \varsigma(\frac{r}{l})$ 
for $l > 0$. We use the generic $\beta$ for the 
functions $\beta_{\vartheta}$ introduced in Section \ref{technical}.
Given $k\in \R$, the function 
$\beta(\cdot-k)$ is a smooth convex function 
and $(\beta(\cdot-k), F^\beta(\cdot, k))$ is a 
convex entropy pair. 

We now write the entropy inequality 
for $v(t,x,\alpha)$, based on the 
entropy pair $(\beta(\cdot-k), F^\beta(\cdot, k))$, and 
then multiply by $\varsigma_l(u_\eps(s,y)-k)$, integrate with 
respect to $ s, y, k$ and take the expectation. The result is
\begin{align}
0\le  & E \Big[\int_{\Pi_T}\int_{\R_x^d}\int_{\R_k} \beta(v(0,x)-k)
\phi_{\delta,\delta_0}(0,x,s,y) \varsigma_l(u_\eps(s,y)-k)\,dk \,dx\,dy\,ds\Big] \notag \\
&  + E \Big[\int_{\Pi_T} \int_{\Pi_T} \int_{\alpha=0}^1 
\int_{\R_k} \beta(v(t,x,\alpha)-k)\partial_t \phi_{\delta,\delta_0}
\varsigma_l(u_\eps(s,y)-k)\,dk\, d\alpha \,dx\,dt\,dy\,ds \Big]\notag \\ 
& + E \Big[ \int_{\Pi_T}\int_{\R_k} \int_{t=0}^T\int_{|z|>0} 
\int_{\alpha=0}^1 \int_{\R_x^d}\Big(\beta \big(v(t,x,\alpha) 
+\eta(x,v(t,x,\alpha);z)-k\big)-\beta(v(t,x,\alpha)-k)\Big) \notag \\
& \hspace{5cm} 
\times \phi_{\delta,\delta_0}\,dx\,d\alpha \,\tilde{N}(dz,dt)
\varsigma_l(u_\eps(s,y)-k)\,dk \,dy\,ds \Big] \notag\\
& + E \Big[\int_{\Pi_T} \int_{t=0}^T\int_{|z|>0}\int_{\R_x^d} 
\int_{\R_k} \int_{\alpha=0}^1 \Big(\beta \big(v(t,x,\alpha) 
+\eta(x,v(t,x,\alpha);z)-k\big)-\beta(v(t,x,\alpha)-k) \notag \\
 & \hspace{4.5cm}-\eta(x,v(t,x,\alpha);z) \beta^{\prime}(v(t,x,\alpha)-k)\Big)
 \phi_{\delta,\delta_0}(t,x;s,y) \notag \\
&\hspace{6cm}\times \varsigma_l(u_\eps(s,y)-k)\,d\alpha\,dk\,dx\,m(dz)\,dt\,dy\,ds\Big]\notag \\
& + E \Big[\int_{\Pi_T}\int_{\Pi_T}\int_{\alpha=0}^1 \int_{\R_k} 
 F^\beta(v(t,x,\alpha),k) \cdot \grad_x \varrho_\delta(x-y)\,\psi(s,y)\,\rho_{\delta_0}(t-s)\notag \\
 &\hspace{4cm} \times \varsigma_l(u_\eps(s,y)-k)\,dk
\,d\alpha\,dx\,dt\,dy\,ds\Big] \notag \\
& \qquad 
=:  I_1 + I_2 + I_3 +I_4 + I_5.
\label{stochas_entropy_1}
\end{align}
 
We now apply the It\^{o}-L\'{e}vy formula 
to \eqref{eq:levy_stochconservation_laws-viscous-new}, giving
\begin{align}
 0\le  &\, E \Big[\int_{\Pi_T}\int_{\R_x^d} \int_{\alpha=0}^1\int_{\R_k} 
 \beta(u_\eps(0,y)-k)\phi_{\delta,\delta_0}(t,x,0,y) \varsigma_l(v(t,x,\alpha)-k)
 \,dk\,d\alpha \,dx\,dy\,dt\Big] \notag \\
 &  + E \Big[\int_{\Pi_T} \int_{\Pi_T} \int_{\alpha=0}^1 \int_{\R_k} 
 \beta(u_\eps(s,y)-k)\partial_s \phi_{\delta,\delta_0}
 \varsigma_l(v(t,x,\alpha)-k)\,dk\, d\alpha \,dy\,ds\,dx\,dt\Big] \notag \\ 
 & + E \Big[\int_{\Pi_T} \int_{s=0}^T\int_{|z|>0} \int_{\R_k} 
 \int_{\alpha=0}^1 \int_{\R_y^d}\Big(\beta \big(u_\eps(s,y) +\eta_\eps(y,u_\eps(s,y);z)-k\big)
 -\beta(u_\eps(s,y)-k)\Big) \notag \\
 & \hspace{6cm} 
 \times \phi_{\delta,\delta_0}
 \varsigma_l(v(t,x,\alpha)-k)\,dy\,d\alpha\,dk \,\tilde{N}(dz,ds)\,dx\,dt \Big]\notag\\
 & + E \Big[\int_{\Pi_T} \int_{s=0}^T\int_{|z|>0}\int_{\R_y^d} 
 \int_{\R_k}\int_{\alpha=0}^1 \Big(\beta \big(u_\eps(s,y) +\eta_\eps(y,(u_\eps(s,y);z)-k\big)
 -\beta(u_\eps(s,y)-k) \notag \\
  & \hspace{6cm}-\eta_\eps(y,u_\eps(s,y);z) \beta^{\prime}(u_\eps(s,y)-k)\Big)
  \phi_{\delta,\delta_0}(t,x;s,y) \notag \\
 &\hspace{6cm}
 \times \varsigma_l(v(t,x,\alpha)-k)\,d\alpha\,dk\,dy\,m(dz)\,ds\,dx\,dt\Big]\notag \\
 & + E\Big[\int_{\Pi_T}\int_{\Pi_T}\int_{\alpha=0}^1 \int_{\R_k}  
 F_\eps^\beta(u_\eps(s,y),k)\cdot \grad_y\varrho_\delta(x-y) \psi(s,y)
  \rho_{\delta_0}(t-s )  \notag \\
  & \hspace {6cm}\times \varsigma_l(v(t,x,\alpha)-k)\,dk
 \,d\alpha\,dx\,dt\,dy\,ds\Big] \notag \\
 & + E \Big[\int_{\Pi_T}\int_{\Pi_T}\int_{\alpha=0}^1 \int_{\R_k}  
 F_\eps^\beta(u_\eps(s,y),k) \cdot \grad_y \psi(s,y) \varrho_\delta(x-y) 
  \rho_{\delta_0}(t-s ) \notag \\
  & \hspace{6cm}\times \varsigma_l(v(t,x,\alpha)-k)\,dk\,d\alpha\,dx\,dt\,dy\,ds\Big] \notag \\
 & -\eps  E \Big[\int_{\Pi_T} \int_{\Pi_T} \int_{\alpha=0}^1 \int_{\R_k} 
 \beta^\prime(u_\eps(s,y)-k)\grad_y u_\eps(s,y) \cdot \grad_y  \phi_{\delta,\delta_0}
 \varsigma_l(v(t,x,\alpha)-k)\,dk\, d\alpha \,dy\,ds\,dx\,dt\Big],
\label{stochas_entropy_2}
\end{align} 
where $ F_\eps^\beta(a,b) = \int_a^b \beta^\prime(\sigma-b)F^\prime_\eps(\sigma)\,d\sigma$. 
It follows by direct computations that there is $p\in \mathbb{N}$ such that
$$
\big|F_\eps^\beta(a,b)- F^\beta(a, b) \big|\le C\eps \big(1+|a|^{2p}+|b|^{2p}\big). 
$$ 

In view of the uniform moment estimates \eqref{uni:moment-esti}, it follows from 
\eqref{stochas_entropy_2} that 
\begin{align}
0\le  & E \Big[\int_{\Pi_T}\int_{\R_x^d} \int_{\alpha=0}^1\int_{\R_k} 
\beta(u_\eps(0,y)-k)\phi_{\delta,\delta_0}(t,x,0,y) 
\varsigma_l(v(t,x,\alpha)-k)\,dk\,d\alpha \,dx\,dy\,dt\Big] \notag \\
 &  + E \Big[\int_{\Pi_T} \int_{\Pi_T} \int_{\alpha=0}^1 \int_{\R_k} 
 \beta(u_\eps(s,y)-k)\partial_s \phi_{\delta,\delta_0}
 \varsigma_l(v(t,x,\alpha)-k)\,dk\, d\alpha \,dy\,ds\,dx\,dt\Big]\notag \\ 
 & + E \Big[\int_{\Pi_T} \int_{s=0}^T\int_{|z|>0} \int_{\R_k} 
 \int_{\alpha=0}^1\int_{\R_y^d}\Big(\beta \big(u_\eps(s,y) +\eta_\eps(y,u_\eps(s,y);z)-k\big)
 -\beta(u_\eps(s,y)-k)\Big) \notag \\
 & \hspace{6cm} 
 \times \phi_{\delta,\delta_0}\varsigma_l(v(t,x,\alpha)-k)
 \,dy\,d\alpha \,dk\,\tilde{N}(dz,ds)\,dx\,dt \Big]\notag\\
 & + E \Big[\int_{\Pi_T} \int_{s=0}^T\int_{|z|>0}\int_{\R_y^d} \int_{\R_k}\int_{\alpha=0}^1 
 \Big(\beta \big(u_\eps(s,y) +\eta_{\eps}(y,(u_\eps(s,y);z)-k\big)
 -\beta(u_\eps(s,y)-k) \notag \\
 & \hspace{6cm}-\eta_{\eps}(y,u_\eps(s,y);z)  \beta^{\prime}(u_\eps(s,y)-k)\Big)
 \phi_{\delta,\delta_0}(t,x;s,y) \notag \\
 &\hspace{6cm}
 \times \varsigma_l(v(t,x,\alpha)-k)\,d\alpha\,dk\,dy\,m(dz)\,ds\,dx\,dt \Big]\notag \\
 & + E \Big[\int_{\Pi_T}\int_{\Pi_T}\int_{\alpha=0}^1 
 \int_{\R_k}  F^\beta(u_\eps(s,y),k)  \cdot \grad_y\varrho_\delta(x-y) \psi(s,y)
 \rho_{\delta_0}(t-s )  \notag \\
  & \hspace {6cm}
  \times \varsigma_l(v(t,x,\alpha)-k)\,dk\,d\alpha\,dx\,dt\,dy\,ds\Big] \notag \\
 & + E \Big[\int_{\Pi_T}\int_{\Pi_T}\int_{\alpha=0}^1\int_{\R_k}  
 F^\beta(u_\eps(s,y),k) \cdot \grad_y \psi(s,y) \varrho_\delta(x-y) 
  \rho_{\delta_0}(t-s ) \notag \\
  & \hspace{6cm}
  \times \varsigma_l(v(t,x,\alpha)-k)\,dk \,d\alpha\,dx\,dt\,dy\,ds\Big] \notag \\
  -& \eps  E\Big[ \int_{\Pi_T} \int_{\Pi_T} \int_{\alpha=0}^1 
  \int_{\R_k} \beta^\prime(u_\eps(s,y)-k)\grad_y u_\eps(s,y)\cdot \grad_y  \phi_{\delta,\delta_0}
 \varsigma_l(v(t,x,\alpha)-k)\,dk\, d\alpha \,dy\,ds\,dx\,dt\Big] \notag 
 \\ &  \hspace{4cm} +C(\delta, \beta,\psi)o(\eps) \notag \\ 
 & \qquad 
 =:  J_1 + J_2 + J_3 +J_4 + J_5 + J_6 + J_7
+C(\delta, \beta,\psi)o(\eps)
\label{stochas_entropy_3}
\end{align}
where $C(\delta, \beta, \psi)$ is a constant 
depending only the quantities in the parentheses.

We now add \eqref{stochas_entropy_1} and \eqref{stochas_entropy_3}, 
and compute limits with respect to the various parameters involved.   
\begin{lem}\label{stochastic_lemma_1}
It holds that 
\begin{align}
  I_1 + J_1   & \underset{\delta_0 \goto 0} \longrightarrow E \Big[\int_{\R_y^d}\int_{\R_x^d}\int_{\R_k} 
  \beta(v(0,x)-k)\psi(0,y)\varrho_{\delta} (x-y) \varsigma_l(u_\eps(0,y)-k)\,dk\,dx\,dy \Big]\notag\\
  &\underset{l \goto 0} \longrightarrow  E \Big[\int_{\R_y^d}\int_{\R_x^d}
  \beta(v(0,x)-u_\eps(0,y))\psi(0,y)\varrho_{\delta} (x-y)\,dx\,dy\Big]\notag\\
  &\underset{\eps \goto 0} \longrightarrow  E \Big[\int_{\R_y^d}\int_{\R_x^d} 
  \beta (v(0,x)-u(0,y))\psi(0,y)\varrho_{\delta} (x-y)\,dx \,dy\Big],\notag
\end{align}
and
\begin{align}
 \lim_{(\vartheta, \delta)\goto (0,0)} & E \Big[\int_{\R_y^d}
 \int_{\R_x^d} \beta_{\vartheta}(v(0,x)-u(0,y))
 \psi(0,y)\varrho_{\delta} (x-y)\,dx \,dy\Big] \notag \\
&= E \Big[\int_{\R_x^d} |v(0,x)-u(0,x)|\psi(0,x)\,dx\Big]. \notag
 \end{align}
\end{lem}
\begin{proof}
The first part of the proof is divided into three steps, and we note that $J_1 = 0$. 

{\textbf{Step 1:}}  In this step, we want 
to let $\delta_0 \goto 0$. For this, let
\begin{align}
 \mathcal{A}_1:=& E \Big[\int_{\Pi_T}\int_{\R_x^d}\int_{\R_k} \beta(v(0,x)-k)\psi(s,y)\varrho_\delta(x-y)\,\rho_{\delta_0}(-s) \varsigma_l(u_\eps(s,y)-k)\,dk \,dx\,dy\,ds\Big] \notag \\
&- E \Big[\int_{\R_y^d}\int_{\R_x^d}\int_{\R_k} 
  \beta(v(0,x)-k)\psi(0,y)\varrho_{\delta} (x-y) \varsigma_l(u_\eps(0,y)-k)\,dk\,dx\,dy\Big]\notag\\
  =&E\Big[\int_{\Pi_T}\int_{\R_y^d}\int_{\R_k} \beta(v(0,x)- u_\eps(s,y)+k)\Big(\psi(s,y)-\psi(0,y)\Big)
  \varrho_\delta(x-y)\notag \\
  &\hspace{6cm}\times \rho_{\delta_0}(-s) \varsigma_l(k)\,dk \,dx\,dy\,ds\Big] \notag \\
&+ E \Big[\int_{\Pi_T}\int_{\R_x^d}\int_{\R_k} 
  \Big(\beta(v(0,x)- u_\eps(s,y)+k) - \beta(v(0,x)- u_\eps(0,y)+k)\Big) \psi(0,y)\notag \\
  &\hspace{4cm} \times \varrho_{\delta} (x-y)
 \rho_{\delta_0}(-s) \varsigma_l(k)\,dk\,dx\,dy\, ds \Big]\notag.
 \end{align}
Since support $ \psi(s,\cdot)\subset K$, we have
\begin{align}
   \big|\mathcal{A}_1\big| & 
   \le ||\partial_t \psi||_{\infty} E \Big[ \int_{\Pi_T}\int_{\R_x^d}
   \int_{\R_k} \chi_K(y) \beta(v(0,x)- u_\eps(s,y)+k)
   \varrho_\delta(x-y)\notag \\
  &\hspace{6cm}\times s\rho_{\delta_0}(-s) \varsigma_l(k)
  \,dk \,dx\,dy\,ds\Big] \notag \\
  & \quad 
  + ||\beta^\prime ||_{\infty} E \Big[\int_{\Pi_T}\int_{\R_x^d}  
  |u_\eps(s,y)-u_\eps(0,y)|\psi(0,y)\varrho_{\delta} (x-y)
 \rho_{\delta_0}(-s) \,dx\,dy\, ds\Big]\notag \\
 & \le ||\partial_t \psi||_{\infty} \delta_0\,||\beta^\prime ||_{\infty} 
 E \Big[\int_{\Pi_T}\int_{\R_x^d}
  \chi_K(y) |v(0,x)- u_\eps(s,y)| 
   \rho_{\delta_0}(-s) \varrho_\delta(x-y) \, dx\,dy\,ds\Big]\notag \\ 
  & \quad 
  + ||\partial_t \psi||_{\infty} \delta_0\,||\beta^\prime ||_{\infty}\,l\, 
  E \Big[\int_{\R_y^d}\int_{\R_x^d} \chi_K(y)
  \varrho_\delta(x-y)\, dx\,dy \Big]\notag \\
  & \quad + ||\beta^\prime ||_{\infty} E \Big[\int_{\Pi_T}\int_{\R_x^d} 
   |u_\eps(s,y)-u_\eps(0,y)|\psi(0,y)\varrho_{\delta} (x-y)
 \rho_{\delta_0}(-s) \,dx\,dy\, ds\Big]\notag \\
 &\le ||\partial_t \psi||_{\infty} \delta_0\,||\beta^\prime ||_{\infty} C(\psi)
 \,\Big( E \big[|| v(0,\cdot,\cdot)||_2\big] +
 \sup_{\eps>0}\sup_{0\le s\le T} E \big[||u_\eps(s,\cdot)||_2\big] \Big) 
\notag \\
 & \quad 
 +  ||\partial_t \psi||_{\infty} \delta_0\,||\beta^\prime ||_{\infty}\,l\,
 C(\psi)+C ||\beta^\prime ||_{\infty} E\Big[ \frac{1}{\delta_0}
 \int_{s=0}^{\delta_0}\int_{\R^d}  |u_\eps(s,y)-u_\eps(0,y)|\psi(0,y)
 \,dy\, ds\Big].\notag 
\end{align}
Clearly, the results of  Lemma \ref{lem:initial-cond} continue to 
hold if we replace $u$ by $u_\eps$. Hence the last term 
vanishes as $\delta_0 \rightarrow 0$. 
Therefore $\mathcal{A}_1 \goto 0 $ as $\delta_0 \goto 0$.
 
\textbf{Step 2:} In this step,  we verify the 
passage to the limit as $l \goto 0$. To this end, let
\begin{align}
& \mathcal{A}_2\notag\\& :=   E \Big[\int_{\R_y^d}\int_{\R_x^d}\int_{\R_k} 
  \beta(v(0,x)-k)\psi(0,y)\varrho_{\delta} (x-y) 
  \varsigma_l(u_\eps(0,y)-k)\,dk\,dx\,dy\Big]\notag\\
  & \quad - E\Big[\int_{\R_y^d}\int_{\R_x^d}
  \beta(v(0,x)-u_\eps(0,y))\psi(0,y)\varrho_{\delta} (x-y)\,dx\,dy\Big]\notag\\
  & =  E\Big[\int_{\R_y^d}\int_{\R_x^d}\int_{\R_k} 
  \Big(\beta(v(0,x)+ k-u_\eps(0,y))-\beta(v(0,x)-u_\eps(0,y))\Big) 
  \psi(0,y) \varrho_{\delta} (x-y) \varsigma_l(k)\,dk\,dx\,dy\Big].\notag
\end{align}
Hence
\begin{align}
\big|\mathcal{A}_2\big| & \le ||\beta^\prime||_{\infty}  
E \Big[\int_{\R^d_y}\int_{\R^d_x}\int_{\R_k} k \psi(0,y)
\varrho_{\delta} (x-y) \varsigma_l(k)\,dk\,dx\,dy\Big]
\le C(\psi) ||\beta^\prime||_{\infty}\,l \notag \\
& \quad
 \goto 0 \quad \text{as}\quad l \goto 0\notag.
\end{align}

\textbf{Step 3:} In this step,  we verify the passage 
to the limit as $\eps \goto 0$. Let
\begin{align*}
 \mathcal{A}_3
 &:= E\Big[\int_{\R_y^d}\int_{\R_x^d} \beta(v(0,x)-u_\eps(0,y))\psi(0,y)\varrho_{\delta}(x-y)\,dx \,dy\Big]
\\ & \quad
- E \Big[\int_{\R_y^d}\int_{\R_x^d} \beta(v(0,x)-u(0,y))\psi(0,y)\varrho_{\delta}(x-y)\,dx \,dy\Big]\notag\\
& = E \Big[\int_{\R_y^d}\int_{\R_x^d}\Big( \beta(v(0,x)-u_\eps(0,y))- \beta(v(0,x)-u(0,y))\Big)
\psi(0,y)\varrho_{\delta}(x-y)\,dx \,dy\Big].\notag
\end{align*} 
Therefore,
\begin{align*}
\big|\mathcal{A}_3\big| & \le 
||\beta^\prime||_{\infty}\int_{\R_y^d} 
|u_\eps(0,y)-u(0,y)|\psi(0,y)\,dy
\rightarrow 0  \quad \text{as $\eps  \goto 0$.}
\end{align*}

For the second part of the lemma, consider 
\begin{align*}
\mathcal{A}_4(\vartheta,\delta) & 
:=  E\Big[\int_{\R_y^d}\int_{\R_x^d} \Big(\beta_{\vartheta}(v(0,x)-u(0,y))- |v(0,x)-u(0,y)|\Big)
\psi(0,y)\varrho_{\delta} (x-y)\,dx \,dy\Big].
\end{align*}
Note that $(\beta_\vartheta)_{\vartheta>0}$ is a 
sequence of functions satisfying 
$\big| \beta_\vartheta(r)-|r| \big|\le C\vartheta$ 
for any $r\in \R$. Therefore,
\begin{align*}
   \big|\mathcal{A}_4(\vartheta,\delta)\big|  
   \le \text{Const}(\psi)\vartheta  \goto 0  \quad \text{as $\vartheta \goto 0$.}
\end{align*}
Furthermore, let
\begin{align*}
\mathcal{A}_5(\delta)& := \Big|  E\Big[\int_{\R_y^d}\int_{\R_x^d} 
|v(0,x)-u(0,y)|\psi(0,y)\varrho_{\delta}(x-y)\,dx \,dy
- \int_{\R_y^d} |v(0,y)-u(0,y)|\psi(0,y)\,dy\Big] \Big |\notag\\
& \le E \Big[\int_{\R_y^d}\int_{\R_x^d} |v(0,y)-v(0,x)|\psi(0,y) 
\varrho_{\delta}(x-y) \,dx\,dy\Big]\\
& \le ||\psi(0,\cdot)||_\infty E \Big[\int_{|z|\le 1}
\int_{\R_y^d} |v(0,y)-v(0,y+\delta z)| \varrho(z) \,dy\,dz\Big].
\end{align*} 
Note that $ \underset{\delta \downarrow 0}\lim\,\int_{\R_x^d} 
|v(0,x)-v(0,x+\delta z)|\,dx \rightarrow 0 $ for $||z||\le 1$,
and therefore by the bounded convergence theorem  
we have  $ \underset{\delta \downarrow 0}\lim\,  
E\Big[\int_{\R_z^d}\int_{\R_x^d} |v(0,x)-v(0,x+\delta z)|
\varrho(z) \,dx\,dz\Big] = 0 $ and hence  $\mathcal{A}_5(\delta)\goto 0$ as 
$\delta\goto 0$.  Finally, since
\begin{align*}
  &\Big|E\Big[\int_{\R_y^d}\int_{\R_x^d} 
  \beta_{\vartheta}(v(0,x)-u(0,y))\psi(0,y)\varrho_{\delta} (x-y)\,dx \,dy\Big]
  - E\Big[\int_{\R^d} |v(0,x)-u(0,x)|\psi(0,x)\,dx\Big]\Big|\\
  & \qquad 
  \le \mathcal{A}_4(\vartheta,\delta) + \mathcal{A}_5(\delta),
\end{align*} 
we can conclude the proof of the second part of the lemma.
 \end{proof}

We now turn our attention to $(I_2 + J_2)$:
\begin{align*}
 I_2 + J_2
  =&  
  E\Big[\int_{\Pi_T\times \Pi_T} \int_{\R_k} \int_{\alpha=0}^1 
 \beta(v(t,x,\alpha)-k)\psi(s,y) \partial_t \rho_{\delta_0}(t-s)\\
&\hspace{6cm} \times
  \varrho_\delta(x-y) \varsigma_l(u_\eps(s,y)-k) \,d\alpha\,dk\,dx\,dt\,dy\,ds\Big] \\
    &~+ E\Big[\int_{\Pi_T\times \Pi_T}\int_{\R_k} \int_{\alpha=0}^1 
  \beta(u_\eps(s,y)-k)\partial_s \psi(s,y) \rho_{\delta_0}(t-s)\\
&\hspace{6cm} \times
 \varrho_\delta(x-y) \varsigma_l(v(t,x,\alpha)-k) \,d\alpha\,dk\,dy\,ds\,dx\,dt\Big] \\
  & ~+
   E\Big[\int_{\Pi_T\times\Pi_T} \int_{\R_k} \int_{\alpha=0}^1 \beta(u_\eps(s,y)-k) \psi(s,y)  
   \partial_s\rho_{\delta_0}(t-s)\\
&\hspace{6cm} \times
   \varrho_\delta(x-y) \varsigma_l(v(t,x,\alpha)-k) \,d\alpha\,dk\,dy\,ds\,dx\,dt\Big] \\
   =& -E\Big[\int_{\Pi_T\times \Pi_T} \int_{\R_k} 
  \int_{\alpha=0}^1 \beta(v(t,x,\alpha)- u_\eps(s,y)+ k)\psi(s,y) \partial_s \rho_{\delta_0}(t-s)\\
&\hspace{8cm} \times \varrho_\delta(x-y) \varsigma_l(k) \,d\alpha\,dk\,dx\,dt\,dy\,ds \Big]\\
  &~+ E\Big[\int_{\Pi_T\times \Pi_T}\int_{\R_k} \int_{\alpha=0}^1 
  \beta((v(t,x,\alpha)- u_\eps(s,y)+ k)) \psi(s,y) \partial_s \rho_{\delta_0}(t-s)\\
&\hspace{8cm} \times
  \varrho_\delta(x-y)  \varsigma_l(k) \,d\alpha\,dk\,dy\,ds\,dx\,dt\Big] \\
  &~+ E\Big[\int_{\Pi_T\times\Pi_T} \int_{\R_k} \int_{\alpha=0}^1 
  \beta(u_\eps(s,y)-k) \partial_s\psi(s,y)\, \rho_{\delta_0}(t-s)\\
&\hspace{7cm} \times
  \varrho_\delta(x-y) \varsigma_l(v(t,x,\alpha)-k) \,d\alpha\,dk\,dy\,ds\,dx\,dt\Big],
 \end{align*}
as $\beta$, $\varsigma_l$ are even functions. Hence, we are left with
 \begin{align*}
I_2+ J_2  =&E\Big[\int_{\Pi_T\times \Pi_T} \int_{\R_k}  
 \int_{\alpha=0}^1 \beta(u_\eps(s,y)-k) \partial_s\psi(s,y)\, \rho_{\delta_0}(t-s)
 \varrho_\delta(x-y)\\
 &\hspace{6cm}\times \varsigma_l(v(t,x,\alpha)-k) \,d\alpha\,dk\,dy\,ds\,dxdt\Big]. 
 \end{align*}

\begin{lem}\label{stochastic_lemma_2}
It holds that
\begin{align*}
    I_2 + J_2  &\underset{\delta_0 \goto 0}\longrightarrow  
    E\Big[ \int_{\Pi_T}\int_{\R_y^d}\int_{\R_k} \int_{\alpha=0}^1 
    \beta(u_\eps(s,y)-k) \partial_s\psi(s,y)\,
    \varrho_\delta(x-y)\varsigma_l(v(s,x,\alpha)-k) \,d\alpha\,dk\,dy\,dx\,ds\Big]\\
    &\underset{l \goto 0}\longrightarrow  
    E \Big[\int_{\Pi_T}\int_{\R_y^d}\int_{\alpha=0}^1 \beta(u_\eps(s,y)-v(s,x,\alpha)) 
    \partial_s\psi(s,y)\, \varrho_\delta(x-y) \,d\alpha\,dy\,dx\,ds\Big]\\
    &\underset{\eps \goto 0}\longrightarrow  E\Big[ \int_{\Pi_T}\int_{\R_y^d}
    \int_{\gamma=0}^1 \int_{\alpha=0}^1 \beta(u(s,y,\gamma)-v(s,x,\alpha)) \partial_s\psi(s,y)
    \varrho_\delta(x-y) \,d\alpha\,d\gamma \,dy\,dx\,ds\Big],
\end{align*} 
and 
\begin{align*}
  \lim_{(\vartheta,\delta)\goto (0,0)} &  
  E \Big[\int_{\Pi_T}\int_{\R_y^d}\int_{\gamma=0}^1 \int_{\alpha=0}^1 
  \beta_{\vartheta}(u(s,y,\gamma)-v(s,x,\alpha)) \partial_s\psi(s,y)
  \varrho_\delta(x-y) \,d\alpha\,d\gamma \,dy\,dx\,ds\Big] \\
  & =E \Big[\int_{s=0}^T \int_{\R^d}\int_{\alpha=0}^1 \int_{\gamma=0}^1 
  |u(s,y,\gamma)-v(s,y,\alpha)| \partial_s\psi(s,y)
  d\gamma\,\,d\alpha\,dy\,ds\Big].
\end{align*}
\end{lem}

\begin{proof}
The proof of the first part of the lemma is divided into three steps.

\textbf{Step 1:} In this step we justify passing to the limit $\delta_0 \goto 0$. Let
\begin{align*}
 &\mathcal{B}_1:=
    \Big| E\Big[\int_{\Pi_T}\int_{\Pi_T} \int_{\R_k} 
 \int_{\alpha=0}^1 \beta(u_\eps(s,y)-k) \partial_s\psi(s,y)\, \rho_{\delta_0}(t-s)
  \varrho_\delta(x-y)\\&\hspace{7cm} \times \varsigma_l(v(t,x,\alpha)-k) \,d\alpha\,dk\,dx\,dt\,dy\,ds\Big]  \\
&\quad\quad 
-E \Big[\int_{\Pi_T}\int_{\R_x^d}\int_{\R_k}\int_{\alpha=0}^1 
\beta(u_\eps(s,y)-k) \partial_s\psi(s,y)\, \varrho_\delta(x-y)
 \varsigma_l(v(s,x,\alpha)-k) \,d\alpha\,dk\,dx\,dy\,ds\Big]\Big| \\
& = \Big| E\Big[\int_{s=\delta_0}^T \int_{\R_y^d} 
\int_{\Pi_T} \int_{\R_k} \int_{\alpha=0}^1 
\Big(\beta(u_\eps(s,y)- v(t,x,\alpha) +k)-\beta(u_\eps(s,y)-v(s,x,\alpha)+k )\Big) \\
&\hspace{5cm}
\times \partial_s\psi(s,y)\, \rho_{\delta_0}(t-s)
\varrho_\delta(x-y)  \varsigma_l(k)  \,d\alpha\,dk\,dx\,dt\,dy\,ds\Big]  \Big |
+o(\delta_0).
\end{align*}
Observe that  
\begin{align*}
&\mathcal{B}_1 \\ &\le 
C(\beta^\prime)  E\Big[ \int_{s=\delta_0}^T\int_{\R_y^d} \int_{\Pi_T} \int_{\alpha=0}^1
|v(t,x,\alpha)-v(s,x,\alpha)| \varrho_{\delta}(x-y)|\partial_s \psi(s,y)|
\rho_{\delta_0}(t-s)\,d\alpha \,dt\,dx\,dy\,ds \Big] \\&\hspace{8cm}+o(\delta_0) \\
& \le C(\beta^\prime, \partial_s \psi)  \Big(E\Big[ \int_{s=\delta_0}^T\int_{\R_y^d} \int_{\Pi_T} \int_{\alpha=0}^1
|v(t,x,\alpha)-v(s,x,\alpha)|^2 \varrho_{\delta}(x-y)\rho_{\delta_0}(t-s)\,d\alpha \,dt\,dx\,dy\,ds\Big] \Big)^{\frac 12} \\&\hspace{8cm}+o(\delta_0)\\
&\big(\text{used Cauchy-Schwartz's inequality 
w.r.t. }~ \varrho_{\delta}(x-y)\rho_{\delta_0}(t-s)
\,d\alpha \,dt\,dx\,dy\,ds\,dP(\omega)\big)\\
& \le  C(\beta^\prime) ||\partial_s \psi||_{\infty}\Big( E \Big[\int_{r=0}^1 \int_{\Pi_T} \int_{\alpha=0}^1
|v(t+ \delta_0\,r ,x,\alpha)-v(t,x,\alpha)|^2 \rho(-r)
\,d\alpha \,dt\,dx\,dr\Big] \Big)^{\frac 12}  +o(\delta_0).
\end{align*}
Note that $\underset{\delta_0\downarrow 0 }\lim\,  
\int_{t=0}^T\int_{\R_x^d} \int_{\alpha=0}^1 
|v(t+\delta_0 r,x,\alpha)-v(t,x,\alpha)|^2\,
 \,d\alpha\,dx\,dt = 0$, almost surely, for 
 every fixed  $ r\in [0,1] $. 
 Therefore, by the bounded convergence theorem,
\begin{align*}
\lim_{\delta_0\downarrow 0}E\Big[\int_{t=0}^T\int_{r=0}^1
\int_{\R_x^d} \int_{\alpha=0}^1 |v(t+\delta_0 r,x,\alpha)-v(t,x,\alpha)|^2
\rho(-r)\,d\alpha\,dx\,dr\,dt\Big]=0.
\end{align*}
This concludes the first step.

\textbf{Step 2:} Let
\begin{align*}
\mathcal{B}_2:&= \Big |E \Big[\int_{\Pi_T}\int_{\R_y^d}\int_{\R_k}\int_{\alpha=0}^1 
\beta(u_\eps(s,y)-k) \partial_s\psi(s,y)
\varrho_\delta(x-y)  
\varsigma_l(v(s,x,\alpha)-k) \,d\alpha\,dk\,dy\,dx\,ds\Big]\\
&\qquad- E\Big[ \int_{\Pi_T}\int_{\R_y^d}\int_{\alpha=0}^1 
\beta(u_\eps(s,y)-v(s,x,\alpha)) \partial_s\psi(s,y)
\varrho_\delta(x-y) \,d\alpha\,dy\,dx\,ds\Big]\Big|\\
&= \Big|E\Big[ \int_{\Pi_T}\int_{\R_x^d}\int_{\R_k} 
\int_{\alpha=0}^1 \Big(\beta(u_\eps(s,y)+k-v(s,x,\alpha))-\beta(u_\eps(s,y)-v(s,x,\alpha))\Big) 
\partial_s\psi(s,y) \\
  & \hspace {5cm}\times \varrho_\delta(x-y)  \varsigma_l(k) \,d\alpha\,dk\,dx\,dy\,ds\Big]\Big|,
\end{align*}
and note that
\begin{align*}
\big|\mathcal{B}_2\big| & \le ||\beta^\prime||_{\infty} E \Big[\int_{\Pi_T}\int_{\R_k}  |k|
|\partial_s\psi(s,y)|\, \varsigma_l(k)\,dk\,dy\,ds\Big]
\le ||\partial_s \psi||_{\infty}\, l\, ||\beta^\prime||_{\infty} \,C(\psi)\goto 0,
\quad \text{as $l\goto 0$.}
\end{align*}  
   
\textbf{Step 3:} Note that $u(s,y,\gamma)$ is the 
$L^2(\R^d\times (0,1))$-valued process that was 
recovered from the Young measure valued 
narrow limit of the sequence $\{u_\eps(s,y)\}_{\eps>0}$ and 
it satisfies Lemmas \ref{conv-young-measure} 
and \ref{lem:measure-conversion}.   
Let 
\begin{align*}
\Gamma_{(x,\alpha)}(s,y,\omega; \xi) 
= \beta(\xi -v(s,x,\alpha)) \partial_s\psi(s,y)\, \varrho_\delta(x-y).
\end{align*} 
Clearly, for every fixed $(x,\alpha)\in \rd\times(0,1)$,  $\Gamma_{(x,\alpha)}$ 
is a Caratheodory function and \newline$\big\{\Gamma_{(x,\alpha)}(s,y,\omega; u_{\eps_n}(s,y))\big\}_n$ 
is uniformly integrable in $L^1\big((\Theta, \Sigma, \mu), \R\big)$ and 
satisfies the conditions of Lemma \ref{conv-young-measure}. 
Hence, for every $(x,\alpha)\in \R^d\times (0,1)$,
\begin{align}
&\lim_{\eps_n\goto 0} \int_{\Omega} \int_{\Pi_T} 
\Gamma_{(x,\alpha)}(s,y,\omega; u_{\eps_n}(s,y))\,ds\,dy\,dP(\omega)\notag \\
=&\int_{\Omega} \int_{\Pi_T} \int_{\gamma=0}^1  
\Gamma_{(x,\alpha)}(s,y,\omega; u(s,y,\gamma))\,d\gamma \,ds\,dy\, dP(\omega). \label{stochastic_estimate_young_1}
\end{align} 
In view of \eqref{stochastic_estimate_young_1}, we now invoke the 
bounded convergence and Fubini theorems to conclude 
\begin{align*}
\lim_{\eps_n \goto 0} &E \Big[\int_{\Pi_T}\int_{\R_y^d}\int_{\alpha=0}^1 
\beta(u_{\eps_n}(s,y)-v(s,x,\alpha)) \partial_s\psi(s,y)
\varrho_\delta(x-y) \,d\alpha\,dy\,dx\,ds\Big]\\
& = E \Big[\int_{\Pi_T}\int_{\R_y^d}\int_{\alpha=0}^1 \int_{\gamma=0}^1 
\beta(u(s,y,\gamma)-v(s,x,\alpha)) \partial_s\psi(s,y)
\varrho_\delta(x-y) \,d\gamma\,d\alpha\,dy\,dx\,ds\Big].
\end{align*} 
This concludes the proof of the first part of the lemma.

\vskip0.2cm
For the second part we proceed as follows:
\begin{align*}
\mathcal{B}_3(\vartheta,\delta) &:= 
\Big| E \Big[ \int_{\Pi_T}\int_{\R_x^d}\int_{\gamma=0}^1 \int_{\alpha=0}^1  
\beta_{\vartheta}(u(s,y,\gamma)-v(s,x,\alpha)) \partial_s\psi(s,y)
\varrho_\delta(x-y) \,d\alpha\,d\gamma\,dx\,dy\,ds\Big]\\
& \qquad -E \Big[\int_{\Pi_T}\int_{\R_x^d}\int_{\gamma=0}^1 \int_{\alpha=0}^1 
|u(s,y,\gamma)-v(s,x,\alpha)| \partial_s\psi(s,y)
\varrho_\delta(x-y) \,d\alpha\,d\gamma\,dx\,dy\,ds\Big]\Big|\\
&\le ||\partial_s\psi||_{\infty}\,E \Big[\int_{|y|\le C_{\psi}}
\int_{\Pi_T}\int_{\gamma=0}^1 \int_{\alpha=0}^1
\big|  \beta_{\vartheta}(u(s,y,\gamma)-v(s,x,\alpha))- |u(s,y,\gamma)-v(s,x,\alpha)|\big| \\
&\hspace{6cm} \times \varrho_{\delta}(x-y)\,d\alpha \,d\gamma \,dx\,ds\,dy\Big].
\end{align*}
Since $ \big|\beta_\vartheta(r)-|r|\big|\le C\vartheta $, for any $ r\in \R$, it follows 
that $\mathcal{B}_3\le ||\partial_s\psi||_{\infty}\,\vartheta\, C(\psi,T)$.  
Observe that
\begin{align*}
\mathcal{B}_4(\delta) &:= 
\Big| E \Big[\int_{\Pi_T}\int_{\R_y^d}\int_{\alpha=0}^1 \int_{\gamma=0}^1
|u(s,y,\gamma)-v(s,x,\alpha)| \partial_s\psi(s,y)
\varrho_\delta(x-y) \,d\gamma\,d\alpha\,dy\,dx\,ds\Big]\\
& \qquad - E \Big[\int_{s=0}^T \int_{\R_y^d}\int_{\alpha=0}^1 \int_{\gamma=0}^1 
|u(s,y,\gamma)-v(s,y,\alpha)| \partial_s\psi(s,y)\, d\gamma\,d\alpha\,dy\,ds\Big]\Big| \\
& \le E \Big[\int_{\Pi_T}\int_{\R_y^d}\int_0^1  |v(s,y,\alpha)-v(s,x,\alpha)|\, |\partial_s\psi(s,y)|
\varrho_\delta(x-y)\,d\alpha\,dy\,dx\,ds\Big]\\
& \le C(\psi)\Big(E \Big[\int_{\Pi_T}\int_{\R_y^d}\int_0^1  |v(s,y,\alpha)-v(s,x,\alpha)|^2
\varrho_\delta(x-y)\,d\alpha\,dy\,dx\,ds\Big]\Big)^{\frac 12} \\
&\qquad (\text{here we used Cauchy-Schwartz's inequality})\\
&   \text{$\goto 0$ as $\delta \goto 0$},
\end{align*} 
where the $\delta\goto 0$ limit follows by arguments similar to 
those used to prove the last part 
of Lemma \ref{stochastic_lemma_1}. 

Since
\begin{align*}
 &\Big| E \Big[\int_{\Pi_T}\int_{\R_x^d}\int_{\gamma=0}^1 \int_{\alpha=0}^1 
\beta_{\vartheta}(u(s,y,\gamma)-v(s,x,\alpha)) \partial_s\psi(s,y)
\varrho_\delta(x-y) \,d\alpha\,d\gamma \,dx\,dy\,ds \Big]\\
& \quad -E \Big[\int_{s=0}^T \int_{\R_y^d}\int_{\alpha=0}^1 \int_{\gamma=0}^1 
|u(s,y,\gamma)-v(s,y,\alpha)| \partial_s\psi(s,y)\,d\gamma\,d\alpha\,dy\,ds\Big]\Big| \\
& \quad \quad 
\le  \mathcal{B}_3(\vartheta,\delta)  
+\mathcal{B}_4(\delta) \goto 0, \quad \text{as $(\vartheta,\delta)\goto (0,0)$,}
\end{align*} 
the second part of the lemma follows.
\end{proof}

Next we consider the stochastic term 
$ I_3 + J_3$; we begin with the following assertion:
\begin{lem}
For any two constants $T_1, T_2\ge 0$ with $T_1<T_2$,
\begin{align}
E\Big[X_{T_1}\int_{T_1}^{T_2} \int_{|z|>0}
\zeta(t,z)\,\tilde{N}(dz,dt)\Big] = 0,
\label{eq:conditional_indep}
\end{align}
where $\zeta$ is a predictable integrand with 
$E \Big[\int_0^T \int_{|z|>0}\zeta^2(t,z)\, m(dz)\, dt\Big] < \infty$ 
and $X$ is an adapted process. 
\end{lem}
\begin{proof}
Let $\mathcal{M}(t)= \int_{0}^{t} \int_{|z|>0}\zeta(s,z)\,\tilde{N}(dz,ds)$. 
Clearly, $\mathcal{M}(t)$ is a martingale, and thus
\begin{align*}
 &E\Big[X_{T_1}\int_{T_1}^{T_2} \int_{|z|>0}\zeta(t,z)\,\tilde{N}(dz,dt)\Big] \\
 & \quad = E\Big[ X_{T_1}\Big(\mathcal{M}(T_2)-\mathcal{M}(T_1) \Big)\Big]\\
 &\quad = E\Big[E\big( X_{T_1} M(T_2)|\mathcal{F}_{T_1}\big)\Big]
 - E\Big[ X_{T_1} \mathcal{M}(T_1)\Big]\\
  &\quad = E\Big[X_{T_1}E\big( M(T_2)|\mathcal{F}_{T_1}\big)\Big]
  - E\Big[ X_{T_1} \mathcal{M}(T_1)\Big]\\
  & \quad= E\Big[ X_{T_1} \mathcal{M}(T_1)\Big]
  - E\Big[ X_{T_1} \mathcal{M}(T_1)\Big] = 0.
\end{align*}
\end{proof}

For any $\beta \in C^\infty(\R)$ with $\beta^\prime, 
\beta^{\prime\prime}\in C_b(\R)$ and 
any nonnegative $\phi\in C_c^\infty(\Pi_{\infty}\times \Pi_\infty)$, we  define 
\begin{align*}
J[\beta, \phi](s; y, v) :=&
\int_{r=0}^T\int_{|z|>0}\int_{\rd_x} \int_0^1 
\Big(\beta\big(v(r,x,\alpha) +\eta(x,v(r,x,\alpha); z)-v\big)
 -\beta\big(v(r,x,\alpha)-v\big)\Big) \\
&\hspace{6cm}\times\phi(r,x,s,y) \,d\alpha \,dx\, \tilde{N}(dz,dr), 
\end{align*} 
where  $0\le s\le T$ and $(y,v)\in \R^d \times \R$.  Note that $\phi$ has compact 
support, i.e., there exists a constant $c_\phi > 0$ 
such that  $\phi(\cdot,\cdot, \cdot, y) = 0$
for $|y| \ge c_\phi$. As a result $ J[\beta, \phi](s; y, v)= 0$ if $|y| > c_\phi$ and $0\le s\le T$.
Furthermore, we extend the process $u_\eps(\cdot,y)$ for negative times by setting 
$u_\eps(s,y) = u_{\eps} (0,y)$ if $s< 0$.
With this convention, by It\^{o}-L\'{e}vy product rule
\begin{align*}
J[\beta,\phi_{\delta,\delta_0}](s;y,v) = &\int_{s-\delta_0}^s\int_{|z|>0}\int_{\rd_x} \int_0^1 
\Big(\beta\big(v(r,x,\alpha) +\eta(x,v(r,x,\alpha); z)-v\big)
 -\beta\big(v(r,x,\alpha)-v\big)\Big) \\&\hspace{6cm}\times
 \phi_{\delta,\delta_0}(r,x,s,y) \,d\alpha \,dx\, \tilde{N}(dz,dr)\\
   =& -\psi(s,y)\int_{s-\delta_0}^s \Big[ \int_{s-\delta_0}^r \int_{|z|> 0} 
   \zeta_{(y, v)}(\sigma,z)\tilde{N}(\,dz,\,d\sigma)\Big] \rho_{\delta_0}^{\prime}(r-s)\,dr
\end{align*} 
where
$$ 
\zeta_{(y, v)}(r,z) = \int_{\rd_x} \int_0^1 
\Big(\beta\big(v(r,x,\alpha) +\eta(x,v(r,x,\alpha); z)-v\big)
 -\beta\big(v(r,x,\alpha)-v\big)\Big)
\varrho_{\delta}(x-y)  \,d\alpha \,dx.
$$ 
Therefore, by Fubini's theorem and \eqref{eq:conditional_indep},
\begin{align}
&E\Big[ J[\beta,\phi_{\delta,\delta_0}](s;y,v)\, 
\varsigma_l(u_\eps(s-\delta_0,y)-v)\Big]\notag \\
& =- \psi(s,y)\int_{s-\delta_0}^sE \Big[ \varsigma_l(u_\eps(s-\delta_0,y)-v) 
\int_{s-\delta_0}^r \int_{|z|> 0} \zeta_{(y, v)}(\sigma,z)\tilde{N}(\,dz,\,d\sigma)\Big] 
\rho_{\delta_0}^{\prime}(r-s)\,dr\notag\\
&=0 \label{eq:martingale-expect}
\end{align} 
for all $(s,y,v)$. Finally, we apply Fubini's theorem along 
with \eqref{eq:martingale-expect} and obtain

\begin{align*}
E\Big[\int_{\R_v}\int_{\Pi_{T}} J[\beta,\phi_{\delta,\delta_0}](s;y,v)\,
\varsigma_l(u_\eps(s-\delta_0,y)-v)\,dy\,ds\,dv\Big]=0.
\end{align*}   
Therefore
\begin{align}
I_3  &= E\Big[\int_{\R_v}\int_{\Pi_{T}} 
J[\beta,\phi_{\delta,\delta_0}](s;y,v)
\Big(\varsigma_l(u_\eps(s,y)-v)-\varsigma_l(u_\eps(s-\delta_0,y)-v)\Big)
\,dy\,ds\,dv\Big]. 
\label{stochastic_estimate_6}
\end{align}

\begin{lem}\label{lem:differentiation} 
The following identities hold:
\begin{align*}
&\partial_v  J[\beta, \phi](s; y, v) =  J[-\beta^\prime, \phi](s; y, v)\\
& \partial_{y_k}  J[\beta, \phi](s; y, v) =  J[\beta, \partial_{y_k}\phi](s; y, v).
\end{align*}
\end{lem}

\begin{proof}
The proof follows by a classical argument 
validating differentiation under the integral sign. 
\end{proof}

\begin{lem} \label{lem:L-infinity estimate}
Let $\beta\in C^\infty(\R)$ be function such 
that $\beta^\prime, \beta^{\prime\prime}\in C_c^\infty(\R) $ and $p$ be a 
positive integer of the form $p= 2^k$ for some $k\in \mathbb{N}$.
If $p\geq d+3$, then there exists a 
constant $C=C(\beta^\prime, \psi,\delta)$ such that
\begin{align}
\sup_{0\le s\le T}
\Big( 
E\Big[|| J[\beta,\phi_{\delta,\delta_0}](s;\cdot,\cdot)||_{L^\infty(\R^d\times \R)}^2\Big]
\Big) 
\le \frac{C(\beta', \psi,\delta )}{\delta_0^\frac{2(p-1)}{p}}.
\label{eq:l-infinity bound}
\end{align}
 \end{lem}

 \begin{proof}
We estimate as follows:
\begin{align}
&E\Big[|| J[\beta,\phi_{\delta,\delta_0}](s;\cdot,\cdot)||_p^p\Big]\notag \\
& \quad 
=E\Big[\int_{\R_v} \int_{\R^d_y} \Big|J[\beta,\phi_{\delta,\delta_0}](s;y,v)\Big|^p \,dy\,dv\Big] \notag \\
&\quad
=  E\Big[\int_{\R_v} \int_{\R^d_y} \Big|\int_{r=0}^T\int_{|z|>0} \int_{\R^d_x}
\int_{\lambda=0}^1\int_{\alpha=0}^1 
\eta(x,v(r,x,\alpha);z) \beta^{\prime}(v(r,x,\alpha)-v +\lambda \eta(x,v(r,x,\alpha);z)) 
\notag \\
& \hspace{5cm}\times 
\rho_{\delta_0}(r-s) \varrho_{\delta}(x-y)\psi(s,y)\, d\alpha \,d\lambda \,dx\,\tilde{N}(dz,dr)
\Big|^p\,\,dy\,dv \Big]\notag \\
& \qquad \big(\text{by the BDG inequality}\big) \notag \\
& \quad  \le C \int_{\R_v} \int_{\R_y^d} 
E\Big[\Big (\int_{r=0}^T\int_{|z|>0} \Big|\int_{\R_x^d} \int_{\alpha=0}^1
\int_{\lambda=0}^1
\beta^{\prime}(v(r,x,\alpha)-v +\lambda \eta(x,v(r,x,\alpha);z)) \notag \\
& \hspace{3cm} \times \eta(x,v(r,x,\alpha);z) \rho_{\delta_0}(r-s) \psi(s,y)\, \varrho_{\delta}(x-y)
\,d\lambda\,d\alpha \,dx \Big|^2\,N(dz,dr)\Big)^\frac{p}{2} \Big]\,dy\,dv\notag \\
& \qquad \big(  \text{by Schwartz's inequality w.r.t.~the measure}
\, \varrho_{\delta}(x-y)\,d\lambda\,d\alpha \,dx\big)\notag\\
& \quad \le C \int_{\R_v} \int_{|y|<C_\psi} 
E\Big[\Big (\int_{r=0}^T\int_{|z|>0}\int_{\rd_x}  \int_{\alpha=0}^1\int_{\lambda=0}^1 
 {\beta^{\prime}}^2(v(r,x,\alpha) -v +\lambda \eta(x,v(r,x,\alpha);z)) \notag \\
& \hspace{3cm} \times \eta^2(x,v(r,x,\alpha);z) \rho_{\delta_0}^2(r-s) \psi^2 (s,y)\varrho_{\delta} (x-y)\,d\lambda \,d\alpha\,dx \,N(dz,dr)\Big)^\frac{p}{2} \Big]\,dy\,dv\notag \\
& \quad \le C \int_{\R_v} \int_{|y|<C_\psi} 
E\Big[\Big (\int_{r=0}^T\int_{|z|>0}\int_{\rd_x}  \int_{\alpha=0}^1\int_{\lambda=0}^1 
 {\beta^{\prime}}^2(v(r,x,\alpha)-v +\lambda \eta(x,v(r,x,\alpha);z)) \notag \\
& \hspace{3cm}\times \eta^2(x,v(r,x,\alpha);z)\rho_{\delta_0}^2(r-s) \psi^2 (s,y)
\,\varrho_{\delta} (x-y)\,d\lambda \,d\alpha\,dx\,\tilde{N}(dz,dr)\Big)^\frac{p}{2}\Big]\,dy\,dv\notag \\ 
& \quad\quad 
+  C \int_{\R_v} \int_{|y|<C_\psi} E\Big[\Big (\int_{r=0}^T\int_{|z|>0}\int_{\rd_x}  
\int_{\alpha=0}^1\int_{\lambda=0}^1 
 {\beta^{\prime}}^2(v(r,x,\alpha)-v +\lambda \eta(x,v(r,x,\alpha);z)) \notag \\
& \hspace{3cm}\times \eta^2(x,v(r,x,\alpha);z) \rho_{\delta_0}^2(r-s) \varrho_{\delta} (x-y)
\psi^2 (s,y)\,d\lambda\,d\alpha \,dx\,m(dz)\,dr \Big)^\frac{p}{2} \Big]\,dy\,dv\notag \\
&\qquad  \Big(\text{noting that} \, p=2^k, \,\text{ and applying the BDG inequality 
followed by the Cauchy-Schwartz's  } \notag\\ & \qquad
\text{inequality another }(k-1)~ \text{ times, we obtain}\Big)\notag \\
&\quad \le \sum_{j=0}^{k-1} C_j  \int_{\R_v} 
\int_{|y|<C_\psi} E\Big[\int_{r=0}^T\int_{|z|>0}
\int_{\rd_x}  \int_{\alpha=0}^1\int_{\lambda=0}^1 
\big |  {\beta^{\prime}}(v(r,x,\alpha)-v +\lambda \eta(x,v(r,x,\alpha);z)) \notag \\
& \hspace{3cm}\times \eta(x,v(r,x,\alpha);z) \rho_{\delta_0}(r-s) \psi (s,y)\big|^{\frac{p}{2^j}}
\,\varrho_{\delta} (x-y)\,d\lambda\,d\alpha \,dx\,m(dz)\,dr \Big]^{2^j}\,dy\,dv\notag \\
& \quad \le \sum_{j=0}^{k-1} C_j  \int_{\R_v} \int_{|y|<C_\psi} 
E\Big[\int_{r=0}^T\int_{|z|>0}\int_{\rd_x}  \int_{\alpha=0}^1\int_{\lambda=0}^1 
\big |  {\beta^{\prime}}(v(r,x,\alpha)-v +\lambda \eta(x,v(r,x,\alpha);z)) g(x)\notag \\
& \hspace{1.5cm}\times(1+|v(r,x,\alpha)|) \rho_{\delta_0}(r-s) \psi (s,y)\big|^{\frac{p}{2^j}}
\,\varrho_{\delta} (x-y)\min(1, |z|^2)\,d\lambda\,d\alpha \,dx\,m(dz)\,dr \Big]^{2^j}dy\,dv\notag \\
& \qquad \big(\text{applying  H\"{o}lder's inequality w.r.t. the measure} 
\, \varrho_{\delta} (x-y)\min(1, |z|^2)\,d\lambda\,d\alpha \,dx\,m(dz)\,dr\big)\notag\\
& \quad \le \sum_{j=0}^{k-1} C_j  \int_{\R_v} \int_{|y|<C_\psi} 
E\Big[\int_{r=0}^T\int_{|z|>0}\int_{\rd_x}  \int_{\alpha=0}^1\int_{\lambda=0}^1 
 \big | g(x) {\beta^{\prime}}(v(r,x,\alpha)-v +\lambda \eta(x,v(r,x,\alpha);z)) \notag \\
& \hspace{1.5cm}\times (1+|v(r,x,\alpha)|) \rho_{\delta_0}(r-s) \psi (s,y)\big|^{p}
\,\varrho_{\delta} (x-y)\min(1, |z|^2)\,d\lambda\,d\alpha \,dx\,m(dz)\,dr\Big]\,\,dy\,dv\notag \\
& \quad \le C E\Big[\int_{|y|<C_\psi}\int_{r=0}^T\int_{|z|>0}
\int_{\rd_x}\int_{\alpha=0}^1  \int_{|v|\le C_{\beta^\prime}+|v(r,x,\alpha)|+|\eta(x,v(r,x,\alpha);z)|}
(1+|v(r,x,\alpha)|^p)||\beta^\prime||_\infty^p \notag \\
& \hspace{1.5cm}\times g^p(x)\rho_{\delta_0}^p(r-s) ||\psi||_\infty^p
\varrho_{\delta} (x-y) \min(1, |z|^2)\,dv \,d\alpha\,dx\,m(dz)\,dr\,dy\Big]\notag \\
& \quad \le C(\beta, \psi) E\Big[\int_{r=0}^T\int_{\rd_x}
\int_{\alpha=0}^1 g^p(x)(1+|v(r,x,\alpha)|^p)\notag\\
&\hspace{3cm}\times(C_{\beta^\prime}+ |g(x)|(1+|v(r,x,\alpha)|))
\rho_{\delta_0}^p(r-s)\,d\alpha \,dx \,dr\,\Big]\notag \\
& \quad \le C(\beta, \psi) \int_{r=0}^T\Big(1+E \big[||v(r,\cdot,\cdot)||^{p+1}_{p+1}\big]\Big)\, \rho_{\delta_0}^p(r-s) \,dr
\notag \\ 
& \quad \le C(\beta,\psi)\Big(1+\sup_{0\le r\le T}
E\big[||v(r,\cdot,\cdot)||_{p+1}^{p+1}\big] \Big) \int_{r=0}^T 
\rho_{\delta_0}^p(r-s) \,dr\notag\\ 
&\quad \le C(\beta,\psi)
\Big(1+\sup_{0\le r\le T}E\Big[||v(r,\cdot)||_{p+1}^{p+1}\Big]\Big) \,
||\rho_{\delta_0}||_\infty^{p-1}\int_{r=0}^T \rho_{\delta_0}(r-s) \,dr\notag \\
& \quad 
\le \frac{C(\beta,\psi)\Big(1+\sup_{0\le r\le T}
E\Big[||v(r,\cdot)||_{p+1}^{p+1}\Big]\Big)}{\delta_0^{p-1}}.
\label{l-1}
\end{align}

Similarly, we can derive the following bounds:
\begin{align}
& E\Big[|| \partial_v J[\beta,\phi_{\delta,\delta_0}](s;\cdot,\cdot)||_p^p\Big] \le
\frac{C(\beta^{\prime\prime}, \psi )}{\delta_0^{p-1}},
\label{l-2}  \\
& E\Big[|| \partial_{y_k} J[\beta,\phi_{\delta,\delta_0}](s;\cdot,\cdot)||_p^p\Big] \le
\frac{C(\beta', \partial_{y_k} \psi, \delta )}{\delta_0^{p-1}}.
 \label{l-3}
\end{align}
Therefore, in view of \eqref{l-1}, \eqref{l-2}, 
and \eqref{l-3}, we have arrived at
\begin{align*}
E\Big[|| J[\beta,\phi_{\delta,\delta_0}](s;\cdot,\cdot)||_{W^{1,p}(\R^d\times \R)}^p\Big] 
\le \frac{C(\beta', \psi,\delta )}{\delta_0^{p-1}}.
\end{align*}
Finally, we use the Sobolev embedding along with  
Cauchy-Schwartz's inequality to arrive at 
\eqref{eq:l-infinity bound}.
\end{proof}

\begin{lem}\label{stochastic_lemma_3} 
It holds that $J_3 = 0$ and 
\begin{align} 
\lim_{l\goto 0}\lim_{\delta_0 \goto 0} I_3  
& = E\Big[\int_{\Pi_{T}}\int_{\R_x^d}\int_{|z|>0} \int_{\alpha=0}^1 
\Big( \beta(v(r,x,\alpha)+ \eta(x,v(r,x,\alpha);z)
-u_\eps(r,y)-\eta_\eps(y,u_\eps(r,y);z)) \notag \\
& \hspace{3cm}
-\beta(v(r,x,\alpha)-u_\eps(r,y)-\eta_\eps(y,u_\eps(r,y);z))+ \beta(v(r,x,\alpha)-u_\eps(r,y)) \notag \\
&  \hspace{4cm}
-\beta(v(r,x,\alpha)+ \eta(x,v(r,x,\alpha);z)-u_\eps(r,y)) \Big)  \notag \\ 
& \hspace{5cm}\times \psi(r,y)\,\varrho_{\delta}(x-y)\,d\alpha\,m(dz)\,dx\,dy\,dr\Big]. \notag
\end{align}
\end{lem}

\begin{proof}
Note that 
\begin{align*}
J_3 & = 
\int_{\Pi_T}\int_{\alpha=0}^1\int_{\R_v} 
E \Big[\varsigma_l(v(t,x,\alpha)-v) \\ &\hspace{2cm}\times\int_{s=t}^{t+\delta_0}\int_{|z|>0} 
\int_{\R_y^d}\Big(\beta(u_\eps(s,y) +\eta_\eps(y,u_\eps(s,y);z)-v)
-\beta(u_\eps(s,y)-v)\Big)\\
& \hspace{7cm}\times \phi_{\delta,\delta_0}\,\,dy\,\tilde{N}(dz,ds)\Big]\,dv\,d\alpha \,dx\,dt\\
        & \qquad  = 0, \quad
        \text{thanks to \eqref{eq:conditional_indep}.}
\end{align*}
For all $y\in \R^d$, $u_\eps(\cdot,y)$ solves
\begin{align*}
du_\eps(s,y)=-\text{div}F_\eps(u_\eps(s,y))ds 
+ \eps  \Delta u_\eps(s,y)\,ds+\int_{|z|>0}\eta_\eps(y,u_\eps(s,y);z)\, \tilde{N}(dz,ds). 
 \end{align*}  
Now we apply the It\^{o}-L\'{e}vy formula to 
$\varsigma_l(u_\eps(s,y)-v)$: 
\begin{align*}
& \varsigma_l(u_\eps(s,y)-v)-\varsigma_l(u_\eps(s-\delta_0,y)-v)\\
& \quad 
= \int_{s-\delta_0}^s \varsigma_l^{\prime} (u_\eps(\sigma,y)-v)\big(-\text{div}F_\eps(u_\eps(\sigma,y)) 
+ \eps  \Delta u_\eps(\sigma,y)\big) \,d\sigma \\
& \quad\quad
+ \int_{s-\delta_0}^s \int_{|z|>0} \Big( \varsigma_l(u_\eps(\sigma,y)+\eta_\eps(y,u_\eps(\sigma,y);z)-v)
- \varsigma_l(u_\eps(\sigma,y)-v)\Big)\,  \tilde{N}(dz,d\sigma)\\
&\quad\quad + \int_{s-\delta_0}^s \int_{|z|>0} 
\int_{\lambda=0}^1(1-\lambda) |\eta_\eps(y, u_\eps(\sigma,y);z)|^2 
\varsigma_{l}^{\prime\prime}( u_\eps(\sigma,y)-v+\lambda \eta_\eps(y, u_\eps(\sigma,y);z) ) 
\,d\lambda \,m(dz)\,d\sigma\\
& \quad 
= - \frac{\partial} {\partial v}\int_{s-\delta_0}^s 
\varsigma_l(u_\eps(\sigma,y)-v)\big(-\text{div}F_\eps(u_\eps(\sigma,y)) 
+ \eps  \Delta u_\eps(\sigma,y)\big)  \,d\sigma \\
& \quad\quad
+ \int_{s-\delta_0}^s \int_{|z|>0} \Big( \varsigma_l(u_\eps(\sigma,y)
+\eta_\eps(y,u_\eps(\sigma,y);z)-v)- \varsigma_l(u_\eps(\sigma,y)-v)\Big)
\, \tilde{N}(dz,d\sigma)\\
& \quad\quad 
+ \int_{s-\delta_0}^s \int_{|z|>0} \int_{\lambda=0}^1
(1-\lambda) |\eta_\eps(y, u_\eps(\sigma,y);z)|^2 
\varsigma_{l}^{\prime\prime}( u_\eps(\sigma,y)-v
+\lambda \eta_\eps(y, u_\eps(\sigma,y);z) )\,d\lambda \,m(dz)\,d\sigma.
 \end{align*}

Therefore, from \eqref{stochastic_estimate_6}, we have
\begin{align}
& I_3 = \notag 
\\ & \quad
E\Big[\int_{\R_v}\int_{\Pi_{T}} J[\beta,\phi_{\delta,\delta_0}](s;y,v)
\Big\{ - \frac{\partial} {\partial v}
\int_{s-\delta_0}^s \varsigma_l(u_\eps(\sigma,y)-v)\big(-\text{div}F_\eps(u_\eps(\sigma,y)) 
+ \eps  \Delta u_\eps(\sigma,y)\big) \,d\sigma  \notag \\
& \quad\quad 
+ \int_{s-\delta_0}^s \int_{|z|>0} 
\Big( \varsigma_l(u_\eps(\sigma,y)+\eta_\eps(y,u_\eps(\sigma,y);z)-v)
- \varsigma_l(u_\eps(\sigma,y)-v)\Big)\, \tilde{N}(dz,d\sigma) \notag \\
&\quad\quad 
+ \int_{s-\delta_0}^s \int_{|z|>0}  \int_{\lambda=0}^1
(1-\lambda) |\eta_\eps(y, u_\eps(\sigma,y);z)|^2 
\varsigma_{l}^{\prime\prime}( u_\eps(\sigma,y)-v
+\lambda \eta_\eps(y, u_\eps(\sigma,y);z) ) \notag \\
& \hspace{5cm} 
\times\,d\lambda\,m(dz)\,d\sigma\,\Big\}\,dy\,ds\,dv\Big] \notag \\
&\qquad (\text{by the It\^{o}-L\'{e}vy product rule and integration by parts})\notag \\
&= E\Big[\int_{\R_v}\int_{\Pi_{T}} 
J[\beta^{\prime},\phi_{\delta,\delta_0}](s;y,v)\Big(\int_{s-\delta_0}^s 
\varsigma_l(u_\eps(\sigma,y)-v)\,\text{div}F_\eps(u_\eps(\sigma,y)) 
\,d\sigma\Big)\,ds\,dy\,dv\Big]\notag  \\
& \quad -E\Big[\int_{\R_v}\int_{\Pi_{T}} J[\beta^{\prime},\phi_{\delta,\delta_0}](s;y,v)
\Big(\int_{s-\delta_0}^s \varsigma_l(u_\eps(\sigma,y)-v)
\, \eps  \Delta u_\eps(\sigma,y)\big) \,d\sigma\Big)\,ds\,dy\,dv\Big]\notag  \\
& 
\quad + E\Big[\int_{\Pi_{T}}\int_{\R_v}\int_{r=s-\delta_0}^s
\int_{\R_x^d}\int_{|z|>0} \int_{\alpha=0}^1 \Big( \beta(v(r,x,\alpha)+ \eta(x,v(r,x,\alpha);z)-v)
-\beta(v(r,x,\alpha)-v)\Big) \notag \\
& \hspace{4cm}\times\Big( \varsigma_l(u_\eps(r,y)+\eta_\eps(y,u_\eps(r,y);z)-v)
-\varsigma_l(u_\eps(r,y)-v)\Big) \notag \\
& \hspace{5cm}\times\rho_{\delta_0}(r-s)\,\psi(s,y)\,\varrho_{\delta}(x-y) 
\,d\alpha \,m(dz)\,dx\,dr\,dv\,dy\,ds\Big] \notag \\
&\quad +E\Big[\int_{\R_v}\int_{\Pi_{T}} 
J[\beta,\phi_{\delta,\delta_0}](s;y,v)\Big\{ \int_{s-\delta_0}^s \int_{|z|>0}
\int_{\lambda=0}^1(1-\lambda) |\eta_\eps(y, u_\eps(\sigma,y);z)|^2\notag\\
&\hspace{4cm}\times \varsigma_{l}^{\prime\prime}( u_\eps(\sigma,y)-v
+\lambda \eta_\eps(y, u_\eps(\sigma,y);z) )\,d\lambda\,m(dz)
\,d\sigma \,\Big\}\,dy\,ds\,dv\Big] \notag \\
&\quad =: A_1^{l,\eps}(\delta,\delta_0) 
+ A_2^{l,\eps}(\delta,\delta_0)+  B^{\eps, l} + A_3^{l,\eps}(\delta,\delta_0). 
\notag
 \end{align}
 
\textbf{Claim 1:}
\begin{align*}
A_1^{l,\eps}(\delta,\delta_0) \goto 0 \quad \text{as $\delta_0 \goto 0$.}
\end{align*}
\noindent{\it Justification:}  Let 
\begin{align}
   G_\eps(u,v)=\int_{r=0}^v \beta^{\prime\prime}(u-r) 
   F_{\eps,k}^\prime(r)dr\quad \text{for}~~u,v \in \R.\notag
\end{align}
It is easy to check that there is a 
positive integer $p$ such that
\begin{align}
  \sup_{\eps>0}| G_\eps(u,v)| \le C_\beta (1+ |u|^p)
  \quad \text{for all $u,v \in \R$.} 
  \label{eq:uniform-growth-estimate}
\end{align}
Furthermore, define
\begin{align}
X_\eps[\phi_{\delta,\delta_0}](s;y,v)& 
:= \int_{\rd_x}\int_{r=0}^T  \int_{|z|>0} \int_{\lambda=0}^1 
\int_{\alpha=0}^1\eta(x,v(r,x,\alpha);z))  
G_\eps(v(r,x,\alpha)+ \lambda \eta(x,v(r,x,\alpha);z),v) \notag \\
& \hspace{5cm}\times\phi_{\delta,\delta_0}(r,x;s,y)
\,d\alpha\,d\lambda\, \tilde{N}(dz,dr)\,dx.\notag
\end{align}
Once again by differentiating under the integral sign,
\begin{align}
&\partial_v X_\eps[\phi_{\delta,\delta_0}](s;y,v) \notag\\
& \quad 
=\int_{\rd_x} \int_{r=0}^T  \int_{|z|>0} \int_{\lambda=0}^1 
\int_{\alpha=0}^1\eta(x,v(r,x,\alpha);z)) \partial_v G_\eps(v(r,x,\alpha)
+ \lambda \eta(x,v(r,x,\alpha);z),v) \notag \\
& \hspace{8cm}\times\phi_{\delta,\delta_0}(r,x;s,y)
\,d\alpha\,d\lambda\,\tilde{N}(dz,dr)\,dx.\notag \\
&\partial_{y_k} X_\eps[\phi_{\delta,\delta_0}](s;y,v) 
= X_\eps[\partial_{y_k} \phi_{\delta,\delta_0}](s;y,v).\notag
\end{align}
One can argue as in Lemma \ref{lem:L-infinity estimate} (with the 
aid of \eqref{eq:uniform-growth-estimate} and moment estimates) to 
arrive at the conclusion that there exists a 
constant $C=C(\beta,\psi)$ and $p \in \mathbb{N}$ such that
\begin{align}
\sup_{\eps > 0}  \sup_{0\le s\le T}\Big(
E\Big[|| X_\eps[ \partial_{y_k} 
\phi_{\delta,\delta_0}](s;\cdot,\cdot)||_{L^\infty(\R^d\times \R)}^2\Big]\Big) 
\le \frac{C(\beta,  \psi )}{\delta_0^\frac{2(p-1)}{p}}. \label{eq:delta_0 estimate}
\end{align}
 
Now we repeatedly use integration by parts to obtain
\begin{align}
&\int_{\R_v} \int_{\Pi_T} J[\beta^\prime,\phi_{\delta,\delta_0}](s,y,v)
\Big(\int_{s-\delta_0}^s \varsigma_l(u_\eps(\sigma,y)-v)F_{\eps,k}^{\prime}(v)
\partial_{y_k} u_\eps(\sigma,y)\,d\sigma\Big) \,ds\,dy\,dv \notag \\
& 
=\int_{\R_v} \int_{\Pi_T} \int_{\sigma=s-\delta_0}^s \int_{\rd_x}
\int_{r=0}^T  \int_{|z|>0} \int_{\lambda=0}^1 \int_{\alpha=0}^1
\beta^{\prime\prime}(v(r,x,\alpha)+ \lambda \eta(x,v(r,x,\alpha);z)-v)\notag \\
&\hspace{5cm} \times \eta(x,v(r,x,\alpha);z) F_{\eps,k}^\prime(v)\phi_{\delta,\delta_0}(r,x;s,y)
\varsigma_l(u_\eps(\sigma,y)-v) \notag \\
& \hspace{6cm}\times \partial_{y_k} u_\eps(\sigma,y)\,d\alpha \,d\lambda\,\tilde{N}(dz,dr)\,dx
\,d\sigma \,ds\,dy\,dv \notag \\
& \quad 
=\int_{\R_v} \int_{\Pi_T} \int_{\sigma=s-\delta_0}^s 
\partial_v X_\eps[\phi_{\delta,\delta_0}](s;y,v)
\varsigma_l(u_\eps(\sigma,y)-v)\partial_{y_k} u_\eps(\sigma,y)
 \,d\sigma \,ds\,dy\,dv \notag \\
 & \quad
 = \int_{\R_v} \int_{\Pi_T} \int_{\sigma=s-\delta_0}^s  
 X_\eps[\phi_{\delta,\delta_0}](s;y,v)\varsigma_l^\prime (u_\eps(\sigma,y)-v)
 \partial_{y_k} u_\eps(\sigma,y)
 \,d\sigma \,ds\,dy\,dv \notag \\
&\quad
= \int_{\R_v} \int_{\Pi_T} \int_{\sigma=s-\delta_0}^s  
X_\eps[\phi_{\delta,\delta_0}](s;y,v) \partial_{y_k} \varsigma_l (u_\eps(\sigma,y)-v)
\,d\sigma \,ds\,dy\,dv \notag \\
&\quad 
=- \int_{\R_v} \int_{\Pi_T} \int_{\sigma=s-\delta_0}^s \partial_{y_k}  
X_\eps[\phi_{\delta,\delta_0}](s;y,v)  \varsigma_l (u_\eps(\sigma,y)-v)
 \,d\sigma \,ds\,dy\,dv \notag \\
& \quad
=- \int_{\R_v} \int_{\Pi_T} \int_{\sigma=s-\delta_0}^s 
X_\eps[\partial_{y_k} \phi_{\delta,\delta_0}](s;y,v)  \varsigma_l (u_\eps(\sigma,y)-v)
\,d\sigma \,ds\,dy\,dv. 
\label{l-4}
\end{align} 
Therefore, from \eqref{l-4} 
and \eqref{eq:delta_0 estimate}, we have
\begin{align}
\big|A_1^{l,\eps}(\delta,\delta_0)\big| &\le \sum_k 
\Big|E\Big[\int_{\R_v} \int_{\Pi_T} \int_{\sigma=s-\delta_0}^s 
X_\eps[\partial_{y_k} \phi_{\delta,\delta_0}](s;y,v)  \varsigma_l (u_\eps(\sigma,y)-v)
 \,d\sigma \,ds\,dy \,dv\Big]\Big| \notag \\
& \le C\delta_0  \frac{C(\beta, \phi,\delta )}{\delta_0^\frac{p-1}{p}} 
= C_1(\beta,\phi, \delta) \delta_0^\frac{1}{p} 
\rightarrow 0\quad \text{as $\delta_0 \rightarrow 0$.}
\notag
\end{align}
   
\textbf{Claim 2:}
\begin{align*}
A_2^{l,\eps}(\delta,\delta_0) \goto 0 \quad \text{as $\delta_0 \goto 0$.}
\end{align*}
\noindent{\it Justification:} Clearly,
\begin{align*}
& | A_2^{l,\eps}(\delta,\delta_0)|
\\ & \quad 
\le E\Big[ \int_v \int_{s=0}^T \int_{|y|\le C_{\phi}}\int_{s-\delta_0}^s 
||J[\beta^{\prime},\phi_{\delta,\delta_0}](s;\cdot,\cdot)||_{L^\infty(\R^d\times \R)} 
\varsigma_l(v-u_\eps(\sigma,y))\,\eps| \Delta u_\eps(\sigma,y)|
\,d\sigma\, dy\,ds\,dv\Big]\notag \\
&  \quad =E\Big[\int_{s=0}^T \int_{|y|\le C_{\phi}}
\int_{s-\delta_0}^s ||J[\beta^{\prime},
\phi_{\delta,\delta_0}](s;\cdot,\cdot)||_{L^\infty(\R^d\times \R)} 
\,\eps |\Delta u_\eps(\sigma,y)|\,d\sigma \,dy\,ds\Big]\notag \\
&\quad
\le \eps C(\psi)\int_{s=0}^T \int_{s-\delta_0}^s 
\Big(E\Big[||J[\beta^{\prime},\phi_{\delta,\delta_0}]
(s;\cdot,\cdot)||_{L^\infty(\R^d\times \R)}^2
\Big]\Big)^\frac{1}{2}
\Big( E \Big[  \int_{|y|\le C_{\phi}}|\Delta u_\eps(\sigma,y)|^2\Big]
\Big)^\frac{1}{2}\,d\sigma\, dy\,ds \notag \\
&\quad
\le C(\beta,\psi,\delta) \Big( \sup_{0\le s\le T} 
E \Big[||J[\beta^{\prime},\phi_{\delta,\delta_0}](s;\cdot,\cdot)||_{L^\infty(\R^d\times \R)}^2
\Big]\Big)^\frac{1}{2} \\
&\hspace{5cm}\times
\eps  \int_{s=0}^T\int_{\sigma=s-\delta_0}^s
\Big( E \Big[ \int_{|y|\le C_{\phi}} |\Delta u_\eps(\sigma,y)|^2\,dy\Big]\Big)^\frac{1}{2}
\,d\sigma\, ds\notag \\
&\quad
\le \frac{C(\beta,\eps,\psi,T)}{\delta_0^\frac{p-1}{p}} \,\delta_0\,\sup_{0\le r\le T} 
E \Big[|| \Delta u_\eps(r)||_2\Big]\notag \\
& \quad 
\le C(\beta,\psi,\eps,T) \delta_0^\frac{1}{p}
\quad \big(\text{as}~\sup_{0\le t\le T}E \Big[||\Delta u_\eps(t,\cdot)||_{L^2(\R^d)}\Big] 
\le C(\eps, T)~ \text{by Lemma \ref{lem:convergence_proof}}  \big)\notag\\
& \quad  \text{$\goto 0$ as $\delta_0\goto 0$.}
\end{align*} 
 
\textbf{Claim 3:}
\begin{align*}
A_3^{l,\eps}(\delta,\delta_0) \goto 0 \quad \text{as $\delta_0 \goto 0$.}
\end{align*}
\noindent{\it Justification:} First, we use integration by parts to conclude
\begin{align*}
&A_3^{l,\eps}(\delta,\delta_0)
\\ & = E\Big[\int_{\R_v} \int_{\Pi_{T}} J[\beta,\phi_{\delta,\delta_0}](s;y,v)
\Big( \int_{s-\delta_0}^s \int_{|z|>0}
\int_{\lambda=0}^1 \varsigma_l^{\prime\prime}(u_\eps(\sigma,y)-v 
+\lambda \eta_\eps(y,u_\eps(\sigma,y);z))
\\ &\hspace{5cm}\times (1-\lambda)\eta_\eps^2(y,u_\eps(\sigma,y);z)
\,d\lambda\,m(dz)\,d\sigma \Big) \,ds\,dy\,dv \Big]\notag \\
&= E\Big[\int_{\R_v} \int_{\Pi_{T}} \int_{\sigma=s-\delta_0}^s \int_{|z|>0}
\int_{\alpha=0}^1 J[\beta^{\prime\prime},\phi_{\delta,\delta_0}](s;y,v)  
\varsigma_l(u_\eps(\sigma,y)-v
+\lambda \eta_\eps(y,u_\eps(\sigma,y);z))  \notag \\
& \hspace{5cm} \times(1-\lambda)\ 
\eta_\eps^2(y,u_\eps(\sigma,y);z)
\,d\lambda\,m(dz)\,d\sigma  \,ds\,dy\,dv \Big],
\notag 
\end{align*} 
and therefore
\begin{align*}
&\big| A_3^{l,\eps}(\delta,\delta_0)\big| \\
& \le E\Big[ \int_{s=0}^T\int_{{|y|<C_\psi}}
\int_{\sigma=s-\delta_0}^s \int_{|z|>0}
||J[\beta^{\prime\prime},\phi_{\delta,\delta_0}](s;\cdot,\cdot)||_\infty 
\eta_\eps^2(y,u_\eps(\sigma,y);z)
\,m(dz)\,d\sigma  \,dy\,ds \Big]\notag \\
& \le C E\Big[\int_{s=0}^T\int_{{|y|<C_\psi}}
\int_{\sigma=s-\delta_0}^s  ||J[\beta^{\prime\prime},
\phi_{\delta,\delta_0}](s;\cdot,\cdot)||_\infty
g^2(y)(1+|u_\eps(\sigma,y)|^2)\,d\sigma \,dy\,ds \Big]\notag \\
& \le C\int_{s=0}^T\int_{\sigma=s-\delta_0}^s 
\Big(E \Big[||J[\beta^{\prime\prime},\phi_{\delta,\delta_0}](s;\cdot,\cdot)||_\infty^2\Big]
\Big)^\frac{1}{2} \Big( E\Big[\int_{|y|<C_\psi} g^4(y)(1+|u_\eps(\sigma,y)|^4
\,dy\Big] \Big)^\frac{1}{2}\,d\sigma \,ds \notag \\
& \le \frac{ C(\beta,\psi,\delta)}{\delta_0^\frac{p-1}{p}} 
\int_{s=0}^T \int_{\sigma=s-\delta_0}^s\Big( 1 +E\Big[||u_\eps(\sigma)||_4^4\Big] 
\Big)^{\frac 12}\,d\sigma \,ds \notag \\
& \le C(\beta,\psi,\delta) \delta_0^\frac{1}{p} 
T \Big( 1 + \sup_{\eps>0}\sup_{0\le t \le T} 
E \Big[||u_\eps(t,\cdot)||_4^4\Big]  \Big)^\frac{1}{2}  \goto 0 \quad 
\text{as $\delta_0\goto 0$.}
 \end{align*}
 
The next claim is about  $B ^{l,\eps}(\delta,\delta_0)$. 

{\bf Claim 4:}
\begin{align}
& \lim_{l\goto 0} \lim_{\delta_0 \goto 0} B ^{l,\eps}(\delta,\delta_0)\notag\\
 =& E\Big[\int_{\Pi_{T}}\int_{\R_x^d}
 \int_{|z|>0} \int_{\alpha=0}^1 \Big\{ \beta(v(r,x,\alpha)+ \eta(x,v(r,x,\alpha);z)
  -u_\eps(r,y)-\eta_\eps(y,u_\eps;z)) \notag \\
  & \hspace{3.5cm}-\beta(v(r,x,\alpha)-u_\eps(r,y)-\eta_\eps(y,u_\eps(r,y);z)) \notag \\
  &  \hspace{3.5cm}-\beta(v(r,x,\alpha)+ \eta(x,v(r,x,\alpha);z)-u_\eps(r,y))  \notag \\ 
  & \hspace{3.5cm}  + \beta(v(r,x,\alpha)-u_\eps(r,y))\Big\}\psi(r,y)\,\varrho_{\delta}(x-y)
  \,d\alpha\,m(dz)\,dx\,dy\,dr\Big] \notag
 \end{align}

\noindent{\it Justification:}  Note that, using integration 
by parts, $B ^{l,\eps}(\delta,\delta_0)$ can be written as
\begin{align*}
  &B ^{l,\eps}(\delta,\delta_0)\\
  &= E\Big[\int_{\Pi_{T}}\int_{\R_v}\int_{r=s-\delta_0}^s\int_{\R_x^d}\int_{|z|>0}
  \int_{\alpha=0}^1 \int_{\lambda=0}^1 \int_{\theta=0}^1 
 \beta^{\prime\prime}(v(r,x,\alpha)-v + \lambda \eta(x,v(r,x,\alpha);z))  \notag \\
 & \hspace{2.5cm} \times  \eta(x,v(r,x,\alpha);z) \eta_\eps(y,u_\eps(r,y);z) 
   \varsigma_l(u_\eps(r,y)+ \theta \eta_\eps(y,u_\eps(r,y);z)-v)  \notag \\
   & \hspace{4cm}\times\rho_{\delta_0}(r-s)\,\psi(s,y)\,\varrho_{\delta}(x-y)
   \,d\theta\,d\lambda\,d\alpha\,m(dz)\,dx\,dr\,dv\,dy\,ds\Big]\\
    &= E\Big[\int_{s=\delta_0}^T\int_{\R_y^d}\int_{\R_v}\int_{r=s-\delta_0}^s\int_{\R_x^d}\int_{|z|>0} 
     \int_{\alpha=0}^1 \int_{\lambda=0}^1 \int_{\theta=0}^1 
 \beta^{\prime\prime}(v(r,x,\alpha)-v + \lambda \eta(x,v(r,x,\alpha);z))  \notag \\
 & \hspace{2.5cm} \times  \eta(x,v(r,x,\alpha);z) \eta_\eps(y,u_\eps(r,y);z) 
   \varsigma_l(u_\eps(r,y)+ \theta \eta_\eps(y,u_\eps(r,y);z)-v)  \notag \\
   & \hspace{4cm}\times\rho_{\delta_0}(r-s)\,\psi(s,y)\,\varrho_{\delta}(x-y)
   \,d\theta\,d\lambda\,d\alpha\,m(dz)\,dx\,dr\,dv\,dy\,ds\Big]\\
   & \quad +E\Big[\int_{s=0}^{\delta_0}\int_{\R_y^d}\int_{\R_v}
   \int_{r=s-\delta_0}^s\int_{\R_x^d}\int_{|z|>0} 
    \int_{\alpha=0}^1 \int_{\lambda=0}^1 \int_{\theta=0}^1 
 \beta^{\prime\prime}(v(r,x,\alpha)-v + \lambda \eta(x,v(r,x,\alpha);z))  \notag \\
 & \hspace{2.5cm} \times  \eta(x,v(r,x,\alpha);z)\eta_\eps(y,u_\eps(r,y);z) 
   \varsigma_l(u_\eps(r,y)+ \theta \eta_\eps(y,u_\eps(r,y);z)-v)  \notag \\
   & \hspace{4cm}\times\rho_{\delta_0}(r-s)\,\psi(s,y)\,\varrho_{\delta}(x-y)
   \,d\theta\,d\lambda\,d\alpha\,m(dz)\,dx\,dr\,dv\,dy\,ds\Big]\\
    &= E\Big[\int_{r=0}^T\int_{\R_y^d}\int_{\R_v}\int_{s=r}^{r+\delta_0}\int_{\R_x^d}\int_{|z|>0}
     \int_{\alpha=0}^1 \int_{\lambda=0}^1 \int_{\theta=0}^1 
\beta^{\prime\prime}(v(r,x,\alpha)-v + \lambda \eta(x,v(r,x,\alpha);z))  \notag \\
 & \hspace{2.5cm} \times  \eta(x,v(r,x,\alpha);z) \eta_\eps(y,u_\eps(r,y);z) 
   \varsigma_l(u_\eps(r,y)+ \theta \eta_\eps(y,u_\eps(r,y);z)-v)  \notag \\
   & \hspace{3cm}\times\rho_{\delta_0}(r-s)\,\psi(s,y)\,\varrho_{\delta}(x-y)
   \,d\theta\,d\lambda\,d\alpha\,m(dz)\,dx\,ds\,dv\,dy\,dr\Big]+o(\delta_0),
 \end{align*}  
where we have used Fubini's theorem to infer the last line. 
Hence
\begin{align*}
& \Big|  B ^{l,\eps}(\delta,\delta_0) 
  - E\Big[\int_{\Pi_{T}}\int_{\R_v}\int_{\R_x^d}\int_{|z|>0} 
   \int_{\alpha=0}^1 \int_{\lambda=0}^1 \int_{\theta=0}^1 
  \beta^{\prime\prime}(v(r,x,\alpha)-v  
 + \lambda \eta(x,v(r,x,\alpha);z)) \\
 &\hspace{3cm} \times  \eta(x,v(r,x,\alpha);z) \eta_\eps(y,u_\eps(r,y);z)  
 \varsigma_l(u_\eps(r,y)+ \theta \eta_\eps(y,u_\eps(r,y);z)-v)   \\
 & \hspace{6cm}\times \,\psi(r,y)\,\varrho_{\delta}(x-y)
 \,d\theta\,d\lambda\,d\alpha\,m(dz)\,\,dx\,dv\,dy\,dr\Big]\Big|\\
&= \Big|  B ^{l,\eps}(\delta,\delta_0)
 - E\Big[\int_{\Pi_{T}}\int_{\R_v}
\int_{s=r}^{r+\delta_0}\int_{\R_x^d}\int_{|z|>0} 
    \int_{\alpha=0}^1 \int_{\lambda=0}^1 \int_{\theta=0}^1 
  \beta^{\prime\prime}(v(r,x,\alpha)-v  + \lambda \eta(x,v(r,x,\alpha);z))  \\
 &\hspace{2cm} 
 \times  \eta(x,v(r,x,\alpha);z)\eta_\eps(y,u_\eps(r,y);z)  \varsigma_l(u_\eps(r,y)+ \theta \eta_\eps(y,u_\eps(r,y);z)-v)   \\
   & \hspace{3.5cm}\times \,\rho_{\delta_0}(r-s)\,\psi(r,y)
   \,\varrho_{\delta}(x-y)\,d\theta\,d\lambda\,d\alpha\,m(dz)\,\,dx\,ds\,dv\,dy\,dr\Big]\Big|  \\
   &\le E\Big[\int_{r=0}^T\int_{\R_y^d}\int_{\R_v}\int_{s=r}^{r+\delta_0}
   \int_{\R_x^d}\int_{|z|>0}  \int_{\alpha=0}^1 \int_{\lambda=0}^1 \int_{\theta=0}^1 
\beta^{\prime\prime}(v(r,x,\alpha)-v + \lambda \eta(x,v(r,x,\alpha);z))   \\
 & \hspace{2.5cm} \times| \eta(x,v(r,x,\alpha);z)|\,  |\eta_\eps(y,u_\eps(r,y);z) |
   \varsigma_l(u_\eps(r,y)+ \theta \eta_\eps(y,u_\eps(r,y);z)-v) \rho_{\delta_0}(r-s) \notag \\
   & \hspace{3cm}\times|\psi(s,y)-\psi(r,y)|\,\varrho_{\delta}(x-y)
   \,d\theta\,d\lambda\,d\alpha\,m(dz)\,dx\,ds\,dv\,dy\,dr\Big]
    + o(\delta_0)\notag\\
   &\le E\Big[\int_{r=0}^T\int_{\R_y^d}\int_{s=r}^{r+\delta_0}\int_{\R_x^d}\int_{|z|>0} \int_{\alpha=0}^1
| \eta(x,v(r,x,\alpha);z)|\, ||\beta^{\prime\prime}||_{\infty}  |\eta_\eps(y,u_\eps(r,y);z) |\notag\\
&\hspace{2.5cm}\times\rho_{\delta_0}(r-s)\,|r-s|\, ||\partial_t\psi||_{\infty}
\,\varrho_{\delta}(x-y)\,d\alpha\,m(dz)\,dx\,ds\,dv\,dy\,dr\Big]
    + o(\delta_0)\notag\\
 & \le \delta_0 C(\beta^{\prime\prime},\psi) E\Big[\int_{\Pi_{T}}\int_{\R_x^d}\int_{|z|>0} \int_{\alpha=0}^1
  |\eta(x,v(r,x,\alpha);z) \eta_\eps(y,u_\eps(r,y);z)| \notag
   \\&\hspace{7cm} \times\varrho_{\delta}(x-y)\,d\alpha\,m(dz)\,dx\,dy\,dr\Big] + o(\delta_0) \notag \\
   & \le \delta_0C(\beta^{\prime\prime},\psi) 
   \Big( 1+ \sup_{0\le r\le T} E\Big[||v(r,\cdot,\cdot)||_2^2\Big]
   + \sup_{\eps>0}\sup_{0\le r \le T} E\Big[||u_\eps(r,\cdot)||_2^2\Big] \Big) 
  + o(\delta_0). \notag 
\end{align*}
and so
\begin{align}
& \lim_{\delta_0 \goto 0} B ^{l,\eps}(\delta,\delta_0)\notag \\
 &=E\Big[\int_{\Pi_{T}}\int_{\R_v}\int_{\R_x^d}\int_{|z|>0} 
 \int_{\alpha=0}^1 \Big( \beta(v(r,x,\alpha)+ \eta(x,v(r,x,\alpha);z)-v)
 -\beta(v(r,x,\alpha)-v)\Big) \notag \\
  & \hspace{4cm}
  \times\Big( \varsigma_l(u_\eps(r,y)+\eta_\eps(y,u_\eps(r,y);z)-v)-\varsigma_l(u_\eps(r,y)-v)\Big) \notag \\
 & \hspace{5cm}\times\,\psi(r,y)\,\varrho_{\delta}(x-y)
 \,d\alpha\,m(dz)\,dx\,dv\,dy\,dr\Big] \notag \\
  &\equiv E\Big[\int_{\Pi_{T}}\int_{\R_v}\int_{\R_x^d}
  \int_{|z|>0}\int_{\alpha=0}^1 \int_{\lambda=0}^1 \int_{\theta=0}^1
 \eta(x,v(r,x,\alpha);z)  
 \beta^{\prime\prime}(v(r,x,\alpha)-v + \lambda \eta(x,v(r,x,\alpha);z)) \notag \\
 &\hspace{5cm} \times \eta_\eps(y,u_\eps(r,y);z)  
 \varsigma_l(u_\eps(r,y)+ \theta \eta_\eps(y,u_\eps(r,y);z)-v)  \notag \\
 & \hspace{5.5cm}
 \times\psi(r,y)\,\varrho_{\delta}(x-y)
 \,d\theta\,d\lambda\,d\alpha\,m(dz)\,dx\,dv\,dy\,dr\Big],
 \label{limit in delta_0}
\end{align} 
where we have first re-written the terms using 
the fundamental theorem of integral calculus and then applied integration 
by parts with respect to $v$. It is now routine to 
pass to the limit $l \rightarrow 0$ in \eqref{limit in delta_0}, and 
hence the conclusion follows.
\end{proof}
 
Next, we consider the term $I_5 + J_5$ and prove the following lemma.
\begin{lem}\label{stochastic_lemma_4}
Assume that $ \vartheta \goto 0, \delta \goto 0$ 
and $\frac{\vartheta}{\delta} \goto 0$. Then
\begin{align*} 
\lim_{\frac{\vartheta}{\delta} \downarrow 0, \, \, 
\vartheta \downarrow 0,\, \,\delta\downarrow 0}
\big[\lim_{\eps_n \downarrow 0} 
\lim_{l\downarrow 0}\lim_{\delta_0\downarrow 0}
\, \, (I_5 + J_5)\big] = 0. 
 \end{align*}
 \end{lem} 
 
\begin{proof}
Note that 
\begin{align}
& \Big| I_5 -E\Big[\int_{s=0}^T \int_{\R_y^d}\int_{\R_x^d}
\int_{\alpha=0}^1  \int_{\R_k}  F^\beta(v(s,x,\alpha),k) 
\cdot\grad_x \varrho_\delta(x-y)\,\psi(s,y)
\varsigma_l(u_\eps(s,y)-k)
\,dk\,d\alpha\,dx\,dy\,ds\Big] \Big| \notag \\
& = \Big| E\Big[\int_{s=0}^T \int_{t=0}^T\int_{\R_y^d}\int_{\R_x^d}\int_{\alpha=0}^1  \int_{\R_k} 
\Big(F^\beta(v(t,x,\alpha),k)-F^\beta(v(s,x,\alpha),k) \Big)\cdot
\grad_x \varrho_\delta(x-y)\,\psi(s,y) \rho_{\delta_0}(t-s) \notag \\
&  \hspace{6cm} \times \varsigma_l(u_\eps(s,y)-k)\,dk\,d\alpha\,dx\,dy\,dt\,ds\Big] \notag \\
& \quad 
+ E\Big[\int_{s=0}^T \int_{t=0}^T\int_{\R_y^d}\int_{\R_x^d} \int_{\alpha=0}^1  \int_{\R_k} 
 F^\beta(v(s,x,\alpha),k) \cdot
\grad_x \varrho_\delta(x-y)\,\psi(s,y) \rho_{\delta_0}(t-s) \notag \\
&  \hspace{6cm} \times \varsigma_l(u_\eps(s,y)-k)
\,dk\,d\alpha\,dx\,dy\,dt\,ds\Big] \notag \\
 & \quad 
-E\Big[\int_{s=0}^T \int_{\R_y^d}\int_{\R_x^d} \int_{\alpha=0}^1  
\int_{\R_k}  F^\beta(v(s,x,\alpha),k)\cdot \grad_x \varrho_\delta(x-y)\,\psi(s,y)\,
\varsigma_l(u_\eps(s,y)-k)\,dk\,d\alpha\,dx\,dy\,ds\Big] \Big|  \notag \\
& \le  E\Big[\int_{s=0}^T \int_{t=0}^T\int_{\R_y^d}\int_{\R_x^d} 
\int_{\alpha=0}^1  \int_{\R_k} \big|F^\beta(v(t,x,\alpha),k)
-F^\beta(v(s,x,\alpha),k) \big|
|\grad_x \varrho_\delta(x-y)|\psi(s,y) \rho_{\delta_0}(t-s) \notag \\
& \hspace{6cm} \times \varsigma_l(u_\eps(s,y)-k)
\,dk\,d\alpha\,dx\,dy\,dt\,ds\Big] \notag \\
& \quad + \Big| E\Big[\int_{s=0}^T \int_{\R_y^d}\int_{\R_x^d} 
\int_{\alpha=0}^1  \int_{\R_k}  F^\beta(v(s,x,\alpha),k)\cdot \grad_x \varrho_\delta(x-y)\,\psi(s,y)
 \Big( 1- \int_{t=0}^T \rho_{\delta_0}(t-s)\,dt \Big) \notag \\
 & \hspace{6cm} \times \varsigma_l(u_\eps(s,y)-k)\,dk\,d\alpha\,dx\,dy\,ds\Big] \Big| \notag \\
 & \le E\Big[\int_{s=\delta_0}^T \int_{t=0}^T\int_{\R_y^d}\int_{\R_x^d}
 \int_{\alpha=0}^1  \int_{\R_k} \big|F^\beta(v(t,x,\alpha),k)-F^\beta(v(s,x,\alpha),k) \big|
|\grad_x \varrho_\delta(x-y)|\psi(s,y) \rho_{\delta_0}(t-s) \notag \\
&  \hspace{6cm} \times \varsigma_l(u_\eps(s,y)-k)
\,dk\,d\alpha\,dx\,dy\,dt\,ds\Big] + o(\delta_0)\notag \\
& \quad +  E\Big[\int_{s=0}^{\delta_0} \int_{\R_y^d}\int_{\R_x^d}
\int_{\alpha=0}^1  \int_{\R_k} \big| F^\beta(v(s,x,\alpha),k)\cdot 
\grad_x \varrho_\delta(x-y)\big|\psi(s,y)
\varsigma_l(u_\eps(s,y)-k)\,dk\,d\alpha\,dx\,dy\,ds\Big] \notag \\
&\qquad  (\text{we have used the fact that}~  \int_{t=0}^T \rho_{\delta_0}(t-s)\,dt \le1, 
\,\text{equality holds if}\, s\ge \delta_0)\notag\\
 & \le  C E\Big[\int_{s=\delta_0}^T \int_{t=0}^T\int_{\R_y^d}
 \int_{\R_x^d} \int_{\alpha=0}^1  \int_{\R_k}  \big|v(t,x,\alpha)-
 v(s,x,\alpha)\big|\big(1 + |v(t,x,\alpha)|^p + |v(s,x,\alpha)|^p\big)
|\grad_x \varrho_\delta(x-y)| \notag \\
  &  \hspace{6cm} \times  \psi(s,y) \rho_{\delta_0}(t-s)
 \varsigma_l(u_\eps(s,y)-k)\,dk\,d\alpha\,dx\,dy\,dt\,ds\Big]\notag +o(\delta_0) \\
& \quad +  E\Big[\int_{s=0}^{\delta_0} \int_{\R_y^d}
\int_{\R_x^d} \int_{\alpha=0}^1  \int_{\R_k} \big| F^\beta(v(s,x,\alpha),k)\cdot
\grad_x \varrho_\delta(x-y)\big|\psi(s,y)
  \varsigma_l(u_\eps(s,y)-k)\,dk\,d\alpha\,dx\,dy\,ds\Big] \notag \\
 &\qquad  (\text{we have used the Lipschitz 
 continuity of}\, F^{\beta}(\cdot, k)\, \text{in above} )\notag\\
 & \le  C \Big(E\Big[ \int_{s=\delta_0}^T \int_{t=0}^T\int_{\R_x^d} 
\int_{\alpha=0}^1    |v(t,x,\alpha)
-v(s,x,\alpha)|^2 \rho_{\delta_0}(t-s)
\,d\alpha\,dx\,dt\,ds\Big] \Big)^\frac{1}{2} + o(\delta_0) \notag \\
&  \le  C \Big( E \Big[\int_{r=0}^1 \int_{\Pi_T} \int_{\alpha=0}^1 
|v(t+ \delta_0\,r ,x,\alpha)-v(t,x,\alpha)|^2 
\rho(-r)\,d\alpha \,dt\,dx\,dr\Big]\Big)^\frac{1}{2} + o(\delta_0) \notag.
\end{align}
Note that $\underset{\delta_0\downarrow 0 }\lim\,\int_{t=0}^T\int_{\R_x^d} 
\int_{\alpha=0}^1 |v(t+\delta_0 r,x,\alpha)-v(t,x,\alpha)|^2\,
\,d\alpha\,dx\,dt \rightarrow 0$ almost surely for all $r\in [0,1]$. 
Therefore, by the bounded convergence theorem,
\begin{align*}
\lim_{\delta_0\downarrow 0}E\Big[\int_{t=0}^T\int_{r=0}^1
\int_{\R_x^d} \int_{\alpha=0}^1  
|v(t+\delta_0 r,x,\alpha)-v(t,x,\alpha)|^2
\rho(-r)\,d\alpha\,dx\,dr\,dt\Big]=0.
\end{align*} 
This implies that
\begin{align*}
&\lim_{\delta_0 \goto 0} I_5\\
& = E\Big[\int_{s=0}^T 
\int_{\R_y^d}\int_{\R_x^d}\int_{\alpha=0}^1 
\int_{\R_k}  F^\beta(v(s,x,\alpha),k)\cdot \grad_x 
\varrho_\delta(x-y)\,\psi(s,y)
\varsigma_l(u_\eps(s,y)-k)\,dk\,d\alpha\,dx\,dy\,ds\Big]\notag\\
&= -E\Big[\int_{s=0}^T \int_{\R_y^d}\int_{\R_x^d} 
\int_{\alpha=0}^1  \int_{\R_k}  
F^\beta(v(s,x,\alpha),u_{\eps}(s,y)-k)\cdot \grad_y \varrho_\delta(x-y)
\,\psi(s,y) \varsigma_l(k)\,dk\,d\alpha\,dx\,dy\,ds\Big].
\end{align*} 
In a similar manner, we find
\begin{align*}
\lim_{\delta_0 \goto 0} J_5= E \Big[\int_{s=0}^T
\int_{\R_x^d}\int_{\R_y^d} \int_{\R_k}\int_{\alpha=0}^1 
F^\beta(u_\eps(s,y),v(s,x,\alpha)-k)\cdot
\grad_y \varrho_\delta(x-y)
\, \psi(s,y)\varsigma_l(k) \,d\alpha\,dk\,dx\,dy\,ds\Big].
\end{align*}
Note that 
\begin{align*}
&I_5 + J_5 \\= &
E\Big[ \int_{\Pi_T}\int_{\Pi_T} \int_{\R_k}\int_{\alpha=0}^1 
\Big( - F^\beta(v(t,x,\alpha),u_\eps(s,y)-k) 
+ F^\beta(u_\eps(s,y),v(t,x,\alpha)-k)\Big)\cdot \grad_y \varrho_\delta(x-y) \notag \\
& \hspace{4.5cm} \times 
\psi(s,y)\rho_{\delta_0}(t-s) \varsigma_l(k)\,d\alpha\,dk\,dx\,dt\,dy\,ds\Big].
\end{align*}
Hence,
\begin{align}
&\lim_{\delta_0 \goto 0} (I_5 + J_5) \notag\\
= & E \Big[\int_{\Pi_T}\int_{\R_x^d} \int_{\R_k}\int_{\alpha=0}^1 
\Big( -F^\beta(v(t,x,\alpha),u_\eps(t,y)-k) 
+ F^\beta(u_\eps(t,y),v(t,x,\alpha)-k)\Big) \cdot \grad_y \varrho_\delta(x-y) \notag \\
& \hspace{5cm}  \times \psi(t,y) 
\varsigma_l(k)\,d\alpha\,dk\,dx\,dt\,dy\Big].\label{eq:limit-in-time} 
\end{align}
There exists $p\in N$ such that for all $a, b,c\in \R$ 
\begin{align}\label{eq:local-lip-flux} 
|F^{\beta}(a,b)-F^{\beta}(a,c)|\le K |b-c|(1+|b|^p+|c|^p) 
~\text{and}~|F^{\beta}(b,a)-F^{\beta}(c,a)|\le K |b-c|(1+|b|^p+|c|^p).
\end{align}
In view of \eqref{eq:local-lip-flux}, we can routinely 
pass to the limit $l\goto 0$ in \eqref{eq:limit-in-time} and conclude 
\begin{align*}
&\lim_{l\goto 0} \lim_{\delta_0 \goto 0} (I_5 + J_5)\\
&=E \Big[\int_{\Pi_T}\int_{\R_x^d} \int_{\alpha=0}^1  \Big(  
F^\beta\big(u_\eps(t,y),v(t,x,\alpha)\big)
-F^\beta\big(v(t,x,\alpha),u_\eps(t,y)\big)\Big)\cdot \grad_y \varrho_\delta(x-y)  \\
& \hspace{4.5cm} \times  \psi(t,y)\,d\alpha\,dx\,dt\,dy\Big].
\end{align*} 
Note that
\begin{align}
|F^{\beta_\vartheta}_k(a,b)-F^{\beta_\vartheta}_k(b,a)|
& \le  |F_k^{\beta_\vartheta}(a,b)-\text{sign}(a-b)(F_k(a)-F_k(b))| \notag\\
&\quad
+ |F_k^{\beta_\vartheta}(b,a)-\text{sign}(b-a)(F_k(b)-F_k(a))| \label{eq:mod_approx_1}
\le C \vartheta (1+|a|^p+|b|^p),
\end{align} 
and therefore 
\begin{align*}
&\Big|  E\Big[\int_{\R_y^d}\int_{\Pi_T} \int_{\alpha=0}^1  
\Big\{F^\beta(u_\eps(t,y),v(t,x,\alpha))- F^\beta(v(t,x,\alpha),u_\eps(t,y))\,\Big\}
\cdot \grad_y \varrho_{\delta}(x-y) 
\psi(t,y)\,d\alpha\,dx\,dt\,dy \Big]\Big |
\\ & \quad
\le  \vartheta \, C  E\Big[\int_{\R_y^d}\int_{\Pi_T} 
\int_{\alpha=0}^1  \big(1+|u_\eps(t,y)|^p+|v(t,x,\alpha)|^p\big) 
|\nabla_y \varrho_{\delta} (x-y)|\psi(t,y) 
\,d\alpha\,dx\,dt\,dy\Big]\\
& \quad \le \frac{\vartheta}{\delta} C\goto 0 
\quad \text{when $(\vartheta, \frac \vartheta\delta,\delta)\rightarrow (0,0,0)$.}
\end{align*}
Hence the lemma follows.
\end{proof}
\begin{lem} \label{stochastic_lemma_5}
It holds that
\begin{align*}
J_6  &\underset{\delta_0 \goto 0}{\rightarrow}  E \Big[\int_{\Pi_T}\int_{\R_x^d}\int_{\alpha=0}^1 
\int_{\R_k}  F^\beta(u_\eps(s,y),k)\cdot\grad_y \psi(s,y)\, \varrho_\delta(x-y) 
 \varsigma_l(v(s,x,\alpha)-k)
 \,dk \,d\alpha\,dx\,dy\,ds\Big] \\
&\underset{l \goto 0}{\rightarrow}  
E \Big[\int_{\Pi_T}\int_{\R_x^d} \int_{\alpha=0}^1 F^\beta(u_\eps(s,y),v(s,x,\alpha))\cdot \grad_y \psi(s,y)
\, \varrho_\delta(x-y)\,d\alpha\,dx\,dy\,ds\Big] \\
 &\underset{\eps \goto 0}{\rightarrow}  E \Big[\int_{\Pi_T}\int_{\R_x^d} 
 \int_{\alpha=0}^1 \int_{\gamma=0}^1  F^\beta(u(s,y,\gamma),v(s,x,\alpha)) \cdot \grad_y \psi(s,y)\,
\varrho_\delta(x-y) \,d\gamma\,d\alpha\,dx\,dy\,ds\Big],
\end{align*} 
and
\begin{align*}
\lim_{(\vartheta,\delta)\goto (0,0)}& 
E \Big[\int_{\Pi_T}\int_{\R_x^d} \int_{\alpha=0}^1 \int_{\gamma=0}^1  
F^\beta(u(s,y,\gamma),v(s,x,\alpha)) \cdot \grad_y \psi(s,y)\,
 \varrho_\delta(x-y) \,d\gamma\,d\alpha\,dx\,dy\,ds\Big]\\
& = E \Big[\int_{\Pi_T} \int_{\alpha=0}^1 \int_{\gamma=0}^1 
F(u(s,y,\gamma),v(s,y,\alpha)) \cdot \grad_y \psi(s,y)
 \,d\gamma\,d\alpha\,dy\,ds\Big].
 \end{align*}
\end{lem}
  
\begin{proof} The first part of the proof is divided into three steps.

\textbf{Step 1:} We will justify the $\delta_0\to 0$ limit.  Define 
\begin{align}
\mathcal{B}_1& :=\Big| E \Big[\int_{\Pi_T}\int_{\Pi_T}\int_{\R_k} 
\int_{\alpha=0}^1 F^\beta(u_\eps(s,y), v(t,x,\alpha)-k) \cdot\grad_y\psi(s,y)\,\rho_{\delta_0}(t-s)
 \varrho_{\delta}(x-y) \notag \\
 & \hspace{4cm} \times \varsigma_l(k)\,d\alpha\,dk\,dy\,ds\,dx\,dt\Big]\notag \\
& \quad 
- E\Big[\int_{\Pi_T}\int_{\R_y^d} \int_{\R_k} \int_{\alpha=0}^1  
F^\beta(u_\eps(s,y),k)\cdot \grad_y \psi(s,y)\,\varrho_{\delta}(x-y)
 \varsigma_l(v(s,x,\alpha)-k) d\alpha\,dk \,dy\,dx\,ds \Big] \Big| \notag \\
 & = \Big| E \Big[\int_{\Pi_T}\int_{\Pi_T}\int_{\R_k} 
 \int_{\alpha=0}^1  \Big( F^{\beta}(u_{\eps}(s,y), v(t,x,\alpha)-k)
 - F^{\beta}(u_{\eps}(s,y), v(s,x,\alpha)-k)\Big)\cdot  \grad_y\psi(s,y)
  \notag \\
 & \hspace{4cm} \times  \rho_{\delta_0}(t-s) \varrho_{\delta}(x-y) \varsigma_l(k) \,d\alpha\,dk\,dy\,ds\,dx\,dt\Big]\notag \\
& \quad
- E \Big[\int_{\Pi_T}\int_{\R_y^d} \int_{\R_k} \int_{\alpha=0}^1  
F^\beta(u_\eps(s,y),v(s,x,\alpha)-k)\cdot \grad_y \psi(s,y)\,\varrho_{\delta}(x-y)  
\notag \\ & \hspace{4cm} \times
\Big( 1- \int_{t=0}^T \rho_{\delta_0}(t-s)\,dt \Big) 
\varsigma_l(k) d\alpha\,dk \,dy\,dx\,ds \Big]  \Big| \notag\\
 &\le C E\Big[\int_{\Pi_T}\int_{\Pi_T}\int_{\R_k} 
 \int_{\alpha=0}^1 \varsigma(k)\,|\grad_y\psi(s,y)| \,\rho_{\delta_0}(t-s)
 \varrho_{\delta}(x-y)\notag\\
 & \hspace{3cm}\times 
 |v(s,x,\alpha)-v(t,x,\alpha)|\big(1+|v(s,x,\alpha)|^p +
| v(t,x,\alpha)|^p + |k|^p\big) \,d\alpha\,dk\,dy\,ds\,dx\,dt\Big]\notag \\
&\quad +E \Big[\int_{s=0}^{\delta_0} \int_{\R_x^d}\int_{\R_y^d} \int_{\R_k} 
\int_{\alpha=0}^1 \big| F^\beta(u_\eps(s,y),k)\cdot\grad_y \psi(s,y)\big|
\,\varrho_{\delta}(x-y) \varsigma_l(v(s,x,\alpha)-k) 
d\alpha\,dk \,dy\,dx\,ds\Big]\notag\\
& \big(\text{we used the inequality} \eqref{eq:local-lip-flux}\big)\notag \\
&\le C E \Big[\int_{s=\delta_0}^T \int_{\R_x^d}\int_{t=0}^T
\int_{\R_k}\int_{\alpha=0}^1 |v(s,x,\alpha)-v(t,x,\alpha)|
\big(1+|v(s,x,\alpha)|^p +| v(t,x,\alpha)|^p + |k|^p\big) \notag \\
& \hspace{4cm} \times \varsigma_l(k)\rho_{\delta_0}(t-s)\,d\alpha\,dk\,dt\,dx\,ds\Big]\notag \\
& \quad
+ CE \Big[\int_{s=0}^{\delta_0} \int_{\R_x^d}\int_{t=0}^T\int_{\R_k} 
\int_{\alpha=0}^1  |v(s,x,\alpha)-v(t,x,\alpha)|\big(1+|v(s,x,\alpha)|^p 
+| v(t,x,\alpha)|^p + |k|^p\big) \notag \\
& \hspace{4cm} \times \varsigma_l(k)\rho_{\delta_0}(t-s)\,d\alpha\,dk\,dt\,dx\,ds\Big]
+o(\delta_0)\notag\\
&\le C(\beta, \psi, l) \Big(E \Big[\int_{s=\delta_0}^T
\int_{t=0}^T \int_{\R^d} \int_{\alpha=0}^1 |v(s,x,\alpha)-v(t,x,\alpha)|^2  
\rho_{\delta_0}(t-s) d\alpha\,\,dx\,dt\,ds\Big] \Big)^\frac{1}{2} + o(\delta_0)\notag\\
& = C(\beta, \psi, l) \Big(E \Big[\int_{r=0}^1 \int_{\R^d}
\int_{t=0}^T \int_{\alpha=0}^1  
|v(t+\delta_0\,r,x,\alpha)-v(t,x,\alpha)|^2  
\rho(-r) d\alpha\,dt\,dx\,dr\Big]\Big)^{\frac 12} 
+o(\delta_0) \notag ,
\end{align} 
where we have used the Schwartz's inequality with 
respect to the measure $\rho_{\delta_0}(t-s)\, d\alpha\, dx\, dt\, ds\,dP(\omega)$.
We recall that $\underset{\delta_0\downarrow 0 }\lim\,  
\int_{t=0}^T\int_{\R_x^d}\int_{\alpha=0}^1 
 |v(t+\delta_0 r,x,\alpha)-v(t,x,\alpha)|^2\,
\,d\alpha\,dx\,dt \rightarrow 0$ almost surely, for all $ r\in [0,1]$. 
Therefore, by the bounded convergence theorem,
\begin{align*}
\lim_{\delta_0\downarrow 0} E\Big[\int_{r=0}^1 
\int_{\R_x^d}\int_{t=0}^T \int_{\alpha=0}^1  
|v(t+\delta_0\,r,x,\alpha)-v(t,x,\alpha)|^2 
\rho(-r) d\alpha\,dt\,dx\,dr\Big] =0,
\end{align*}
and therefore the first step follows.

\textbf{Step 2:} We will justify the $l\to 0$ limit. Let
\begin{align*}
\mathcal{B}_2 :=& E \Big[\int_{\Pi_T}\int_{\R_x^d} 
\int_{\alpha=0}^1  \int_{\R_k}  F^\beta(u_\eps(s,y),k)\cdot \grad_y \psi(s,y) \,
\varrho_\delta(x-y)  \varsigma_l(v(s,x,\alpha)-k)\,dk\,d\alpha\,dx\,dy\,ds\Big] \\
 & \qquad
-E \Big[\int_{\Pi_T}\int_{\R_x^d} \int_{\alpha=0}^1  
 F^\beta(u_\eps(s,y),v(s,x,\alpha))\cdot \grad_y \psi(s,y)\,
\varrho_\delta(x-y)\,d\alpha\,dx\,dy\,ds\Big] \\
&=E \Big[\int_{\Pi_T}\int_{\R_x^d} \int_{\alpha=0}^1 
\int_{\R_k} \Big( F^\beta(u_\eps(s,y),k) 
- F^\beta(u_\eps(s,y),v(s,x,\alpha))\Big)\cdot  \grad_y \psi(s,y) \\
&\hspace{5cm}\times \varrho_\delta(x-y) 
\varsigma_l(v(s,x,\alpha)-k)\,dk\,d\alpha\,dx\,dy\,ds\Big].
\end{align*}
Therefore, by \eqref{eq:local-lip-flux}, there exists a 
natural number $p$ such that
\begin{align*}
|\mathcal{B}_2|& \le C  E \Big[\int_{\Pi_T}\int_{\R_y^d}
\int_{\alpha=0}^1  \int_{\R_k}  |v(s,x,\alpha)-k|\big( 1+ |v(s,x,\alpha)|^p 
+ | v(s,x,\alpha)-k|^p\big)|\grad_y \psi(s,y)|  \\
 &\hspace{5cm} 
\times \varrho_\delta(x-y) \varsigma_l(v(s,x,\alpha)-k)
\,dk\,d\alpha\,dx\,dy\,ds\Big] \\
& \le C ||\grad_y\psi(s,\cdot)||_{\infty}  
E \Big[\int_0^T\int_{\R_x^d}\int_{\alpha=0}^1  \int_{\R_k}  |v(s,x,\alpha)-k|( 1+ |v(s,x,\alpha)|^p )\varsigma_l(v(s,x,\alpha)-k)\,dk\,d\alpha\,dx\,ds \Big] \\
&\quad  + C ||\grad_y\psi(s,\cdot)||_{\infty}E \Big[\int_0^T\int_{\R_x^d} 
\int_{|y|\le C_\psi}\int_{\alpha=0}^1 \int_{\R_k}  
|v(s,x,\alpha)-k|^{p+1}\varsigma_l(v(s,x,\alpha)-k) \\
 & \hspace{6cm} \times
\varrho_\delta(x-y)\,dk\,d\alpha\,dy\,dx\,ds \Big] \\
& \le C ||\grad_y\psi(s,\cdot)||_{\infty}  E \Big[\int_0^T\int_{\R_x^d}\int_{\alpha=0}^1  
\int_{\R_k}  |v(s,x,\alpha)-k|
( 1+ |v(s,x,\alpha)|^p )\varsigma_l(v(s,x,\alpha)-k)\,dk\,d\alpha\,dx\,ds \Big] \\
&\hspace{6cm}   + C\,T\, ||\grad_y\psi(s,\cdot)||_{\infty}\,l^{p+1} \int_{|y|
\le C_\psi}\int_{\R_x^d}\varrho_{\delta}(x-y)\,dx\,dy\\
& \le  C\,T\, ||\grad_y\psi(s,\cdot)||_{\infty}\,l 
\Big( l^p + \sup_{0\le s\le T}E\Big[||v(s,\cdot,\cdot)||_p^p\Big] \Big)
\goto 0\quad \text{as $l\goto 0$.}
\end{align*}

\textbf{Step 3:} We now justify the passage to the limit $\eps_n \goto 0$. 
Let 
$$ G_x(s,y,\omega,\xi)=\int_{\R_x^d}\int_{\alpha=0}^1  
F^\beta(\xi,v(s,x,\alpha))\cdot \grad_y \psi(s,y)\,
\varrho_\delta(x-y)\,d\alpha\,dx. 
$$
As in Lemma \ref{stochastic_lemma_2}, $G_x(s,y,\omega,\xi)$ is 
a Caratheodory function for every $x\in \R^d$ and 
$\{ G_x(s,y,\omega, u_{\eps_n}(s,y))\}_n$ is bounded and 
uniformly integrable. This allows us to conclude that
\begin{align}
& \lim_{\eps_n\rightarrow 0}E \Big[\int_{\R_y^d}
\int_{\Pi_T}\int_{\alpha=0}^1  F^\beta(u_{\eps_n}(s,y),v(s,x,\alpha))\cdot 
\grad_y \psi(s,y)\,\varrho_\delta(x-y)\,d\alpha\,dx\,ds\,dy\Big]\notag \\
&= E\Big[\int_{\R_y^d}\int_{\Pi_T}\int_{\alpha=0}^1 
\int_{\gamma=0}^1 F^\beta(u(s,y,\gamma),v(s,x,\alpha))\cdot \grad_y \psi(s,y)\,
 \varrho_\delta(x-y)\,d\gamma\,d\alpha\,dx\,ds\,dy\Big].\notag 
\end{align} 
This completes the  proof of the first half of the lemma.
\vspace{.3cm}

To prove the second half of the lemma, let us denote
\begin{align*}
\mathcal{B}(\vartheta,\delta):=
&\Big|E\Big[ \int_{\Pi_T}\int_{\R_y^d} \int_{\alpha=0}^1  
\int_{\gamma=0}^1 F^{\beta_\vartheta}(u(s,y,\gamma),v(s,x,\alpha))\cdot
\grad_y \psi(s,y)\,\varrho_{\delta}(x-y)\, d\gamma \,d\alpha\,dy\,dx\,ds\Big]\notag \\
&\quad -E \Big[\int_{\Pi_T}\int_{\R_x^d}\int_{\alpha=0}^1 
\int_{\gamma=0}^1 F(u(s,y,\gamma),v(s,x,\alpha))\cdot \grad_y \psi(s,y)\,
\varrho_\delta(x-y)\,d\gamma\,d\alpha\,dx\,dy\,ds\Big]\Big|\notag  \\
&\le E\Big[\int_{\Pi_T}\int_{\R_y^d} \int_{\alpha=0}^1 
\int_{\gamma=0}^1 \sum_{k=1}^d\Big|F_k^{\beta_\vartheta}(u(s,y,\gamma),v(s,x,\alpha)) 
- F_k(u(s,y,\gamma),v(s,x,\alpha))\Big|\notag\\
&\hspace{4cm} \times  |\partial_{y_k} \psi(s,y)|
\,\varrho_{\delta}(x-y)\,d\gamma \,d\alpha\,dy\,dx\,ds\Big].
\end{align*}
 By \eqref{eq:mod_approx_1}, we conclude
\begin{align}
\mathcal{B}(\vartheta,\delta) & \le 
C\vartheta\Big( 1 + \sup_{0\le s\le T}E \Big[||v(s,\cdot,\cdot)||_p^p\Big] 
+ \sup_{\eps>0}\sup_{0\le s\le T}E \Big[||u_\eps(s)||_p^p\Big] \Big). \notag
\end{align} 
Now we estimate as follows: 
\begin{align*}
&\Big| E \Big[\int_{\Pi_T}\int_{\R_x^d}
\int_{\alpha=0}^1 \int_{\gamma=0}^1  F(u(s,y,\gamma),v(s,x,\alpha))\cdot \grad_y \psi(s,y)\,
\varrho_\delta(x-y) \,d\gamma\,d\alpha\,dx\,dy\,ds \Big]\\
& \quad
- E\Big[\int_{\Pi_T}\int_{\alpha=0}^1 \int_{\gamma=0}^1  
F(u(s,y,\gamma),v(s,y,\alpha))\cdot \grad_y \psi(s,y)
\,d\gamma\,d\alpha\,dy\,ds\Big]\Big|\\
& \le E \Big[\int_{s=0}^T \int_{\R_y^d}\int_{\R_x^d}\int_{\alpha=0}^1 
\int_{\gamma=0}^1  \big|F(u(s,y,\gamma),v(s,x,\alpha))-F(u(s,y,\gamma),v(s,y,\alpha))\big|
\\ & \hspace{5cm} \times
|\grad_y \psi(s,y)|\, \varrho_\delta(x-y) \,d\gamma\,d\alpha\,dx\,dy\,ds\Big]\\
&  \qquad \big(\text{by the  inequality}~ \eqref{eq:local-lip-flux} \big)\\
&\le C  E \Big[\int_{s=0}^T \int_{\R_y^d}\int_{\R_x^d}\int_{\alpha=0}^1 
\int_{\gamma=0}^1|v(s,y,\gamma)-v(s,x,\gamma) 
| \big(1+|v(s,y,\gamma)|^p+|u(s,y,\alpha)|^p\big) 
\\ & \hspace{5cm} \times
|\grad_y \psi(s,y)|
\, \varrho_\delta(x-y) \,d\alpha\,d\gamma\,dx\,dy\,ds\Big]\\
&  \qquad (\text{by Cauchy-Schwartz's inequality w.r.t.~the measure}~ 
\varrho_{\delta}(x-y)\,d\gamma\,d\alpha\,dx\,dy\,ds\,dP(\omega))\\
& \le C\Big( E\Big[ \int_{s=0}^T\int_{\R_y^d}\int_{\R_x^d} 
\int_{\gamma=0}^1|v(s,y,\gamma)-v(s,x,\gamma) |^2 
\varrho_{\delta}(x-y) \,d\gamma\,dx\,dy\,ds\Big] \Big)^{\frac 12}\\
& = C\Big( E\Big[ \int_{s=0}^T\int_{\R_z^d}\int_{\R_y^d}\int_{\gamma=0}^1
|v(s,y,\gamma)-v(s,y+\delta z,\gamma) |^2 \varrho(z) \,d\gamma\,dy\,dz\,ds\Big] \Big)^{\frac 12}
= o(\delta).
\end{align*}
Therefore,
\begin{align}
&\Big| E\Big[\int_{\Pi_T}\int_{\R_x^d}
\int_{\alpha=0}^1 \int_{\gamma=0}^1  
F^\beta(u(s,y,\gamma),v(s,x,\alpha))\cdot \grad_y \psi(s,y)\,
\varrho_\delta(x-y) \,d\gamma\,d\alpha\,dx\,dy\,ds \Big]\notag \\
& \qquad
- E\Big[\int_{\Pi_T}\int_{\alpha=0}^1 \int_{\gamma=0}^1 
F(u(s,y,\gamma),v(s,y,\alpha))\cdot \grad_y \psi(s,y)
\,d\gamma\,d\alpha\,dy\,ds\Big]\Big| \notag \\
& \quad \le \text{Const}(\psi)\, \vartheta + o(\delta) 
\goto 0\quad 
\text{as $(\vartheta,\delta)\goto (0,0)$.} \notag
\end{align} 
\end{proof}

\begin{lem}\label{stochastic_lemma_6} 
It holds that
\begin{align}
&\lim_{l\goto 0}\lim_{\delta_{0}\goto 0} J_4 = E \Big[\int_{\Pi_T} 
\int_{\R_x^d}\int_{|z|>0}\int_{\lambda =0}^1\int_{\alpha=0}^1  
(1-\lambda)\beta^{\prime\prime} \big(u_\eps(s,y)-v(s,x,\alpha)
+\lambda \eta_{\eps}(y,u_\eps(s,y);z)\big)\notag \\
&\hspace{5cm} \times|\eta_{\eps}(y,u_\eps(s,y);z)|^2\psi(s,y)
\varrho_{\delta}(x-y)\,d\alpha\,d\lambda \,m(dz)\,dx\,dy \,ds\Big],
\label{eq:J_4-delta} 
\end{align} 
and
\begin{align}
&\lim_{l\goto 0}\lim_{\delta_{0}\goto 0} I_4 
= E \Big[\int_{\Pi_T} \int_{\R_x^d}\int_{|z|>0} 
\int_{\lambda =0}^1\int_{\alpha=0}^1  (1-\lambda)
\beta^{\prime\prime} \big(v(s,x,\alpha)-u_\eps(s,y) 
+\lambda \eta(x,v(s,x,\alpha);z)\big)\notag \\
&\hspace{5cm} \times|\eta(x,v(s,x,\alpha);z)|^2
\psi(s,x)\varrho_{\delta}(x-y)
\,d\alpha \,d\lambda \,m(dz)\,dx\,dy \,ds\Big].\label{eq:I_4-delta}
\end{align}
\end{lem}

\begin{proof}
We will establish \eqref{eq:J_4-delta} in detail. 
The proof of \eqref{eq:I_4-delta} is very similar, 
and thus left to the reader. Note that 
$J_4$ can be rewritten as
\begin{align*}
J_4& = E\Big[ \int_{\Pi_T\times \Pi_T}\int_{|z|>0}
\int_{\R_k} \int_{\alpha=0}^1 \int_{\lambda =0}^1 (1-\lambda)
\beta^{\prime\prime} \big(u_\eps(s,y)-k +\lambda \eta_{\eps}(y,u_\eps(s,y);z)\big)\notag\\
 &\hspace{1.5cm}\times |\eta_{\eps}(y,u_\eps(s,y);z)|^2 \psi(s,y)\rho_{\delta_0}(t-s)\varrho_{\delta}(x-y) \varsigma_l(v(t,x,\alpha)-k)\,
d\lambda\, d\alpha \,dk\,m(dz)\,dx\,dt \,dy\,ds\Big]\\
&=E \Big[\int_{\Pi_T\times \Pi_T}\int_{|z|>0}
\int_{\R_k} \int_{\alpha=0}^1 \int_{\lambda =0}^1 (1-\lambda) 
\beta^{\prime\prime} \big(u_\eps(s,y)-v(t,x,\alpha) + k +\lambda \eta_{\eps}(y,u_\eps(s,y);z)\big)\notag\\
&\hspace{2cm}\times |\eta_{\eps}(y,u_\eps(s,y);z)|^2\psi(s,y)\rho_{\delta_0}(t-s)\varrho_{\delta}(x-y) \varsigma_l(k)
\, d\lambda\, d\alpha \,dk\,m(dz)\,dx\,dt \,dy\,ds\Big].
\end{align*} 
Furthermore,
\begin{align*}
& \Big| J_4 -E \int_{\Pi_T} \int_{\R_x^d}\int_{|z|>0} 
\int_{\R_k}\int_{\alpha=0}^1 \int_{\lambda =0}^1 
(1-\lambda)
\beta^{\prime\prime} (u_\eps(s,y)-v(s,x,\alpha)+k +\lambda \eta_{\eps}(y,u_\eps(s,y);z))\notag\\
 &\hspace{5cm}\times|\eta_{\eps}(y,u_\eps(s,y);z)|^2 \psi(s,y)\varrho_{\delta}(x-y) 
 \varsigma_l(k)\,d\lambda\, d\alpha\,dk\,m(dz)\,dx \,dy\,ds \Big|\\
 =  & \Big| E \Big[\int_{\Pi_T\times \Pi_T}\int_{|z|>0}\int_{\R_k} 
 \int_{\alpha=0}^1 \int_{\lambda =0}^1 (1-\lambda)
\Big(\beta^{\prime\prime} \big(u_\eps(s,y)-v(t,x,\alpha)
+k +\lambda \eta_{\eps}(y,u_\eps(s,y);z)\big) \notag \\
& \hspace{6cm} -\beta^{\prime\prime} \big(u_\eps(s,y)-v(s,x,\alpha)+k+\lambda \eta_{\eps}(y,u_\eps(s,y);z)\big)\Big)
\notag \\& \hspace{2cm}
\times |\eta_{\eps}(y,u_\eps(s,y);z)|^2 
\psi(s,y)\rho_{\delta_0}(t-s)\varrho_{\delta}(x-y) \varsigma_l(k)
 \,d\lambda\, d\alpha \,dk\,m(dz)\,dx\,dt \,dy\,ds\Big]\\
&\quad - E \Big[\int_{s=0}^T \int_{\R_y^d\times \R_x^d}
\int_{|z|>0}\int_{\alpha=0}^1\int_{\R_k} \int_{\lambda =0}^1
 (1-\lambda)|\eta_{\eps}(y,u_\eps(s,y);z)|^2 
 \beta^{\prime\prime} \big(u_\eps(s,y)-k +\lambda \eta_\eps(y,u_\eps(s,y);z)\big)\notag\\
 &\hspace{2cm}\times
 \psi(s,y)\Big(1-\int_{t=0}^T\rho_{\delta_0}(t-s)\,dt\Big)
 \,\varrho_{\delta}(x-y)\varsigma_l(v(s,x,\alpha)-k)
 d\lambda\,dk\, d\alpha\,m(dz)\,dx\,dy\,ds\Big] \Big|\\
 & \le ||\beta^{\prime\prime\prime}||_{\infty}\,
 E \Big[ \int_{s=0}^T \int_{t=0}^T\int_{\R_y^d\times \R_x^d}
 \int_{|z|>0}\int_{\alpha=0}^1
|\eta_{\eps}(y,u_\eps(s,y);z)|^2 |v(t,x,\alpha)-v(s,x,\alpha)|\\
&\hspace{5cm} \times \psi(s,y)\rho_{\delta_0}(t-s)
\varrho_{\delta}(x-y) \,d\alpha \,m(dz)\,dx\,dy \,dt\,ds\Big] +o(\delta_0)\\
& \le \text{Const}(\beta,\eta)\,E \Big[ \int_{s=0}^T \int_{t=0}^T
\int_{\R_y^d  \times \R_x^d} \int_{\alpha=0}^1 g^2(y)(1+|u_\eps(s,y)|^2)
|v(t,x,\alpha)-v(s,x,\alpha)|\\
&\hspace{6cm} \times \psi(s,y)\rho_{\delta_0}(t-s)
\varrho_{\delta}(x-y) d\alpha\,dx\,dy \,dt\,ds\Big] +o(\delta_0)\\
& \le \text{Const}(\beta,\eta)\,E \Big[ \int_{s=\delta_0}^T \int_{t=0}^T
\int_{\R_y^d  \times \R_x^d}\int_{\alpha=0}^1 g^2(y)(1+|u_\eps(s,y)|^2)
|v(t,x,\alpha)-v(s,x,\alpha)|\\
&\hspace{6cm} \times \psi(s,y)\rho_{\delta_0}(t-s)
\varrho_{\delta}(x-y) d\alpha\,dx\,dy \,dt\,ds\Big] \\
&\quad 
+\text{Const}(\beta,\eta)\,E  \Big[\int_{s=0}^{\delta_0} \int_{t=0}^T
\int_{\R_y^d  \times \R_x^d}\int_{\alpha=0}^1 g^2(y)(1+|u_\eps(s,y)|^2)
|v(t,x,\alpha)-v(s,x,\alpha)|\\
&\hspace{6cm} \times \psi(s,y)\rho_{\delta_0}(t-s)
\varrho_{\delta}(x-y) d\alpha\,dx\,dy \,dt\,ds\Big] +o(\delta_0) \\
&\qquad (\text{by the Cauchy-Schwartz's inequality})\\
& \le \text{Const}(\beta,\eta)\sqrt{E  \Big[\int_{s=\delta_0}^T 
\int_{t=0}^T\int_{\R_y^d\times \R_x^d}g^4(y)(1+|u_\eps(s,y)|^4)
\psi(s,y)\rho_{\delta_0}(t-s)\varrho_{\delta}(x-y)\,dx\,dy \,dt\,ds\Big]}\notag \\
&\hspace{.1cm}\times \sqrt{E \Big[ \int_{s=\delta_0}^T 
\int_{t=0}^T\int_{\R_y^d\times \R_x^d}\int_{\alpha=0}^1 
|v(t,x,\alpha)-v(s,x,\alpha)|^2
\psi(s,y)\rho_{\delta_0}(t-s)\varrho_{\delta}(x-y)\,d\alpha dxdy \,dt\,ds\Big]}+o(\delta_0)\\
& \le  \text{Const}(\beta,\eta,\psi)  \sqrt{E \Big[ \int_{s=\delta_0}^T 
\int_{t=0}^T\int_{\R_x^d}\int_{\alpha=0}^1 |v(t,x,\alpha)-v(s,x,\alpha)|^2
\rho_{\delta_0}(t-s)\,d\alpha\,dx \,dt\,ds\Big]}~~ +o(\delta_0)\\
& \le \text{Const}(\beta,\eta,\psi)  \sqrt{E \Big[ \int_{r=0}^1 \int_{t=0}^T
\int_{\R_x^d}\int_{\alpha=0}^1|v(t,x,\alpha)-v(t+r\delta_0,x,\alpha)|^2
\rho(-r)\, d\alpha\,dx \,dt\,dr\Big]}~~ +o(\delta_0).
\end{align*}
As before $\underset{\delta_0\downarrow 0 }\lim\,  
\int_{r=0}^T\int_{\rd} \int_{\alpha=0}^1 |v(t+\delta_0 r,x,\alpha)-v(t,x,\alpha)|^2\,
d\alpha\,dx\,dt= 0$ for all $r$. Therefore, by the dominated convergence 
theorem,  $E \Big[\int_{r=0}^1 \int_{t=0}^T\int_{\R^d} 
\int_{\alpha=0}^1 |v(t,x,\alpha)-v(t+r\delta_0,x,\alpha)|^2
\rho(-r)\,d\alpha\,dx \,dt\,dr\Big] \rightarrow 0$ as $\delta_0\rightarrow 0$.

Let
\begin{align*}
\mathcal{N}& :=E \Big[\int_{\Pi_T\times \R_x^d}\int_{|z|>0}
\int_{\R_k} \int_{\alpha=0}^1\int_{\lambda =0}^1 (1-\lambda)|\eta_{\eps}(y,u_\eps(s,y);z)|^2 
\beta^{\prime\prime} \big(u_\eps(s,y)-k +\lambda \eta_{\eps}(y,u_\eps(s,y);z)\big)\notag\\
 &\hspace{5cm} \times\psi(s,y)\varrho_{\delta}(x-y) 
 \varsigma_l(v(s,x,\alpha)-k)
\,d\lambda\,d\alpha \,dk\,m(dz)\,dx\,dy \,ds\Big]\\
&\qquad-E \Big[\int_{\Pi_T\times \R_x^d}\int_{|z|>0} 
\int_{\alpha=0}^1\int_{\lambda =0}^1 (1-\lambda)
\beta^{\prime\prime} \big(u_\eps(s,y)-v(s,x,\alpha) +\lambda \eta_{\eps}(y,u_\eps(s,y);z)\big)\notag\\
 &\hspace{5cm}\times |\eta_{\eps}(y,u_\eps(s,y);z)|^2  \psi(s,y)\varrho_{\delta}(x-y)
 \,d\lambda\,d\alpha \,m(dz)\,dx\,dy \,ds\Big]\\
 &=E \Big[\int_{\Pi_T\times \R_x^d}\int_{|z|>0}\int_{\R_k}
 \int_{\alpha=0}^1 \int_{\lambda =0}^1 (1-\lambda)|\eta_{\eps}(y,u_\eps(s,y);z)|^2 
\Big(\beta^{\prime\prime} \big(u_\eps(s,y)-k +\lambda \eta_{\eps}(y,u_\eps(s,y);z)\big) \\
& \hspace{6cm}
-\beta^{\prime\prime} \big(u_\eps(s,y)-v(s,x,\alpha) 
+\lambda \eta_{\eps}(y,u_\eps(s,y);z)\big)\Big)\psi(s,y)\notag\\
&\hspace{5cm}\times\varrho_{\delta}(x-y) 
\varsigma_l(v(s,x,\alpha)-k)\,d\lambda\,d\alpha \,dk\,m(dz)\,dx\,dy \,ds\Big].
\end{align*}
We estimate $\mathcal{N}$ as follows: 
\begin{align*}
|\mathcal{N}| &\le C E \Big[\int_{\Pi_T\times \R_x^d}
\int_{|z|>0}\int_{\R_k}  \int_{\alpha=0}^1 |\eta_{\eps}(y,u_\eps(s,y);z)|^2 |v(s,x,\alpha)-k|
\psi(s,y) \varrho_{\delta}(x-y) \\
&\hspace{3.5cm}\times 
\varsigma_l(v(s,x,\alpha)-k)\,d\alpha \,dk\,m(dz)\,dx\,dy \,ds\Big]\\
&\le C\,l\,E \Big[\int_{\Pi_T\times \R_x^d}\int_{|z|>0}  |\eta_{\eps}(y,u_\eps(s,y);z)|^2 
\psi(s,y) \varrho_{\delta}(x-y)\,m(dz)\,dx\,dy \,ds\Big]\\
&\le C(\psi,\eta)\, l\goto 0 \quad \text{as $l\goto 0$.}
\end{align*}
Hence, \eqref{eq:J_4-delta} is established.
\end{proof}

\begin{lem}\label{stochastic_lemma_7} 
For fixed $\delta > 0$ and $\beta$, it holds that
$$
\limsup_{(\eps,\delta_0,l)\rightarrow 0}\,  |J_7| = 0.
$$
\end{lem}

\begin{proof}
Note that  
\begin{align}
|J_7| &\le \eps ||\beta^\prime||_{\infty} 
\Big|E \Big[\int_{\Pi_T}\int_{\R_x^d}|\grad_y u_\eps(s,y)| 
|\grad_y[\psi(s,y)\varrho_\delta(x-y)|\,dx\,dy\,ds\Big]\Big|\notag \\
&\le\eps \, ||\beta^\prime||_{\infty} E\Big[\int_{|y|\le K}\int_{t=0}^T\int_{\R_x^d}
|\grad_y u_\eps(t,y)|\,|\grad_y[\psi(t,y)\varrho_\delta(x-y)]|\,dx\,dt\,dy\Big]\notag \\
& \qquad (\text{by the Cauchy-Schwartz's inequality})\notag \\
& \le C(\beta, \psi) \eps^{\frac 12}\Big(E\Big[\int_{\Pi_T} \eps |\grad_y u_\eps(t,y)|^2 
\,dy\,dt\Big]\Big)^\frac{1}{2} \Big(E \Big[\int_{|y|\le K}\int_0^T \big |\int_{\R_x^d}
\grad_y[\psi(t,y)\varrho_\delta(x-y)]\,dx\big|^2\,dt\,dy\Big]\Big)^\frac{1}{2}\notag \\
& \le  C(\beta,\psi,\delta)\,\eps^{\frac{1}{2}} \Big(\sup_{\eps>0} 
E\Big[|\eps \int_{t=0}^T \int_{\R_y^d} |\grad_y u_\eps(t,y)|^2
 \,dy\,dt|\Big]\Big)^\frac{1}{2}\notag \\
&\le  C(\beta,\psi,\delta)\,\eps^{\frac{1}{2}} 
\quad (\text{by \eqref{gradient-esti}})\notag\\
& \goto 0 \quad \text{as $\eps\goto 0$.}
\notag
\end{align}
Thus $ \underset{(\eps,\delta_0,l)\rightarrow 0}{\limsup}\,  |J_7| = 0$, which 
completes the proof.
\end{proof}
\begin{lem} \label{stochastic_lemma_8} 
Assume that $ \vartheta \goto 0 $, $ \delta\goto 0$, 
and $\vartheta^{-1}\delta^{2}\goto 0 $. Then
\begin{align*}
& \limsup_{\vartheta \goto 0,\delta\goto 0,\vartheta^{-1}
\delta^2\goto 0^+ }\limsup_{\eps \rightarrow 0}
\Big[ \lim_{l\goto 0}\lim_{\delta_0 \goto 0}
\Big((I_3 +J_3)+ (I_4 + J_4)\Big)\Big]=0.
\end{align*}
\end{lem}

\begin{proof}
We combine Lemmas \ref{stochastic_lemma_3} and 
\ref{stochastic_lemma_6} to conclude that
\begin{align}
& \lim_{l\goto 0}\lim_{\delta_0 \goto 0} 
\Big((I_3 +J_3)+ (I_4 + J_4)\Big) \notag \\
&=E \Big[\int_{\Pi_T}\int_{\R_x^d}\Big(\int_{|z|>0}
\int_{\alpha=0}^1 \Big\{\beta \big(v(t,x,\alpha)
-u_\eps(t,y)+\eta(x,v(t,x,\alpha);z)-\eta_{\eps}(y,u_\eps(t,y);z)\big)
\notag \\ & \hspace{4.5cm} 
-\beta \big(v(t,x,\alpha)-u_\eps(t,y)\big)-\big(\eta(x,v(t,x,\alpha);z)-\eta_{\eps}(y,u_\eps(t,y);z)\big) 
\notag \\ 
&\hspace{5cm} \times \beta^\prime \big(v(t,x,\alpha)-u_\eps(t,y)\big)\Big\}
\, d\alpha\,m(dz)\Big) \psi(t,y)\varrho_\delta(x-y)\,dx\,dy\,dt\Big] \notag \\
&\le E \Big[\int_{\Pi_T}\int_{\R_x^d}\Big(\int_{|z|>0}
\int_{\alpha=0}^1 \Big\{\beta \big(v(t,x,\alpha)
-u_\eps(t,y)+\eta(x,v(t,x,\alpha);z)-\eta(y,u_\eps(t,y);z)\big)
\notag \\ & \hspace{4.5cm} 
-\beta \big(v(t,x,\alpha)-u_\eps(t,y)\big)-\big(\eta(x,v(t,x,\alpha);z)-\eta(y,u_\eps(t,y);z)\big) 
\notag \\ 
&\hspace{5cm} \times \beta^\prime \big(v(t,x,\alpha)-u_\eps(t,y)\big)\Big\}
\, d\alpha\,m(dz)\Big) \psi(t,y)\varrho_\delta(x-y)\,dx\,dy\,dt\Big] \notag \\
&\qquad+ C(\beta, \psi)(\eps+ \eps^2) \notag\\
&= E \Big[\int_{\Pi_T}\int_{\R_x^d}\Big( \int_{|z|>0} 
\int_{\alpha=0}^1 \big(\beta(a+b)-\beta(a)-b\beta^{\prime}(a)\big)
\,d\alpha\,m(dz)\Big)\,\psi(t,y)\varrho_{\delta}(x-y)\,dx\,dy\,dt\Big]\notag\\
&\hspace{8cm}+ C(\beta, \psi)o(\eps)\notag\\
&=E \Big[\int_{\Pi_T}\int_{\R_x^d}\Big( \int_{|z|>0}
\int_{\alpha=0}^1\int_{\theta=0}^1 b^2
(1-\theta)\beta^{\prime\prime}(a+\theta\,b)
\,d\theta\,d\alpha\,m(dz)\Big)\,\psi(t,y)\varrho_{\delta}(x-y)\,dx\,dy\,dt\Big]\notag \\
&\hspace{8cm}+C(\beta, \psi)o(\eps)
\label{eq:endgame}
\end{align}
where $ a=v(t,x,\alpha)-u_\eps(t,y)$ 
and $b=\eta(x,v(t,x,\alpha);z)-\eta(y,u_\eps(t,y);z)$. 
Note that 
\begin{align}
b^2\beta^{\prime\prime}(a+\theta\,b)
&=\big(\eta(x,v(t,x,\alpha);z)-\eta(y,u_\eps(t,y);z)\big)^2\,
\beta^{\prime\prime}\Big(a+\theta\,\big(\eta(x,v(t,x,\alpha);z)-\eta(y,u_\eps(t,y);z)\big)\Big)\notag \\
& \leq \Big(|v(t,x,\alpha)-u_\eps(t,y)|^2 + K|x-y|^2\Big)
(1\wedge | z|^2)\,\beta^{\prime\prime}(a+\theta\,b)\notag \\ 
&= \Big( a^2 + K^2|x-y|^2\Big) 
\,\beta^{\prime\prime}(a+\theta\,b)\, (1\wedge | z|^2). 
\label{eq:nonlocal-estim}
\end{align}

We need to find a suitable upper bound on $a^2\,\beta^{\prime\prime}(a+\theta\,b)$. 
Note that $\beta^{\prime\prime}$ is nonnegative and symmetric around zero. 
Thus, we can assume without loss of generality that $a \ge 0$. 
Then, by assumption \ref{A2},
\begin{align*}
v(t,x,\alpha)-u_\eps(t,y) +\theta \,b \ge  -K|x-y|
+(1- \lambda^*) ( v(t,x,\alpha)-u_\eps(t,y))
\end{align*}
for $\theta \in [0,1]$. 
In other words    
\begin{align}
0 \le a \le (1-\lambda^*)^{-1}(a+ \theta b+ K|x-y|). 
\label{eq:nonlocal-estim-1}
\end{align}
We substitute $\beta= \beta_{\vartheta}$ in 
\eqref{eq:nonlocal-estim}, and 
use \eqref{eq:nonlocal-estim-1} to obtain
\begin{align}
b^2 \beta''_\vartheta (a+\theta\, b) 
&  \le (1-\lambda^*)^{-2}(a+\theta\,b+ K|x-y|)^2 
\, \beta''_\vartheta (a+\theta b)
\,(|z|^2\wedge 1) + \frac{ K|x-y|^2}{\vartheta}\,(|z|^2\wedge 1) \notag\\
& \le  2(1-\lambda^*)^{-2} (a+ \theta b)^2 
\beta''_\vartheta(a+\theta b)(|z|^2\wedge 1) +C(K,\lambda^*)
\frac{|x-y|^2}{\vartheta}(|z|^2\wedge 1 )\notag\\
&  \le \Big[2(1-\lambda^*)^{-2} 
C\vartheta +C(K,\lambda^*) \frac{|x-y|^2}{\vartheta}\Big]
(|z|^2\wedge 1 ),
\label{eq:endgame2}
\end{align} 
as $\sup_{r\in \R}\, r^2 \beta^{\prime\prime}_\vartheta(r) 
\le \vartheta$ by \eqref{eq:approx to abosx}.

Now combine \eqref{eq:endgame}-\eqref{eq:endgame2} to write 
the following inequality:    
\begin{align*}
&E \Big[\int_{(0,T]}\int_{x,y}\int_{|z|>0} 
\int_{\alpha=0}^1 \Big( \beta_\vartheta(a+b)-\beta_\vartheta(a) -b\beta_\vartheta'(a)\Big)
\psi(t,y)\,\varrho_{\delta}(x-y)\,d\alpha \,m(dz)\,dx\,dy\,dt\Big]\notag \\
& \qquad \leq C_1\Big(\vartheta + \vartheta^{-1}\delta^2\Big)T,\notag
\end{align*}
where the constant $C_1$ depends 
only on $\psi$ and is in particular independent of $\eps$. 
We now let $\vartheta\goto 0$, $ \delta\goto 0$ 
and $\vartheta^{-1}\delta^{2}\goto0$, yielding
\begin{align*}
& \limsup_{\vartheta \goto 0,\delta\goto 0,\vartheta ^{-1}\delta^2\goto 0}
\limsup_{\eps \rightarrow 0}
\Big[ \lim_{l\goto 0}\lim_{\delta_0 \goto 0}
\Big((I_3 +J_3)+ (I_4 + J_4)\Big)\Big]\le0.
\end{align*} 
This wraps up the proof, once we observe in \eqref{eq:endgame} that
$\underset{l\goto 0} \lim\,\underset{\delta_0 \goto 0}
\lim\,\Big((I_3 +J_3)+ (I_4 + J_4)\Big)$ 
is nonnegative.  
\end{proof}
All of the above results can be combined into the following proposition.
\begin{prop} \label{kato_inequality}
Let $v(t,x,\alpha)$ be a given generalized entropy 
solution of \eqref{eq:levy_stochconservation_laws} with initial data
$v(0,x)$ and $u(t,x,\gamma)$ be the generalized entropy 
solution with initial data $u(0,x)$, which has been extracted 
out of a Young measure valued subsequential limit of the 
sequence $\{u_\eps(t,x)\}_{\eps>0}$ of viscous approximations. 
Then, for any nonnegative $H^1( [0,\infty) \times \R^d)$ 
function $\psi(t,x)$ with compact support, it holds that
\begin{align}
   0\le & E \Big[\int_{\R_x^d} | v(0,x)-u(0,x)|\psi(0,x)\,dx\Big] 
   + E \Big[\int_{\Pi_T}  \int_{\gamma=0}^1 
   \int_{\alpha=0}^1 |v(t,x,\alpha)-u(t,x,\gamma)|
   \partial_t \psi(t,x)\,d\alpha\,d\gamma \,dx\,dt\Big] \notag \\
   & \hspace{ 3cm}+ E \Big[\int_{\Pi_T}\int_{\gamma=0}^1 
   \int_{\alpha=0}^1 F\Big(v(t,x,\alpha),u(t,x,\gamma)\Big)
  \cdot \grad_x \psi(t,x) \,d\alpha\,d\gamma\,dx\,dt \Big]
\label{eq:kato} .
\end{align}
\end{prop}

\begin{proof}
We add \eqref{stochas_entropy_1} and \eqref{stochas_entropy_3}
and then take the limits 
$$
\underset{\eps_n \downarrow 0} \lim \, \underset{l\rightarrow 0}
\lim\, \underset{\delta_0\downarrow 0} \lim,
$$ 
invoking Lemmas \ref{stochastic_lemma_1}, \ref{stochastic_lemma_2} , 
\ref{stochastic_lemma_3},   \ref{stochastic_lemma_4},  
\ref{stochastic_lemma_5}, \ref{stochastic_lemma_6}, \ref{stochastic_lemma_7}, 
and \ref{stochastic_lemma_8}.  In the resulting 
expression, we take $\delta = \vartheta^{\frac 23}$ 
and then send $\vartheta \goto 0+$ with the 
second parts of Lemmas \ref{stochastic_lemma_1}, 
\ref{stochastic_lemma_2}, \ref{stochastic_lemma_4},  
\ref{stochastic_lemma_5}, and \ref{stochastic_lemma_8} 
in mind, thereby arriving at
\begin{align}
0\le & E \Big[\int_{\R_x^d} | v(0,x)-u(0,x)|\psi(0,x)\,dx\Big] 
+ E \Big[\int_{\Pi_T}\int_{\gamma=0}^1 
\int_{\alpha=0}^1 |v(t,x,\alpha)-u(t,x,\gamma)|
\partial_t \psi(t,x)\,d\alpha\,d\gamma \,dx\,dt\Big] \notag \\
& \hspace{ 2cm} + E \Big[\int_{\Pi_T}\int_{\gamma=0}^1 
\int_{\alpha=0}^1 F\big(v(t,x,\alpha),u(t,x,\gamma)\big)\cdot
\grad_x \psi(t,x)\,d\alpha\,d\gamma\,dx\,dt\Big] ,\notag
\end{align} 
which holds for any nonnegative $\psi \in C_c^2\big([0,\infty)\times \R^d\big)$.
It now follows by a routine approximation 
argument that \eqref{eq:kato} holds for any $\psi$ with compact support such 
that $\psi\in H^1( [0,\infty) \times \R^d)$.
\end{proof}

\begin{proof}[Proof of Theorem \ref{thm:uniqueness}]
Let $v(t,x,\alpha)$ be a generalized entropy 
solution of \eqref{eq:levy_stochconservation_laws} with initial data
$v(0,x)$ and $u(t,x,\gamma)$ be the solution that has 
been obtained as the Young measure valued limit of 
the sequence $\{u_\eps(t,x)\}_{\eps>0}$, where 
$u_\eps$ solves \eqref{eq:levy_stochconservation_laws-viscous-new} 
with initial data $u_0^\eps(\cdot)$.  
Now from Proposition \ref{kato_inequality}, for
any nonnegative $\psi(t,x)\in H^1( [0,\infty) \times \R^d)$ 
with compact support, it holds that
\begin{align}
0\le & E \Big[\int_{\R_x^d} | v(0,x)-u(0,x)|\psi(0,x)\,dx\Big] 
+ E \Big[\int_{\Pi_T}\int_{\gamma=0}^1 
\int_{\alpha=0}^1 |v(t,x,\alpha)-u(t,x,\gamma)|
\partial_t \psi(t,x)\,d\alpha\,d\gamma \,dx\,dt\Big] \notag \\
& \hspace{ 2cm} 
+ E \Big[\int_{\Pi_T}\int_{\gamma=0}^1 
\int_{\alpha=0}^1 F\big(v(t,x,\alpha),u(t,x,\gamma)\big)\cdot
\grad_x \psi(t,x)\,d\alpha\,d\gamma \,dx\,dt\Big].
\label{stochastic_estimate_5}
\end{align}

For each $ n\in \mathbb{N} $, define
\begin{align*}
\phi_n(x)=
\begin{cases} 
1,\quad \text{if} ~ |x|\le n\\
2(1-\frac{|x|}{2n}), \quad \text{if} ~n < |x|\le 2n\\
0,\quad \text{if}~|x|> 2n.
\end{cases}
\end{align*}
For each $h>0 $ and fixed $t\geq 0$, define 
\begin{align*}
\psi_h^t(s)=
\begin{cases} 
1,\quad \text{if} ~ s\le t\\
1-\frac{s-t}{h}, \quad \text{if} ~t \le s\le t+h\\
0,\quad \text{if}~s> t+h.
\end{cases}
\end{align*}
Clearly \eqref{stochastic_estimate_5} holds 
with $\psi(s,x)=\phi_n(x)\psi_h^t(s)$. 

Let $\mathbb{T}$ be the set all points $t$ 
in $[0, \infty)$ such that $t$ is a right Lebesgue point of 
$$
A_n(s)= E\Big[\int_{\R_x^d}\int_{\gamma=0}^1 
\int_{\alpha=0}^1\phi_n(x) 
|v(s,x,\alpha)-u(s,x,\gamma)|\,d\alpha\,d\gamma\,dx\Big],
$$ 
for all $n$. Clearly, $\mathbb{T}^C$ has zero 
Lebesgue measure. Fix  $t\in \mathbb{T}$.
Thus, from \eqref{stochastic_estimate_5} we have
\begin{align}
&\frac{1}{h}\int_{t}^{t+h} E\Big[\int_{\R_x^d} 
\int_{\gamma=0}^1 \int_{\alpha=0}^1  |v(s,x,\alpha)-u(s,x,\gamma)|
\phi_n(x)\,d\alpha\,d\gamma \,dx\Big]\,ds \notag \\
& \qquad 
\le  E\Big[\int_{\Pi_{T}}\int_{\gamma=0}^1 
\int_{\alpha=0}^1  F\big(v(s,x,\alpha),u(s,x,\gamma)\big)
\cdot\grad_x \phi_n(x)\,\psi_h^t(s)
\,d\alpha\,d\gamma\,dx\,ds\Big] \notag \\
& \hspace{2cm}
+ E\Big[\int_{\R_x^d} |v(0,x)-u(0,x)|\phi_n(x)\,dx\Big].\notag 
\end{align}
Taking limit as $h\goto 0$, we obtain
\begin{align}
& E\Big[\int_{\R_x^d}\int_{\gamma=0}^1 
\int_{\alpha=0}^1|v(t,x,\alpha)-u(t,x,\gamma)|
\phi_n(x)\,d\alpha\,d\gamma\,dx\Big] \notag \\
& \qquad 
\le  E\Big[\int_{\R^d}\int_{s=0}^t \int_{\gamma=0}^1 
\int_{\alpha=0}^1  F\big(v(s,x,\alpha),u(s,x,\gamma)\big)\cdot\grad_x \phi_n(x)
\,d\alpha\,d\gamma\,ds\,dx\Big] \notag \\
& \hspace{2cm}
+ E\Big[\int_{\R_x^d} |v(0,x)-u(0,x)|\phi_n(x)\,dx\Big]\notag \\
& \qquad 
\le C(T) \frac{1}{n}\Big(1+ \sup_{0\le s \le T}E\Big[||v(s,\cdot,\cdot)||_p^p\Big] 
+  \sup_{0\le s \le T} E \Big[||u(s,\cdot,\cdot)||_p^p \Big] \Big)\notag \\
& \hspace{2cm}
+ E\Big[\int_{\R_x^d} |v(0,x)-u(0,x)|\phi_n(x)\,dx\Big].
\label{stochastic_estimate_final}
\end{align}
Letting $n\goto \infty$, we obtain from \eqref{stochastic_estimate_final} 
with $v(0,x)=u(0,x)$,
\begin{align}\notag
E \Big[\int_{\R^d} \int_{\gamma=0}^1 \int_{\alpha=0}^1 
|v(t,x,\alpha)-u(t,x,\gamma)|\,d\alpha\,d\gamma\,dx\Big]=0.
\end{align}
From this the claims of the theorem follow 
in a standard way (see \cite{Eymard1995,panov}).
 \end{proof}


\begin{thebibliography}{99}
 \bibitem{Balder}
E.~J.~Balder.
\newblock Lectures on Young measure theory and its applications in economics.
\newblock{\em Rend. Istit. Mat.Univ. Trieste}, 31 Suppl. 1:1-69, 2000.

\bibitem{BaVaWit}
C. Bauzet, G. Vallet and P. Wittbold.
\newblock The Cauchy problem for a conservation law with 
a multiplicative stochastic perturbation.
\newblock{\em Journal of Hyperbolic Differential Equations}. 9 (2012), no. 4, 661-709.

\bibitem{BisMaj}
I.~H.~Biswas and A.~K.~Majee.
\newblock Stochastic conservation laws: weak-in-time formulation 
and strong entropy condition. 
\newblock{\em submitted},  2013.
 
 \bibitem{Chen:2012fk}
G.-Q. Chen, Q.~Ding, and K.~H. Karlsen.
\newblock On nonlinear stochastic balance laws.
\newblock {\em Arch. Ration. Mech. Anal.}, 204(3):707--743, 2012.
 
 \bibitem{dafermos}
C.~M.~ Dafermos.
\newblock {\em Hyperbolic conservation laws in continuum physics}, volume 325
  of {\em Grundlehren der Mathematischen Wissenschaften [Fundamental Principles
  of Mathematical Sciences]}.
\newblock Springer-Verlag, Berlin, 2000.
 
\bibitem{Vovelle2010}
A.~ Debussche and J. Vovelle. 
\newblock Scalar conservation laws with stochastic forcing. 
\newblock{\em J. Funct. Analysis}, 259 (2010), 1014-1042. 

\bibitem{xu}
Z.~Dong and T.~G. Xu.
\newblock One-dimensional stochastic {B}urgers equation driven by {L}{\'e}vy
  processes.
\newblock {\em J. Funct. Anal.}, 243(2):631--678, 2007.

\bibitem{E:2000lq}
W.~E, K.~Khanin, A.~Mazel, and Y.~Sinai.
\newblock Invariant measures for {B}urgers equation with stochastic forcing.
\newblock {\em Ann. of Math. (2)}, 151(3):877--960, 2000.

\bibitem{evansweak}
L.~C.~ Evans.
\newblock {\em Weak convergence methods for nonlinear partial differential
  equations}, volume~74 of {\em CBMS Regional Conference Series in
  Mathematics}.
\newblock Published for the Conference Board of the Mathematical Sciences,
  Washington, DC, 1990.
 
 \bibitem{Eymard1995}
 R. ~Eymard, T.~Gallou't and R.~ Herbin.
 \newblock Existence and uniqueness of the entropy solution to a nonlinear hyperbolic equation.
 \newblock{\em Chinese Ann. Math., Ser. B. } 16(1) (1995) 1Ð14.
 
 \bibitem{nualart:2008}
J.~ Feng and D.~ Nualart.
\newblock Stochastic scalar conservation laws.
\newblock {\em J. Funct. Anal.}, 255(2):313--373, 2008.

\bibitem{godu}
E.~ Godlewski and P.~ Raviart.
\newblock {\em Hyperbolic systems of conservation laws}, volume 3/4 of {\em
  Math{\'e}matiques \& Applications (Paris) [Mathematics and Applications]}.
\newblock Ellipses, Paris, 1991.

\bibitem{Hausenblas2008}
E.~Hausenblas and J.~ Seidler.
\newblock Stochastic convolutions driven by martingales: maximal inequalities and exponential integrability. 
\newblock{\em Stoch. Anal. Appl. } 26 (2008), no. 1, 98-119.

\bibitem{risebroholden1997}
H.~ Holden and N.~H.~ Risebro.
\newblock Conservation laws with random source.
\newblock{\em Appl. Math. Optim}, 36(1997), 229-241.

\bibitem{Kim2003}
J.~U.~ Kim.
\newblock On a stochastic scalar conservation law.
\newblock{Indiana Univ. Math. J.}  52 (1) (2003) 227-256. 

\bibitem{Lions2013}
 P.~ L. ~ Lions, B.~ Perthame, and P.~ E.~ Souganidis.
 \newblock Scalar conservation laws with rough (stochastic) fluxes.
 \newblock{\em Stochastic Partial Differential Equations: Analysis and Computations }, 1(4), 2013, pp 664-686
 
 \bibitem{Lions2014}
 P.~ L. ~ Lions, B.~ Perthame, and P.~ E.~ Souganidis.
 \newblock Scalar conservation laws with rough (stochastic) fluxes; the spatially dependent case.
 \newblock{\em Preprint },  2014.

\bibitem{marinelli2010}
C.~Marinelli, C.~ Pr\'{e}vot, and M.~ R\"{o}ckner.
\newblock Regular dependence on initial datra for stochastic 
evolution equations with multiplicative Poisson noise.
\newblock {\em Journal of Functional Analysis} 258(2010), 616-649.

\bibitem{malek}
F.~ Otto.
\newblock {\em Weak and measure-valued solutions to evolutionary {PDE}s},
volume~13 of {\em Applied Mathematics and Mathematical Computation}.
\newblock J.~M{\'a}lek, J.~Ne\v{c}as, M.~Rokyta, 
and M.~R{{\o}circ{u}}\v{z}i\v{c}ka eds Chapman \& Hall, London, 1996.

\bibitem{panov}
E.~Y.~Panov.
\newblock On measure-valued solutions of the Cauchy problem for 
a first-order quasilinear equations.
\newblock{\em Investig. Mathematics}, 60(2):335-377, 1996.

\bibitem{peszat}
S.~Peszat and J.~Zabczyk.
\newblock {\em Stochastic partial differential equations 
with {L}{\'e}vy noise}, volume 113 of {\em Encyclopedia of Mathematics and its Applications}.
\newblock Cambridge University Press, Cambridge, 2007.

\bibitem{Protter1990}
P.~ Protter.
\newblock Stochastic Integration and Differential Equations.
\newblock Springer-Verlag, Berlin, 1990.

\bibitem{vallet2000}
G. Vallet
\newblock Dirichlet problem for a nonlinear conservation laws.
\newblock {\em Rend. Istit. Mat. Univ. Trieste }26(1Ð2) (1994) 349Ð394.

\end{thebibliography}
\end{document}